\documentclass[reqno]{amsart}

\usepackage[T1]{fontenc}
\usepackage[utf8]{inputenc}
\usepackage[english]{babel}

\usepackage{mathtools}
\usepackage{amssymb}
\usepackage{amsthm}
\usepackage{physics}
\usepackage{cancel}

\usepackage{graphicx}
\usepackage[a4paper,left=2.5cm,right=2.5cm,top=2.5cm,bottom=2.5cm]{geometry} 
\usepackage{float}
\usepackage{xcolor}
\usepackage{cancel}

\usepackage[square,numbers]{natbib}
\usepackage{hyperref}
\hypersetup{
  colorlinks=true,
  linkcolor=black,
  citecolor=black,
  urlcolor=blue
}

\numberwithin{equation}{section}

\theoremstyle{plain}
\newtheorem{theorem}{Theorem}[section]          
\newtheorem{proposition}[theorem]{Proposition}  

\newtheorem{innercor}{Corollary}[theorem]
\newenvironment{corollary}
  {%
   \innercor}
          {\endinnercor}

\newtheorem{innerlem}{Lemma}[theorem]
\newenvironment{lemma}
  {%
   \innerlem}
  {\endinnerlem}

\theoremstyle{definition}
\newtheorem{definition}[theorem]{Definition}
\newtheorem{hypothesis}[theorem]{Hypothesis}

\theoremstyle{remark}
\newtheorem{remark}[theorem]{Remark}

\newcommand{\N}{\mathbb{N}}
\newcommand{\R}{\mathbb{R}}

\newcommand{\F}{\mathcal{F}}
\newcommand{\diff}{\,\mathrm{d}}
\DeclareMathOperator*{\argmin}{arg\,min}
\newcommand{\Sch}{\mathrm{Sch}}
\newcommand{\Leb}{\mathrm{Leb}}

\title{A convergence rate for the entropic JKO scheme}
\author{Aymeric BARADAT}
\author{Sofiane CHERF}
\address{Universite Claude Bernard Lyon 1, CNRS, Centrale Lyon, INSA Lyon, Universit\'e Jean Monnet, ICJ UMR5208, 43 bd du 11 Novembre 1918, 69622
Villeurbanne, France}
\email{$\{$baradat,cherf$\}$@math.univ-lyon1.fr}

\begin{document}

\begin{abstract}
The so-called JKO scheme, named after Jordan, Kinderlehrer and Otto~\cite{jordan1998variational}, provides a variational way to construct discrete time approximations of certain partial differential equations (PDEs) appearing as gradient flows in the space of probability measures equipped with the Wasserstein metric. The method consists of an implicit Euler scheme, which can be implemented numerically.

Yet, in practice, evaluating the Wasserstein distance can be numerically expensive. To address this problem, a common strategy introduced in~\cite{Peyré2015} and which has been shown to produce faster computations, is to replace the Wasserstein distance with its entropic regularization, also known as the Schr\"odinger cost. In~\cite{baradat2025usingsinkhornjkoscheme}, the first author, Hraivoronska and Santambrogio, proved that if the regularization parameter $\varepsilon$ is proportional to the time step $\tau$, that is, $\varepsilon = \alpha \tau$ for some $\alpha > 0$, then as $\tau \to 0$, this change results in adding to the limiting PDE the additional linear diffusion term $\frac{\alpha}{2} \Delta \rho$. Our goal in this article is to provide a convergence rate under convexity assumptions between the entropic JKO scheme and the solution of the initial PDE as both $\alpha$ and $\tau$ tend to zero. This will appear as a consequence of a new bound between the classical and entropic JKO schemes.
\end{abstract}
\maketitle

\tableofcontents

\section{Introduction}
\subsection{Definition of JKO and Entropic JKO} Consider a functional~$\mathcal F : \mathcal{P}_2(\mathbb{R}^d) \to \mathbb{R} \cup \{ + \infty \}$, where~$\mathcal P_2(\R^d)$ is the set of probability measures with finite second moment. To fix the ideas, in this subsection, think of a functional of type
\begin{equation}
\label{eq:explicit_F}
\mathcal F : \rho \longmapsto \int_{\mathbb{R}^d} V(x)\, \rho(x)\diff x 
+ \frac{1}{2} \int_{\mathbb{R}^d} (W * \rho)(x)\, \rho(x)\diff x 
+ \int_{\mathbb{R}^d} f(\rho(x))\diff x,
\end{equation}
where $V : \R^d \to \R_+$, $W: \R^d \to \R_+$ and $f: \R_+ \to \R_+$ are given smooth nonnegative functions, and $\mathcal F$ is set to $+ \infty$ if $\rho$ is not absolutely continuous with respect to the Lebesgue measure. It is well known (see~\cite{jordan1998variational,ambrosio2005gradient,santambrogio2015optimal}) that the formal gradient flow of $\mathcal F$ in the Wasserstein space is the PDE
\begin{equation}
    \label{eq:gradient_flow}
         \begin{cases}
             \partial_t\rho -\operatorname{div}\left(\rho\nabla\left(\frac{\delta \F}{\delta\rho}(\rho)\right)\right)=0,\\
             \rho(0,\cdot)=\mu,
         \end{cases}
\end{equation}
where the function $\frac{\delta\F}{\delta\rho}(\rho)$ is the so-called first variation of $\F$, which equals
\begin{equation*}
    \frac{\delta \mathcal F}{\delta \rho} (\rho)= V + W \ast \rho + f'(\rho)
\end{equation*}
in the above case. This explains why the seminal work~\cite{jordan1998variational} suggested to construct solutions of equation~\eqref{eq:gradient_flow} using an implicit Euler scheme: this is the famous JKO scheme. 

In practice, the Wasserstein distance can be costly to evaluate numerically. For this reason and in a lot of applications, authors prefer to replace it by its \emph{entropic} counterpart studied in~\cite{leonard2014Survey}. Indeed, on the one hand, this entropic regularization converges towards the Wasserstein distance~\cite{leonard2012schrodinger,conforti2021formula,malamut2025convergence}. On the other hand, this change allows to use the very efficient Sinkhorn algorithm~\cite{Sinkhorn1967,Cuturi20213,benamou2015iterative}. We refer for instance to~\cite{carlier2021entropic} for an application of this idea to the computation of Wasserstein barycenters, to~\cite{benamou2019generalized} in the context of incompressible flows. In the context of the JKO scheme, this change was proposed by Peyré in~\cite{Peyré2015}.

In this work, we want to compare the classical JKO scheme and this perturbed scheme. To begin, let us introduce them. Given a measure $\mu \in \mathcal P_2(\R^d)$, a time-step parameter $\tau>0$, a regularization parameter $\alpha>0$, and an energy functional $\mathcal{F}$ as before, we define one step of the JKO and entropic JKO schemes as follows, when the formulas make sense:
\begin{align*}
    J_\tau^0(\mu) &:= \argmin_{\rho \in \mathcal{P}_2(\mathbb{R}^d)} \left\{ \frac{W_2^2(\mu, \rho)}{2\tau}  + \mathcal{F}(\rho) \right\}, \tag{JKO}\label{JKO} \\
    J_\tau^\alpha(\mu) &:= \argmin_{\rho \in \mathcal{P}_2(\mathbb{R}^d)} \left\{ \frac{\Sch^{\alpha\tau}(\mu, \rho)}{\tau}  + \mathcal{F}(\rho) \right\}, \tag{Ent JKO}\label{Ent JKO}
\end{align*}
where the Wasserstein distance $W_2$ and its entropic regularization of regularization level $\alpha \tau$ \emph{a.k.a.}\ the Schrodinger cost $\Sch^{\alpha\tau}$ are defined below in Section~\ref{notations}. The reason why the level of regularization is taken proportional to the time-step parameter $\tau$ will be clear in a few lines.
    When minimizers exist but are not unique, $J^0_\tau$ and $J^\alpha_\tau$ have to be understood as any choice among the minimizers.

We then define the iterates of the scheme as follows:
\begin{itemize}
    \item for all $k \in \N^*$,
    \begin{equation*}
        J_{k,\tau}^0(\mu) := (J_\tau^0)^{\circ k}(\mu)=\underbrace{J_\tau^0\circ \dots\circ J_\tau^0}_{k \text{ times}}(\mu),
    \end{equation*}
    is the measure obtained after $k$ steps of the classical JKO scheme with time step $\tau$;
    \item for all $k \in \N^*$, 
    \begin{equation*}
    J_{k,\tau}^\alpha(\mu) := (J_\tau^\alpha)^{\circ k}(\mu)=\underbrace{J_\tau^\alpha\circ \dots\circ J_\tau^\alpha}_{k \text{ times}}(\mu),
    \end{equation*}
    is the measure obtained after $k$ steps of the entropic JKO scheme with time step $\tau$ and regularization parameter $\alpha$.
\end{itemize}

Since \cite{jordan1998variational, Otto1999,ambrosio2005gradient}, it is known that under convexity or coercivity assumptions on the functional $\mathcal F$, the JKO scheme converges to solutions of the limiting equation~\eqref{eq:gradient_flow} in the sense that $J_{n, t/n}^0(\mu)$ converges towards a distributional solution of~\eqref{eq:gradient_flow} at time $t$ as $n \to \infty$. 

Concerning the entropic JKO scheme, the first convergence result is due to Carlier, Duval, Peyré, and Schmitzer in \cite{Carlier2017}. They studied the case when the regularization parameter $\alpha = \alpha(\tau)$ is itself a function of $\tau$, approaching zero with a rate such that $\alpha |\ln \alpha|= \mathcal O(\tau)$ (or equivalently, $\alpha = \mathcal O(\tau / |\ln(\tau)|)$). In this case, they show as before that $J_{n, t/n}^\alpha(\mu)$ converges towards a solution at time $t$ of~\eqref{eq:gradient_flow}. 

Building on asymptotics obtained in~\cite{adams2011large,duong2013wasserstein,erbar2015from}, the first author, Hraivoronska and Santambrogio, made an improvement in \cite{baradat2025usingsinkhornjkoscheme}, where they studied the case where $\alpha(\tau)$ is of order one. They show that provided $\alpha(\tau) \to \alpha_\infty$ as $\tau \to 0$, then $J_{n, t/n}^\alpha(\mu)$ converges towards a solution at time $t$ of~
\begin{equation}
    \label{eq:gradient_flow_reg}
         \begin{cases}
             \partial_t\rho -\operatorname{div}\left(\rho\nabla\left(\frac{\delta \F}{\delta\rho}(\rho)\right)\right)=\displaystyle{\frac{\alpha_\infty}{2}} \Delta \rho, \\
             \rho(0,\cdot)=\mu,
         \end{cases}
\end{equation}
 instead of~\eqref{eq:gradient_flow}, that is, with an additional term $\frac{\alpha_\infty}{2}\Delta\rho$ on the right hand side of the limiting PDE. In particular, the case $\alpha_\infty=0$, extends the result of~\cite{Carlier2017}. 

However, neither~\cite{Carlier2017} nor~\cite{baradat2025usingsinkhornjkoscheme} provide explicit bounds between the schemes or between the entropic scheme and the corresponding solutions of equation~\eqref{eq:gradient_flow_reg}. These are the questions that we want to address in the present paper. Our main contribution, stated at Theorem~\ref{the theorem}, is an explicit bound in $\alpha$ and $\tau$ on the Wasserstein distance between $J^0_{n,\tau}(\mu)$ and $J^\alpha_{n,\tau}(\mu)$. This result seems to us to be particularly interesting since combined with the known convergence rate of the JKO scheme, we easily deduce an explicit bound in $\alpha$ and $\tau$ on the Wasserstein distance between $J^0_{n,\tau}(\mu)$ and the solution of~\eqref{eq:gradient_flow} at time $t$, see Corollary~\ref{cv rate entropic JKO}. 

Another part of our work consisted in studying the optimality of our bound in $\alpha$ as $\tau \to 0$. To that aim, we compare our discrete result with a bound obtained in the continuous case, and show that they only differ from a factor $2$, but keep the orders of magnitude. Then, by extensively studying an example where everything can be computed, we identify precisely where the optimality of both the continuous and discrete bounds is lost.

Since the pioneering work~\cite{ambrosio2005gradient}, it is well understood that the stability of equation~\eqref{eq:gradient_flow} in the Wasserstein distance, and hence the question of convergence of the JKO scheme, is deeply linked to the so-called geodesic convexity of the functional $\mathcal F$ -- a property discovered by McCann in~\cite{mccann1997convexity} -- or more precisely, to convexity along generalized geodesics, see Definition~\ref{general cnx geo def}. Naturally, we have to work with this assumption as well. Unfortunately, no similar property has been discovered so far in the entropic setting, which explains why the convergence established in~\cite{baradat2025usingsinkhornjkoscheme} does not come with a rate. Therefore, we had to bypass this difficulty by exploiting the stability of the classical JKO scheme only, and not of its entropic counterpart.

\subsection{Main Results}

The following set of assumptions on $\mathcal F$ will play a crucial role in our study.
\begin{hypothesis}
    \label{hypothesis}
We assume that $\mathcal F$ satisfies:
\begin{enumerate}
    \item $\F$ is lower semicontinous (l.s.c.) with respect to the weak convergence in $\mathcal P_2(\R^d)$, where we say that $(\rho_n)_n$ converges weakly in $\mathcal P_2(\R^d)$ if it weak-$\star$ converges in duality with continuous bounded functions (later, we will say \emph{converges narrowly}), and has uniformly bounded second moments.
    \item $\mathcal{F}$ is $\lambda$-convex along generalized geodesics (see Definition~\ref{general cnx geo def}).
    \item There exists $K\in\R$ such that for all $\mu \in \mathcal P_2(\R^d)$ and $t \geq 0$, $\F(\mu*\sigma_{t}) \leq \F(\mu) + K\frac{t}{2}$, where $(\sigma_t)$ is the heat kernel, that is, the fundamental solution of $\partial_t\sigma_t=\frac{1}{2}\Delta\sigma_t$.    
\end{enumerate}
\end{hypothesis}
We will justify this set of assumptions and give examples of functionals satisfying them in Subsection~\ref{comments hyp}. For the moment, let us just notice that the points (1) and (2) allow Ambrosio, Gigli and Savaré~\cite{ambrosio2005gradient} to find, for all $t \geq 0$ and $\mu_0 \in \mathcal P_2(\R^d)$ such that $\mathcal F(\mu_0) < + \infty$, a limit
\begin{equation}
    \label{eq:def_rho0}
    \rho^0(t) := \lim_{n \to + \infty} J^0_{n,t/n}(\mu_0),
\end{equation}
with explicit convergence rates. We will recall them in Section~\ref{sec:proof_cor}, but let us already mention that the rate corresponding to the case when $\lambda = 0$ and $\mathcal F$ is below bounded writes:
\begin{equation}\label{eq:bound_AGS_simple}
W_2^2\!\left(\rho^0(t),\,J_{n,t/n}^0(\mu_0)\right)
\le
\frac{t}{n}\Bigl(\mathcal F(\mu_0)-\mathcal F(J^0_{t/n}(\mu_0))\Bigr)
\le
\frac{t}{n}\Bigl(\mathcal F(\mu_0)-\inf_{\mathcal P_2(\mathbb R^d)}\mathcal F\Bigr).
\end{equation}

Therefore, these assumptions are sufficient to give a meaning to the notion of gradient flow of $\mathcal F$ in $\mathcal P_2(\R^d)$, even in cases when the PDE~\eqref{eq:gradient_flow} cannot be written. Of course, in most practical cases such as~\eqref{eq:explicit_F}, equation~\eqref{eq:gradient_flow} can be written and $\rho^0$ is indeed its unique distributional solution.

With these assumptions at hand, we can provide our main bound in Theorem~\ref{the theorem} below. This bound involves the Boltzmann entropy, defined as follows.
\begin{definition}
For all $\rho\in\mathcal P_2(\R^d)$, the Boltzmann entropy is defined as
\[
H(\rho) := \left\{\begin{aligned}
    &\int \rho(x) \ln \rho(x) dx, &&\mbox{if }\rho \mbox{ has a density w.r.t.\ the Lebesgue measure},\\
    &+\infty, &&\mbox{else.}
\end{aligned}
\right.
\]
(Here, we used the same notation for a measure and its density with respect to the Lebesgue measure.) This quantity is always well defined in~$\R\cup\{+\infty\}$ in virtue of Proposition~\ref{below bound ent}.

\end{definition}

Our main result is:
\begin{theorem}[Convergence estimate]\label{the theorem}
Let $\mathcal{F}$ satisfy Hypothesis~\ref{hypothesis} with given $K$ and $\lambda$. Let $\tau,\alpha$ be some positive parameter, and $\mu_0 \in \mathcal P_2(\R^d)$ be an initial condition satisfying $\mathcal F(\mu_0)< + \infty$ and $H(\mu_0) < +\infty$. Finally, let us fix $n\geq0$. Then the iterates $J_{n,\tau}^0(\mu_0)$ and $J_{n,\tau}^\alpha(\mu_0)$ exist and satisfy:
\begin{itemize}
    \item If $\lambda=0$, then
    \begin{equation*}
W_2(J_{n,\tau}^0(\mu_0), J_{n,\tau}^\alpha(\mu_0))\leq \sqrt{2\tau(\F(\mu_0)-\F(J_{n,\tau}^0(\mu_0))}+\sqrt{n\tau\alpha\left(H(\mu_0)-H(J_{n,\tau}^\alpha(\mu_0))+K n\tau\right)}.
\end{equation*}
\item If $\lambda\neq0$ and $\tau \leq  \frac{1}{2\lambda_-} $ (where here and in the whole text, we use the convention $1/0 = + \infty$ so that there is no condition on $\tau$ when $\lambda \geq 0$), then 
\begin{multline*}
  W_2(J_{n,\tau}^0(\mu_0), J_{n,\tau}^\alpha(\mu_0))\leq \sqrt2(1+4\lambda_-\tau)^{\frac32} (1-\lambda_-\tau)^{-n}\sqrt{\tau}\sqrt{(\F(\mu_0)-\F(J_{n,\tau}^0(\mu_0))}\\
+(1+3\lambda_-\tau)\sqrt{\frac{1-(1+\lambda\tau)^{-2n}}{2\lambda}}\sqrt{\alpha\left(H(\mu_0)-H(J_{n,\tau}^\alpha(\mu_0))+K n\tau\right)}.
\end{multline*}
\end{itemize}
\end{theorem}
To us, the main interest of this result is that, combined with bounds on the convergence of the JKO scheme such as~\eqref{eq:bound_AGS_simple}, it implies as a corollary a bound between the iterates of the entropic JKO scheme and the corresponding gradient flow $\rho^0$ (and hence, in practical cases, the unique weak solution of~\eqref{eq:gradient_flow}). 
\begin{corollary}
    \label{cv rate entropic JKO}
In the context of Theorem~\ref{the theorem}, given $t \geq 0$ and $\rho^0(t)$ defined as in~\eqref{eq:def_rho0}, for all $\alpha_0>0$, there exists a constant $C$, depending only on the second moment of $\mu_0$, $t$, $\mathcal{F}(\mu_0)$, $H(\mu_0)$, $\lambda$, $K$ and $\alpha_0$, such that for all $n \in \N^*$ and $\alpha \leq \alpha_0$, 
\[
W_2(J_{n,t/n}^\alpha(\mu_0),\rho^0(t))\leq C\left(\sqrt\alpha + \sqrt\frac{1}{n}\right).
\]
\end{corollary}

The bound of Theorem \ref{the theorem} can be compared with the best bound we know between the solutions of the limiting equations~\eqref{eq:gradient_flow} and~\eqref{eq:gradient_flow_reg}. This bound has been communicated to us by Fanch Coudreuse, and for the sake of completeness, we reproduce its proof in Appendix~\ref{app:fanch}. As we will see then, this proof relies on formal computations on the PDEs, and therefore becomes rigorous when equations~\eqref{eq:gradient_flow} and~\eqref{eq:gradient_flow_reg} admit sufficiently regular solutions. This is for instance the case when $\mathcal F$ is of the form~\eqref{eq:explicit_F} with $V$, $W$ and $f$ sufficiently regular.
\begin{theorem}\label{bound continuous case}
    Let $\mathcal F$ be of the form~\eqref{eq:explicit_F}. We assume that:
    \begin{itemize}
        \item The PDEs~\eqref{eq:gradient_flow} and~\eqref{eq:gradient_flow_reg} admit regular solutions $\rho^0 $ and $\rho^\alpha$ for all times $t\geq 0$;
        \item $\mathcal F$ is $\lambda$-geodesically convex for some $\lambda \in \R$.
    \end{itemize}
    Then, for all $t>0$, 
    \begin{itemize}
        \item if $\lambda=0$ the following inequalities hold:
        \begin{equation}\label{eq:bound continuous}            W_2(\rho^0(t),\rho^\alpha(t))\leq\frac{\alpha}{2}\int_0^t \sqrt{\int |\nabla\ln(\rho^\alpha)|^2\mathrm{d}\rho^\alpha}\mathrm{d}s\leq \sqrt{\frac{\alpha}{2}} \sqrt{t (H(\mu_0)-H(\rho^\alpha_t)+K t)}.
        \end{equation}
    \item if $\lambda\neq0$ the following inequalities hold:
    \begin{equation}\label{eq:bound continuous lambda}
        W_2(\rho^0(t),\rho^\alpha(t))\leq\frac{\alpha}{2}\int_0^t e^{\lambda (s-t)} \sqrt{\int |\nabla\ln(\rho^\alpha)|^2\mathrm{d}\rho^\alpha}\mathrm{d}s\leq \sqrt{ \frac{\alpha}{2}  \frac{1-e^{-2\lambda t}}{2\lambda}} \sqrt{(H(\mu_0)-H(\rho^\alpha_t)+K t)}.
    \end{equation} 
    \end{itemize}
\end{theorem}
\begin{remark}
In order to compare our bound of Theorem \ref{the theorem} and the bound of Theorem \ref{bound continuous case}, let us fix $t>0$, define $\tau=\frac{t}{n}$ and let $n$ go to $+\infty$; the following limits hold true:
 \begin{gather*}
 \lim_{n \to +\infty}\frac{1-(1+\lambda\tau)^{-2n}}{2\lambda}=\frac{1-e^{-2\lambda t}}{2\lambda},\\
 \lim_{n \to +\infty}(1-\lambda_-\tau)^{-n}=e^{\lambda_-t},\\
 \limsup_{n\to+\infty}-H(J_{n,\tau}^\alpha(\mu_0))\leq -H(\rho^\alpha_t).
 \end{gather*}
 The two first lines are direct, and the last one is a consequence of the lower semicontinuity of the entropy together with the convergence result of the entropic JKO scheme towards the solutions of~\eqref{eq:gradient_flow_reg} stated in~\cite{baradat2025usingsinkhornjkoscheme}. Now, taking the limsup in the bound of Theorem~\ref{the theorem}:
\[
\limsup_{n\to+\infty} W_2(J_{n,\tau}^0(\mu_0), J_{n,\tau}^\alpha(\mu_0))\leq 
\left\{
\begin{aligned}   &\sqrt{\alpha t \big(H(\mu_0)-H(\rho^\alpha(t))+Kt\big)}, &&\text{if $\lambda =0$},\\
    &\sqrt{\alpha \frac{1-e^{-2\lambda t}}{2\lambda} \big(H(\mu_0)-H(\rho^\alpha(t))+Kt\big)}, &&\text{if $\lambda \neq0$},
\end{aligned}
\right.
\]
which is twice bound between the limits stated at Theorem~\ref{bound continuous case}. This argument shows that our bound is close to being sharp in $\alpha$, see Theorem~\ref{optimality discrete cororlary} for a precise statement. However, our bound is far from being optimal in $\tau$, since it does not even converge to $0$ as $\alpha$ goes to $0$.
\end{remark} 
\subsection{Comments on the hypothesis}\label{comments hyp}
Let us explain why Hypothesis~\ref{hypothesis} is a natural set of assumptions and give some usual cases where it is verified.
\begin{itemize}
    \item The lower semicontinuity is a natural assumption to ensure the existence of minimizers along the schemes. The lower semicontinuity we require on $\F$ is weaker than the lower semicontinuity for the narrow convergence and stronger than the $W_2$ lower semicontinuity (for which we are not able to prove existence for the entropic scheme). 
    \item The convexity along generalized geodesics is a strengthen version of the geodesic convexity which is equivalent in all practical cases, see subsection \ref{general cnx geo explain} for the definition and some explanations. This hypothesis, which is already required in \cite{ambrosio2005gradient} to obtain the convergence and an explicit convergence rate of the JKO scheme toward its limit, will provide stability through the discrete Evolution Variational Inequality (discrete E.V.I); see Theorem \ref{discrete E.V.I}. 
    \item An important heuristic idea that guided us is that following the flow of $\mathcal{F} + \frac{\alpha}{2} H$ (i.e., solving equation~\eqref{eq:gradient_flow_reg}) for a time $\mathrm{d}t$ should be asymptotically equivalent to first following the flow of $\mathcal{F}$ for a time $\diff t$, and then following the flow of $\frac{\alpha}{2} H$ for a time $\diff t$. Indeed, if $\rho^0$ and $\rho^\alpha$ are regular solutions of~\eqref{eq:gradient_flow} and~\eqref{eq:gradient_flow_reg} respectively, starting from the same initial measure $\mu_0$, then:
    $$
    \partial_t \left(\rho^0*\sigma_{\alpha t}\right)\Big|_{t=0}=\operatorname{div}\left(\mu_0\nabla\frac{\delta\F}{\delta\rho}(\mu_0)\right)+\frac{\alpha}{2}\Delta \mu_0=\partial_t\rho^\alpha\Big|_{t=0}
    $$
With this in mind, in order to compare the flow of $\F$ and of $\F+\frac{\alpha}{2}H$, it is not surprising that we need a bound on the flow of $H$ (i.e., the heat flow) along our solution, which is our last hypothesis. Formally (that is, for sufficiently regular functionals and densities), this assumption is equivalent to the fact that for all sufficiently regular probability measure $\rho$,
\begin{equation*}
    - \int \nabla \frac{\delta \F}{\delta \rho}(\rho) \cdot \nabla \rho \leq K.
\end{equation*}
\end{itemize}

Let us provide some examples relying on the form~\eqref{eq:explicit_F}. Our Hypothesis \ref{hypothesis} covers the following classic cases (extended and proved in~Subsection \ref{proof classic verify}).
\begin{proposition}\label{classic functional light ver}
Let $\mathcal F$ be of the form~\eqref{eq:explicit_F} that is:
\begin{equation*}
\mathcal F : \rho \in \mathcal P_2(\R^d) \longmapsto \int_{\mathbb{R}^d} V(x)\, \rho(x)\diff x 
+ \frac{1}{2} \int_{\mathbb{R}^d} (W * \rho)(x)\, \rho(x)\diff x 
+ \int_{\mathbb{R}^d} f(\rho(x))\diff x,
\end{equation*}
for some functions $V$, $W$ and $f$, where $\mathcal F$ is set to $+\infty$ if $\rho$ is not absolutely continuous with respect to the Lebesgue measure. We assume that:
 \begin{itemize}
     \item $V,W$ are nonnegative, and $f$ is either nonnegative or positively proportional to $s \in \R_+\mapsto s \log s$.
        \item $V,W$ are of regularity $C^1$ with globally Lipchitz derivatives,
        \item $f$ is convex with superlinear growth  and verify the McCann condition, which means that the map $s \mapsto s^df(s^{-d})$ is convex and non increasing on $(0,+\infty)$.
    \end{itemize}
    Then $\mathcal{F}$ satisfies Hypothesis \ref{hypothesis} for some $\lambda$ and $K$.
\end{proposition}
\begin{remark}
    Some classic examples of functions $f$ that verify the previous hypothesis are:
    $$
    f(s)=s\ln(s) \quad \text{and}\quad f(s)=\frac{s^m}{m-1}\quad \text{for }m>1.
    $$
\end{remark}
\subsection{Structure of the Article}

In Section~\ref{notations}, we gather the definitions of the objects appearing in our study: the Wasserstein distance in Subsection~\ref{Wasser}, the relative and Boltzmann entropy in Subsection~\ref{subsec:entropy} and the Schr\"odinger cost in Subsection~\ref{sch cost}. As we will see later, the main ingredient in the proofs is the convexity along the generalized geodesics, presented in Subsection~\ref{general cnx geo def}, and on of its consequences, the discrete Evolution Variational Inequality (E.V.I.). In fact, this central inequality arises very naturally when we adopt a lifting point of view. For completeness, we explain this lifting and how it implies the E.V.I in Subsection~\ref{subsec:EVI}. To our knowledge, this article is the first one to study systematically the entopic JKO scheme for $\lambda$-convex functionals. Therefore, one of the most natural questions is the existence of minimizers along the scheme. This is done in Subsection \ref{Welldefined scheme} for both schemes. Unfortunately, as presented in Subsection~\ref{subsec:uniqueness}, we are only able to show uniqueness of the minimizers of the entropic JKO scheme in some reduced classes of functionals. Finally, in Subsection \ref{proof classic verify}, we show that Hypothesis \ref{hypothesis} is satisfied for a large class of functionals of the form~\eqref{eq:explicit_F}.

Section~\ref{proof of theorem} contains the proof of the main convergence result stated in Theorem~\ref{the theorem}.

Section~\ref{sec:proof_cor} contains the proof of Corollary~\ref{cv rate entropic JKO}.

In section~\ref{Sec:optimality in alpha} we investigate the optimality in $\alpha$ of both the bound at the continuous level and the bound between the two schemes. We are able to find examples where the first inequality of \eqref{eq:bound continuous} and \eqref{eq:bound continuous lambda}  are equalities. At the discrete level, we are able to show an analog version of this sharp bound in the continuous case, and for the same examples, this bound is an equality up to adding a term that goes to~$0$ when $\tau$ goes to $0$.

\subsection*{Acknowledgments} The authors wish to express their gratitude to Fanch Coudreuse for explaining to them the bound presented in Theorem~\ref{bound continuous case}. They also want to thank Hugo Malamut, Maxime Sylvestre and Filippo Santambrogio for interesting discussions and remarks. They finally acknowledge the support of European union via the ERC AdG
101054420 EYAWKAJKOS.

\section{Notations and preliminaries}\label{notations}

\subsection{The Wasserstein distance}
The quadratic Wasserstein distance admits three classical equivalent formulations: 
(1) the primal (Kantorovich) formulation, see Subsection~\ref{par:primal}, 
(2) the dual formulation, see Subsection~\ref{subsec:wasserstein dual formulation}, 
and (3) the dynamic Benamou--Brenier formulation, see Subsection~\ref{subsec: Benamou brenier formulation of wasserstein}.  
Let us start by introducing the primal formulation.
\subsubsection{Primal formulation of the Wasserstein distance}
\label{par:primal}
\begin{definition}[Wasserstein Distance]\label{Wasser}
We denote by \(W_2\) the $2$-Wasserstein distance associated with the Euclidean distance, defined for every $\mu,\nu\in\mathcal P_2(\R^d)$ by:
\begin{equation*}
W_2(\mu, \nu) := \left( \inf_{\gamma \in \Pi(\mu, \nu)} \int_{\R^d \times \R^d} |x - y|^2 \diff\gamma(x, y) \right)^{1/2},
\end{equation*}
where $\Pi(\mu, \nu)$ is the set of all couplings between \(\mu\) and \(\nu\). These are the measures $\gamma\in \mathcal P(\R^d \times \R^d)$ that satisfy that for every $\varphi\in \mathcal{C}^0_c(\R^d,\R)$
\[
\int_{\R^d\times\R^d} \varphi(x)\diff\gamma(x, y)  = \int_{\R^d}\varphi(x)\diff\mu(x) \quad\mbox{and} \quad \int_{\R^d\times\R^d} \varphi(y)\diff\gamma(x, y)  = \int_{\R^d}\varphi(y)\diff\nu(y).
\]
\end{definition}
It is well known that minimizers always exist~\cite{santambrogio2015optimal}. In the following, we call these minimizers optimal transport plans.  
If moreover $\mu$ is absolutely continuous w.r.t.\ the Lebesgue measure, then by Brenier's theorem \cite{Brenier1991}, there exists a unique optimal plan, and this plan is concentrated on the graph of a map $T$, called the \emph{optimal transport map}, which is the gradient of a convex function. Moreover, $T$ is unique $\mu$-almost everywhere. In particular,
\[
\gamma = (I_d,T)_{{_\#}}\mu \qquad \mbox{and} \qquad T_{{_\#}}\mu=\nu,
\]
where if $p \in \mathcal P(\R^n)$, $n \in \N^*$ and $A: \R^n \to \R^m$, $m \in \N^*$ is a measurable map well defined $p$-almost everywhere, $A_{{_\#}}p$ is the push-forward of $p$ by $A$, defined for all Borelian subset $E$ of $\R^m$ by $A_{{_\#}}p(E) = p(A^{-1}(E))$.
Moreover,
\[
T=\argmin_{\substack{S \in L^2(\mu ; \R^d) \\ S{_\#}\mu=\nu}}\left\{\int |x-S(x)|^2\diff\mu(x)\right\}\quad \text{and}\quad W_2^2(\mu,\nu) = \int_{\R^d} |x-T(x)|^2 \, \diff\mu(x).
\]

\subsubsection{Dual formulation}\label{subsec:wasserstein dual formulation}
The minimization problem defining the Wasserstein distance comes with a dual problem which can be expressed as follows (see \cite{santambrogio2015optimal} for more details and a proof of the equality of the primal and dual optimal values):
\[
\frac{W_2^2(\mu, \nu)}{2} = \sup_{\substack{\varphi, \psi \\ \varphi \oplus \psi \leq \frac{|x - y|^2}{2}}} \left\{ \int \varphi \diff\mu + \int \psi \diff\nu \right\}
\]
where the supremum is taken over $(\varphi,\psi)\in L^1(\mu)\times L^1(\nu)$ that satisfies pointwise for all $x , y \in \R^d$ $\varphi(x) + \psi(y) \leq |x-y|^2/2$. Moreover, the previous supremum is achieved, and if the pair $(\varphi,\psi)$ achieves this maximum, $\varphi$ and $\psi$ are called Kantorovich potentials respectively from $\mu$ to $\nu$ and from $\nu$ to $\mu$.

When an optimal map exists, it can be recovered from Kantorovich potentials.
\begin{proposition}
\label{prop:link_primal_dual}
   Let $\mu,\nu\in\mathcal P_2(\R^d)$. Assume that $\mu$ is absolutely continuous with respect to the Lebesgue measure.  
Let $(\varphi,\psi)$ be an associated optimal Kantorovich pair.  
Then $\varphi$ is differentiable $\mu$-a.e. and the optimal transport map $T$ satisfies for $\mu$-almost every $x$:
\[
T(x) = x - \nabla \varphi(x).
\]
Consequently,
\[
W_2^2(\mu,\nu) = \int_{\R^d} |\nabla\varphi(x)|^2 \, \diff\mu(x).
\]
    \end{proposition}
\subsubsection{The Benamou-Brenier formulation}\label{subsec: Benamou brenier formulation of wasserstein}
These two formulations of the Wasserstein distance are said to be \emph{static}: only the initial and final distribution of mass matter. Alternatively, the Benamou-Brenier formula \cite{Benamou2000} offers a dynamic viewpoint by determining a continuous trajectory followed by the distribution of mass during transport. It can stated as follows:
\begin{proposition}
\label{prop:BB}
For all $\mu,\nu \in \mathcal P_2(\R^d)$,
    $$
    \frac{W_2^2(\mu,\nu)}{2}=\inf_{\substack{(\rho, c) \\ \partial_t \rho + \operatorname{div}(\rho c) = 0}} \int_0^1 \int \frac{|c_t|^2}{2} \diff\rho_t \diff t,
    $$
    where the infimum is taken over curves $(\rho_t)$, $t \in [0,1]$, valued in $\mathcal P_2(\R^d)$, and vector fields $c = c_t(x)$, $t \in [0,1]$ and $x \in \R^d$, such that the PDE $\partial_t\rho +\operatorname{div}(\rho c)=0$ holds distributionally, with $\rho(0)=\mu$ and $\rho(1)=\nu$. The infimum is achieved, and if the pair $(\rho_t,c_t)$ achieves this minimum, the curve $(\rho_t)$ is called a geodesic between $\mu$ and $\nu$ and $(c_t)$ is called its associated velocity fields.

    Alternatively, rescaling the time according to a positive parameter $\tau>0$, we have
    for all $\mu,\nu \in \mathcal P_2(\R^d)$:
    $$
    \frac{W_2^2(\mu,\nu)}{2\tau}=\inf_{\substack{(\rho, c) \\ \partial_t \rho + \operatorname{div}(\rho c) = 0}} \int_0^\tau \int \frac{|c_t|^2}{2} \diff\rho_t \diff t
    $$
    where the infimum is taken over weak solutions $(\rho_t,c_t)$, $t \in [0,\tau]$ of $\partial_t\rho +\operatorname{div}(\rho c)=0$ such that $\rho(0)=\mu$ and $\rho(\tau)=\nu$.
\end{proposition}

Let us consider $\mu,\nu \in \mathcal P_2(\R^d)$, and a curve $(\rho_t)$, $t \in [0,1]$, valued in $\mathcal P_2(\R^d)$, connecting $\mu$ to $\nu$. Let us call $\pi_1$ and $\pi_2$ the canonical projections from $(\R^d)^2$ to $\R^d$. It can be proved (see~\cite[Section~5.4]{santambrogio2015optimal}) that $(\rho_t)$ is a geodesic in the sense of Proposition~\ref{prop:BB} if and only if there exists $\gamma \in \Pi(\mu,\nu)$ an optimal transport plan such that for all $t \in [0,1]$, 
\begin{equation}
\label{eq:link_geod_gamma}
\rho_t = (t \pi_2 + (1-t) \pi_1){{_\#}} \gamma.
\end{equation}

Therefore, geodesics can be related to Kantorovich potentials. In fact, if $\mu$ and $\nu$ are absolutely continuous with respect to the Lebesgue measure, in virtue of the Brenier theorem, the optimal transport plan $\gamma$ is unique, and then the Wasserstein geodesic $(\rho_t)$ from $\mu$ to $\nu$ is unique as well. In this case, formula~\eqref{eq:link_geod_gamma} implies the following proposition.

\begin{proposition}
Let $\mu,\nu \in \mathcal P_2(\R^d)$ be absolutely continuous with respect to the Lebesgue measure, let $(\rho_t)$ be the Wasserstein geodesic between $\mu$ and $\nu$, and let $(\varphi,\psi)$ be an associated pair of Kantorovich potentials. Let $a \in C^1_b(\R^d)$. We have
\begin{equation*}
\left. \frac{\mathrm d}{\mathrm d t} \int a\, \diff\rho_t\right|_{t=0}
= - \int \nabla a \cdot \nabla \varphi \, \mathrm d\mu,
\qquad\text{and}\qquad
\left. \frac{\mathrm d}{\mathrm d t} \int a\, \diff\rho_t \right|_{t=1}
= \int \nabla a \cdot \nabla \psi \, \mathrm d\nu .
\end{equation*}
\end{proposition}
\begin{proof}
Let $\gamma$ be the unique transport optimal plan, $(\rho_t)$ the unique geodesic from $\mu$ to $\nu$, and $(\varphi,\psi)$ a pair of Kantorovic potentials. Let $a \in C^1_b(\R^d)$. By formula~\eqref{eq:link_geod_gamma}, we have for all $t \in [0,1]$
\begin{equation*}
    \int a \, \diff \rho_t = \int a(ty + (1-t)x) \, \diff \gamma(x,y).
\end{equation*}
Moreover, as $\mu$ is absolutely continuous with respect to the Lebesgue measure, we already saw in paragraph~\ref{par:primal} and Proposition~\ref{prop:link_primal_dual} that
\begin{equation*}
    \gamma = (I_d, T){_\#} \mu \qquad \mbox{with} \qquad T = I_d - \nabla \varphi,
\end{equation*}
where the second equality holds $\mu$-almost everywhere. Therefore, for all $t \in [0,1]$,
\begin{equation*}
    \int a \, \diff \rho_t = \int a\Big(x - t \nabla \varphi(x)\Big) \, \diff \mu(x).
\end{equation*}
The first part of the statement follows easily using the fact that in virtue of Proposition~\ref{prop:link_primal_dual}, $\nabla \varphi \in L^2(\mu)$. The second part of the statement is obtained in the same way, interverting the roles of $\mu$ and $\nu$.
\end{proof}
\begin{remark}\label{rem:derivative at time 0 and 1 along geodesics}
    Let $\mu,\nu \in \mathcal P_2(\R^d)$ be absolutely continuous with respect to the Lebesgue measure, let $(\rho_t)$ be the Wasserstein geodesic between $\mu$ and $\nu$, and let $(\varphi,\psi)$ be an associated pair of Kantorovich potentials. The previous proposition stated that the following equalities hold in the distributional sense:
    \[
    \left.\partial_t\rho_t\right|_{t=0}=-\operatorname{div}(\mu\nabla\varphi)\quad\text{and}\quad \left.\partial_t\rho_t\right|_{t=1}=-\operatorname{div}(\nu\nabla\psi).
    \]
\end{remark}
\subsection{The entropy functional}\label{subsec:entropy}
The Schrödinger cost is defined by adding an entropy penalization term in the primal formulation of the Wasserstein distance. Hence, before defining the Schrödinger cost, we have to introduce our notion of entropy and its key properties.
\begin{definition}[Relative entropy]
     Let $P, R$ two Borel probality  measures
on $\R^n$, $n \in \N^*$ (in what follows, we will consider the cases $n = d$ and $n = 2d$). We denote by $H(P\|R)$ the relative entropy defined as:

\begin{equation*}
    H(P\|R)=\left\{\begin{aligned}
        &\int\ln\left(\frac{\diff P}{\diff R}\right)\diff P \quad\text{if $P\ll R$},\\
        &+\infty \quad\text{otherwise}.
    \end{aligned}\right.
\end{equation*}
\end{definition}
The next proposition, which is an easy consequence of the Jensen inequality, insures that the relative entropy is defined in $\R_+\cup\{+\infty\}$.

\begin{proposition}\label{prop: positivity entropie}
    Let $P,R$ two Borel probability measures. Whenever $P\ll R$, then $\ln\left(\frac{\diff P}{\diff R}\right)_-\in L^1(P)$. Thus, the relative entropy is well defined. Moreover $$H(P\|R)\in\R_+\cup\{+\infty\}.$$
\end{proposition}

The Boltzmann entropy is defined as the relative entropy with respect to the Lebesgue measure (denoted by $\Leb$). We give a separate definition because $\Leb$ is not a probability measure.
\begin{definition}
    Let $\rho\in\mathcal P_2(\R^n)$ then the Boltzmann entropy (simply called entropy in the following) is defined by:
    \begin{equation*}
    H(\rho)=\left\{\begin{aligned}
        &\int\ln\left(\frac{\diff \rho}{\diff \Leb}\right)\diff \rho \quad\text{if $\rho\ll \Leb$},\\
        &+\infty \quad\text{otherwise}.
    \end{aligned}\right.
\end{equation*}
    where $\Leb$ is the Lebesgue measure on $\R^n$.
\end{definition}
Since the Lebesgue is not a probability measure, Proposition~\ref{prop: positivity entropie} does not apply. However, the next proposition insures that for all $\rho\in\mathcal P_2(\R^d)$, the entropy is always defined as an element of $\R\cup\{+\infty\}$.

\begin{proposition}\label{below bound ent}
    Let $\rho \in \mathcal{P}_2(\mathbb{R}^n)$
then the entropy $H(\rho)$ is always well defined in $\mathbb{R} \cup \{+\infty\}$, and the following bound holds:
\[
H(\rho)\geq -\frac n2\ln\left(\frac 2n\pi e\right)-\frac n 2\ln\left(\int|x|^2\diff\rho(x)\right).
\]
\end{proposition}
\begin{proof}
     If $\rho$ is not absolutely continuous with respect to the Lebesgue measures then $H(\rho)=+\infty$ and the proposition is obvious. Otherwise, if $\rho \ll \mathrm{Leb}$, by Proposition~\ref{prop: positivity entropie},
     \[
\int_{\mathbb{R}^n} 
\ln\!\left( \frac{\mathrm d\rho}{\mathrm d\sigma_t} \right)
\, \mathrm d\rho 
\;\ge 0 ,
\]
for every $t>0$, where $\sigma_t$ is the heat kernel at time $t$ defined in Hypothesis~\ref{hypothesis}. Moreover, we have for all $x \in \R^d$ 
\begin{equation*}
    \ln(\sigma_t(x))=-\frac{|x|^2}{2t}-\frac{n}{2}\ln(2\pi t). 
\end{equation*}
Since $\rho\in\mathcal P_2(\R^n)$, then $\ln(\sigma_t)\in L^1(\rho)$, and in particular:
\[
H(\rho)
\;\ge\;
-\int_{\mathbb{R}^d} \frac{|x|^{2}}{2t}\, \mathrm d\rho(x)
\;-\; \frac{n}{2}\ln(2\pi t),
\]
Optimizing the quantity in the right hand side with respect to $t$ by taking $t=\frac{1}{n}\int|x|^2\diff\rho(x)$ we obtain the desired inequality.
\end{proof}
An important property of the entropy is its behavior corresponding to disintegration of measures, that we state here without a proof. We refer to \cite{leonard2014Survey} for more details, where this property is called additivity of the entropy. It implies for instance that 
pushing forward measures reduces the value of the entropy.

\begin{proposition}\label{prop:addivity of entropy}
Let $n,m\in\N^*$, $P, R$ two Borel probability measures
on $\R^n$ and $T : \R^n \to  \R^m$ a measurable map.

Let $(P^y)_{y\in\R^m}$ and $(R^y)_{y\in\R^m}$ be the measurable families of probability measures obtained by disintegrating $P$ and $R$ with respect to $T$. In other words, $(P^y)_{y\in\R^m}$ and $(R^y)_{y\in\R^m}$ are families respectively well defined for
$T{_\#} P$ and $T{_\#} R$ almost all $y\in \R^m$, that are such that for all $y \in \R^m$ where they are well defined, $P^y$ and $R^y$ are concentrated on the set
$\{x \in \R^n \mbox{ such that } T(x)=y\}$, and such that for all measurable function
$\varphi$ nonnegative or bounded
\[
\int \varphi \, \diff P
=
\int_{\R^m} \left( \int \varphi \, \diff P^y \right) \diff T{_\#} P(y)
\quad \text{and} \quad
\int \varphi \, \diff R
=
\int_{\R^m} \left( \int \varphi \, \diff R^y \right) \diff T{_\#} R(y).
\]
Then, we have
\begin{equation*}
H(P \| R)
=
H(T{_\#} P \| T{_\#} R)
+
\int H(P^y \| R^y)\, \diff T{_\#} P(y).
\end{equation*}
In particular,
\[
H(T{_\#} P \| T{_\#} R) \le H(P \| R).
\]
\end{proposition}

\subsection{The Schrodinger cost}\label{sch cost}
The Schrödinger cost can be seen as a regularized version of the optimal transport obtained by adding a term of entropy in the infimum. 
\begin{definition}[Schr\"odinger cost]\label{def:sch cost}  For $\mu,\nu\in\mathcal P_2(\R^d)$, $\alpha >0$ and $\tau>0$, we denote by $\Sch^{\alpha \tau}$ the Schrödinger cost with parameter $\alpha \tau$ defined as:
     \[\Sch^{\alpha\tau}(\mu,\nu)=\inf_{\gamma\in\Pi(\mu,\nu)}\alpha \tau H(\gamma\|R_{\alpha\tau})=\inf_{\gamma\in\Pi(\mu,\nu)} \left\{ \int \frac{|x-y|^2}{2}\diff\gamma(x,y)+\alpha \tau H(\gamma)+\frac{\alpha\tau d}{2} \ln(2\pi\alpha\tau)\right\}, \]
     where $\Pi(\mu, \nu)$ is the set of all couplings between \(\mu\) and \(\nu\)  and $R_{\alpha\tau}$ is the measure on $\R^d\times\R^d$ with density 
     \begin{equation*}
     R_{\alpha\tau}(x,y)=\frac{1}{\sqrt{2\pi\alpha\tau}^d}\exp\left(-\frac{|x-y|^2}{2\alpha\tau}\right), \qquad x,y \in \R^d.
     \end{equation*}
\end{definition}
     The Shr\"odinger cost $\Sch^{\alpha\tau}(\mu,\nu)$ is finite if and only if $H(\mu)<+\infty$ and $H(\nu)<+\infty$, see~\cite{leonard2014Survey}. When $\Sch^{\alpha\tau}(\mu,\nu)$ is finite, in view of the strict convexity of $H$, there exists a unique minimizer $\gamma$.

Just like the Wasserstein distance, the Schrödinger cost has a dynamic formulation of Benamou-Brenier type, here written for the rescaled time $t \in [0,\tau]$ (see~\cite{Gentil2017} for a proof). 
\begin{proposition}[Dynamic formulation of the Schrödinger c
ost]\label{Sch Benamou Brenier}
 Let $\mu,\nu\in\mathcal P_2(\R^d)$ with $H(\mu)<+\infty$ and $H(\nu)<+\infty$. The Schrödinger cost can be expressed by one of the following equivalent formulations:
\begin{align*}
\frac{\Sch^{\alpha\tau}(\mu,\nu)}{\tau} &= \alpha H(\mu) + \min \left\{ \frac{1}{2} \int_0^\tau \int |\vec{v_t}|^2 \, \, \mathrm{d}\rho_t \, \mathrm{d}t\ \middle|\ 
\begin{aligned}
    &\rho(0) = \mu, \quad \rho(\tau) = \nu, \\
    &\partial_t \rho + \operatorname{div}(\rho \vec{v}) = \frac{\alpha}{2} \Delta \rho
\end{aligned}
\right\},\\
&= \frac{\alpha}{2} (H(\mu) + H(\nu)) + \min \left\{ \frac{1}{2} \int_0^\tau \int |c_t|^2 + \left| \frac{\alpha}{2} \nabla \ln(\rho_t) \right|^2 \, \mathrm{d}\rho_t \, \mathrm{d}t \ \middle|\ 
\begin{aligned}
    &\rho(0) = \mu, \quad \rho(\tau) = \nu, \\
    &\partial_t \rho + \operatorname{div}(\rho c) = 0
\end{aligned}
\right\}.
\end{align*}
Moreover, the infimum in each case is achieved for the same $\rho$ and $c = \vec{v} - \frac{\alpha}{2} \ln(\rho)$.
\end{proposition}
Since our goal will be to compare schemes involving respectively the Wasserstein distance and the Schrödinger cost, it will be useful to be able to compare these quantities. Although elementary, the next proposition is the first step in comparing our schemes. 
\begin{proposition}\label{below bound sch by W-2}
    For all $\mu,\nu\in\mathcal P_2(\R^d)$ with $H(\mu)<+\infty$ and $H(\nu)<+\infty$, there holds:
    \begin{align*}
\frac{\Sch^{\alpha\tau}(\mu,\nu)}{\tau}&\geq \frac{\alpha}{2} (H(\mu) + H(\nu))+\frac{W_2^2(\mu,\nu)}{2\tau}+\frac{\alpha^2}{8}\int_0^\tau|\nabla\ln(\rho^\alpha_t)|^2\diff\rho^\alpha_t\diff t \\
&\geq \frac{\alpha}{2} (H(\mu) + H(\nu))+\frac{W_2^2(\mu,\nu)}{2\tau},
\end{align*}
where $t\mapsto\rho^\alpha_t$ is the interpolation given by the Benamou-Brenier formulation of the Schrodinger cost defined in Definition~\ref{Sch Benamou Brenier}.
\begin{proof}

This bound can be easily deduced from the the second characterization of Proposition~\ref{Sch Benamou Brenier}: 
calling $(\rho^\alpha,c^\alpha)$ the corresponding minimizer, 
\begin{equation}\label{eq: sch benamou brenier eval au min}    
\frac{\Sch^{\alpha\tau}(\mu, \nu)}{\tau} = \alpha \frac{H(\mu) + H(\nu)}{2} +  \int_0^\tau \int \frac{|c_t^\alpha|^2}{2} \diff\rho_t \diff t + \frac{\alpha^2}{8} \int_0^\tau \int |\nabla \ln \rho_t^\alpha|^2 \diff\rho_t^\alpha \, \diff t.
\end{equation}
But the Benamou-Brenier formulation of the Wasserstein distance implies that:
\[
\frac{W_2^2(\mu,\nu)}{2} =\inf_{\substack{(\rho, c) \\ \partial_t \rho + \operatorname{div}(\rho c) = 0}} \int_0^\tau \int \frac{|c_t|^2}{2} \, \diff\rho_t\diff t\leq \int_0^\tau \int \frac{|c_t^\alpha|^2}{2} \diff\rho_t^\alpha \diff t.
\]
Plugging this inequality in formula~\eqref{eq: sch benamou brenier eval au min}, we obtain:
\[
\frac{\Sch^{\alpha\tau}(\mu, \nu)}{\tau}\geq \alpha \frac{H(\mu) + H(\nu)}{2} + \frac{W_2^2(\mu, \nu)}{2\tau}+ \frac{\alpha^2}{8} \int_0^\tau \int |\nabla \ln \rho^\alpha_t|^2 \rho^\alpha_t \, \diff t.
\]
We conclude by positivity of $\int_0^\tau \int |\nabla \ln \rho^\alpha_t|^2 \rho^\alpha_t \, \diff t$.
\end{proof}
\end{proposition}

\subsection{Generalized geodesics convexity}\label{general cnx geo def}
A crucial assumption for our study is the convexity of the functional $\F$ along generalized geodesics. This subsection follows the definitions and properties from~\cite{ambrosio2005gradient}. The starting point of this notion is the following proposition, necessary to define what is a generalized geodesic. 
\begin{proposition}\label{3 plan optimal}
    Let $\mu, \nu, \rho \in  \mathcal{P}_2(\mathbb{R}^d)$. There exists $\Pi\in\mathcal P_2((\R^d)^3)$ such that $(\pi_1,\pi_2){_\#}\Pi$ is an optimal transport plan between $\mu$ and $\nu$ and $(\pi_1,\pi_3){_\#}\Pi$ is an optimal transport plan between $\mu$ and $\rho$, where $\pi_1,\pi_2,\pi_3$ denote the canonical projections of $(\mathbb{R}^d)^3$ onto $\mathbb{R}^d$.
\end{proposition}
The proof can be found in \cite[Lemma~2.1]{ambrosio_users_2013}.
\begin{remark}
    If $\mu$ is absolutely continuous with respect to the Lebesgue measure and $T,R$ are the optimal maps given by the Brenier theorem (see paragraph~\ref{par:primal}) then the measure $\Pi$ given by Proposition~\ref{3 plan optimal} is unique, and $\Pi=(I_d,T,R){_\#}\mu$.
\end{remark} 

We can now introduce the notion of generalized geodesics.
\begin{definition}
    Let $\mu,\nu,\rho \in \mathcal P_2(\R^d)$. The curve $(\rho_t)_{t \in [0,1]}$ is called a generalized geodesic between $\nu$ and $\rho$ based on $\mu$, if there exists a measure $\Pi \in \mathcal P((\R^d)^3)$ such that $(\pi_1,\pi_2){_\#}\Pi$ is an optimal transport plan between $\mu$ and $\nu$ and $(\pi_1,\pi_3){_\#}\Pi$ is an optimal transport plan between $\mu$ and $\rho$, and such that for all $t \in [0,1]$, $\rho_t = ((1-t) \pi_2 + t \pi_3){_\#} \Pi$.
\end{definition}

\begin{remark}
As a consequence of Proposition~\ref{3 plan optimal}, for all $\mu,\nu,\rho \in \mathcal P_2(\R^d)$, there exists a generalized geodesic between $\nu$ and $\rho$ based on $\mu$. Also, due to formula~\eqref{eq:link_geod_gamma}, when $\nu$ or $\rho$ equals $\mu$, any generalized geodesic between $\nu$ and $\rho$ based on $\mu$ is simply a geodesic between $\nu$ and $\rho$.
\end{remark}
Now that we have defined generalized geodesics, we can define the notion of $\lambda$-convexity along generalized geodesics of a given functional $\F$.
\begin{definition}\label{def:convexity along generalized geodesic}
The functional $\mathcal{F}$ is said to be $\lambda$-convex along generalized geodesics if for all measures $\mu, \nu, \rho \in  \mathcal{P}_2(\mathbb{R}^d)$ and for every $\Pi \in  \mathcal{P}_2((\mathbb{R}^d)^3)$ such that $(\pi_1,\pi_2){_\#}\Pi$ is an optimal transport plan between $\mu$ and $\nu$ and $(\pi_1,\pi_3){_\#}\Pi$ is an optimal transport plan between $\mu$ and $\rho$, for all $t \in [0,1]$, we have
\[
\mathcal{F}((t\pi_3 + (1-t)\pi_2){_\#} \Pi) \leq t\, \mathcal{F}(\rho) + (1-t)\, \mathcal{F}(\nu) - \frac{\lambda}{2} t(1-t) \int |y - z|^2 \, d(\pi_2,\pi_3){_\#}\Pi(y,z).
\]
\end{definition}

\begin{remark}
    If $\mathcal{F}$ is convex along generalized geodesics, then it is also convex along geodesics in the Wasserstein space.
\end{remark}
At first sight, convexity along generalized geodesics may seem involved. Yet, on the one hand, most classical functionals known to be convex along geodesics are also convex along generalized geodesics. On the other hand, by considering a slightly unusual perspective on the JKO scheme, we show that this hypothesis naturally arises. This is the purpose of paragraph~\ref{general cnx geo explain} below, which is independent of the rest of the article, and whose aim is to help the reader get acquainted with this notion. 
\subsection{A useful Hilbertian interpretation of the JKO scheme}\label{subsec:EVI}
\subsubsection{JKO as a gradient flow in a Hilbert space.}\label{general cnx geo explain}
The purpose of this paragraph is to explain why it is natural to assume $\mathcal F$ to be geodesically convex along generalized geodesics when computing the iterates of the JKO scheme. In fact, the distance $W_2$ is not just any distance. It is specifically related to the $L^2$-norm, which is Hilbertian. Geodesic convexity along generalized geodesics is connected to convexity in $L^2$ through this connection. Let $(\Omega,\mathbb P)=([0,1]^d,\Leb)$ be seen as a probability space. To our functional $\F$, we associate the following functional $F$:
\begin{equation}  
\label{F}
F:\left\{\begin{aligned}
     \mathcal{H}:=L^2(\Omega,\mathbb{P};\mathbb{R}^d) &\longrightarrow \mathbb{R} \\
   X &\longmapsto \mathcal{F}(X{_\#}\mathbb{P}).
\end{aligned}\right.
\end{equation}
When possible, given $X \in \mathcal H$, we define the proximal operator:
\[
E_\tau(X) := \arg\min_{Y \in \mathcal H} \left\{ \frac{\|X - Y\|^2_{L^2}}{2\tau} + F(Y) \right\}.
\tag{Prox Op}\label{Prox OP}
\]
As usual, when minimizers are not unique, $E_\tau(X)$ denotes any minimizer.

\begin{theorem}[Equivalence between schemes]\label{equiv scheme}
Let $X_0\in L^2(\Omega,\mathbb P,\R^d)$ be such that $\mu_0:=X_0{_\#}\mathbb P\ll\Leb$. There exists a minimizer $E_\tau(X_0)$ in \eqref{Prox OP} if and only if there exist a minimizer $J_\tau^0(\mu_0)$ in \eqref{JKO}. More precisely, the following holds:
\[
\left\{ Y{_\#} \mathbb P, \, Y \in \argmin_{Y \in \mathcal H} \left\{\frac{\|X_0 - Y\|^2_{L^2}}{2\tau} + F(Y) \right\}\right\}=\argmin_{\rho \in \mathcal{P}_2(\mathbb{R}^d)} \left\{ \frac{W_2^2(\mu, \rho)}{2\tau}  + \mathcal{F}(\rho) \right\}.
\]
In particular, if $X_0{_\#}\mathbb{P} = Y_0{_\#}\mathbb{P}$, then $$  \left\{ Y{_\#} \mathbb P, \,Y\in\argmin_{Y \in \mathcal H} \left\{ \frac{\|X_0 - Y\|^2_{L^2}}{2\tau} + F(Y) \right\}\right\} =\left\{ Y{_\#} \mathbb P, \,Y\in\argmin_{Y \in \mathcal H} \left\{ \frac{\|Y_0 - Y\|^2_{L^2}}{2\tau} + F(Y) \right\}\right\}.$$
\end{theorem}
Under Hypothesis~\ref{hypothesis}, Proposition~\ref{prop:linear increasing of entropy along JKO alpha neq0} of the Appendix implies that if the starting point of the JKO scheme has finite entropy, and hence is absolutely continuous with respect to the Lebesgue measure, then this property remains true for all of its iterates. This allows to use Theorem~\ref{equiv scheme} iteratively. In fact, more involved proofs relying on an adaptation of \cite[Proposition 6.13]{kallenberg2002} would imply the same result without assuming that $X_0{_\#}\mathbb P$ is absolutely continuous, but we do not want to enter these details as we do not need them. 
\begin{proof}

By definition of $F$,
\begin{align*}
    \inf_{Y \in \mathcal H} \left\{ \frac{\|X_0 - Y\|^2_{L^2}}{2\tau} + F(Y) \right\} 
    &= \inf_{\rho \in \mathcal{P}(\mathbb{R}^d)} \inf_{\substack{Y \in \mathcal H \\ Y{_\#}\mathbb{P} = \rho}} \left\{ \frac{\|X_0 - Y\|^2_{L^2}}{2\tau} + \mathcal{F}(\rho) \right\} \\
    &= \inf_{\rho \in \mathcal{P}_2(\mathbb{R}^d)} \left\{ \frac{1}{2\tau}\inf_{\substack{Y \in \mathcal H \\ Y{_\#}\mathbb{P} = \rho}} \|X_0 - Y\|^2_{L^2} + \mathcal{F}(\rho) \right\}.
\end{align*}
Let us show that given $\rho \in \mathcal P_2(\R^d)$,
\begin{equation}
\label{eq:W_link_L2}
    \inf_{\substack{Y \in \mathcal H \\ Y{_\#}\mathbb{P} = \rho}} \|X_0 - Y\|^2_{L^2} = W_2^2(\mu_0, \rho).
\end{equation}
First, for all $Y \in \mathcal H$ such that $Y{_\#} \mathbb P = \rho$, calling $\gamma := (X_0,Y){_\#} \mathbb P$, we have $\gamma \in \Pi(\mu_0,\rho)$, and so
\begin{equation*}
    \| X_0 - Y \|^2_{L^2} = \int |x-y|^2 \diff \gamma(x,y) \geq W_2^2(\mu_0,\rho).
\end{equation*}
On the other hand, since $X_0{_\#}\mathbb P=\mu_0\ll\Leb$, Brenier's theorem provides a map $T_\rho \in L^2(\mu_0; \R^d)$ such that $T_\rho{_\#}\mu_0=\rho$ and $W_2^2(\mu_0,\rho)=\int |x-T_\rho(x)|^2\diff\mu_0(x)$. Therefore, considering $Y := T_\rho(X_0)$, we have
\begin{equation*}
    \| X_0 - Y \|_2^2 = \int |X_0 - T_\rho(X_0)|^2 \diff \mathbb P = \int |x - T_\rho(x)|^2 \diff \mu_0(x) = W_2^2(\mu_0,\rho),
\end{equation*}
and~\eqref{eq:W_link_L2} follows. Therefore,
\[
\inf_{Y \in \mathcal H} \left\{ \frac{\|X_0 - Y\|^2_{L^2}}{2\tau} + F(Y) \right\} = \inf_{\rho \in \mathcal{P}_2(\mathbb{R}^d)} \left\{ \frac{W_2^2(\mu_0, \rho)}{2\tau} + \mathcal{F}(\rho) \right\}.
\]

Moreover, our proof shows that there is a one-to-one correspondence between the minimizers: 
If $Y$ is a minimizer on the left-hand side, then $\rho = Y{_\#}\mathbb{P}$ is a minimizer on the right-hand side. Conversely, if $\rho$ minimizes the right-hand side, then $Y = T_\rho(X_0)$ is the unique minimizer of the left-hand side such that $Y{_\#}\mathbb{P} = \rho$. 
\end{proof}

Therefore, provided $\mu_0=X_0{_\#}\mathbb P\ll \Leb$, finding the minimizer in~\eqref{JKO} is equivalent to solving the minimization problem~\eqref{Prox OP} in the Hilbert space $\mathcal H$. In view of \eqref{Prox OP}, it would be convenient to assume $F$ to be convex. However this assumption is very restrictive in terms of $\mathcal F$, as it fails for all functionals $\mathcal F$ of type $\rho \mapsto \int f(\rho)$, for any $f$ convex and superlinear. Yet, the fact that $F$ is of the form $X \mapsto \mathcal{F}(X{_\#}\mathbb{P})$ allows us to find a weaker assumption guaranteeing good properties for~\eqref{Prox OP}. Indeed, the Brenier theorem together with the proof of Theorem~\ref{equiv scheme} implies
\[
\inf_{Y \in \mathcal H} \left\{ \frac{\|X_0 - Y\|^2_{L^2}}{2\tau} + F(Y) \right\}
= \inf_{\substack{Y = \nabla\psi(X_0) \\ \psi \text{ convex}}} \left\{ \frac{\|X_0 - Y\|^2_{L^2}}{2\tau} + F(Y) \right\}.
\]
Hence, to ensure existence along the scheme, it suffices that the function 
\[
F_\tau^{X_0}:Y \mapsto \frac{\|X_0 - Y\|^2_{L^2}}{2\tau} + F(Y)
\]
admits a minimizer in the set
\[
C(X_0) = \left\{ T(X_0) \mid T \text{ is the gradient of a convex function} \right\}.
\]
Let us take a closer look at the structure of this set. 
\begin{proposition}[Properties of $C(X)$]
For every $X \in \mathcal{H}$ such that $X{_\#}\mathbb{P} \ll \text{Leb}$, the set $C(X)$ satisfies:
\begin{itemize}
    \item $C(X)$ is convex.
    \item $C(X)$ is closed for the strong topology of $L^2$.
    \item $C(X)$ is stable under multiplication by a positive constant.
    \item For every $\rho \in \mathcal{P}_2(\mathbb{R}^d)$, there exists a unique $Y_\rho^X \in C(X)$ such that $Y_\rho^X{_\#}\mathbb{P} = \rho$ and $W_2(\rho, X{_\#}\mathbb{P}) = \|X - Y_\rho^X\|_{L^2}$.
\end{itemize}
\end{proposition}

Thus, a natural assumption on $\F$ is that $F$ is convex on $C(X)$ for every $X$. This assumption is equivalent to convexity along generalized geodesics.
\begin{proposition}\label{equiv cnx}
The functional $\mathcal{F}$ is {$\lambda$-}convex along generalized geodesics in the sense of Definition~\ref{def:convexity along generalized geodesic} for all absolutely continuous base point if and only if for every $X$ such that $X{_\#}\mathbb P\ll\Leb$, $F$ defined by formula~\ref{F} is {$\lambda$-}convex on $C(X)$.
\end{proposition}
\begin{proof}
    Let us assume that $\F$ is {$\lambda$-}convex along generalized geodesics of absolutely continuous base point, and show that for every $X$ such that $X{_\#}\mathbb P\ll\Leb$, $F$ is {$\lambda$-}convex on $C(X)$. Let us fix $X\in \mathcal H$ and consider $Y,Z\in C(X)$. Since $Y$ and $Z$ are in $C(X)$, by the converse of Brenier's theorem, calling $\Pi=(X,Y,Z){_\#}\mathbb{P}$ and $\pi_1,\pi_2,\pi_3$ the canonical projections from $(\R^d)^3$ to $\R^d$, the plans $(\pi_1,\pi_2){_\#}\Pi$ and $(\pi_1,\pi_3){_\#}\Pi$ are optimal between their marginals. By Definition~\ref{general cnx geo def} of $\lambda$-convexity along generalized geodesics, for all $t \in [0,1]$,
    \[
\mathcal{F}((t\pi_3 + (1-t)\pi_2){_\#} \Pi) \leq t\, \mathcal{F}(Z{_\#}\mathbb{P}) + (1-t)\, \mathcal{F}(Y{_\#}\mathbb{P}) - \frac{\lambda}{2} t(1-t) \int |y - z|^2 \, d(\pi_2,\pi_3){_\#}\Pi(y,z).
\]
or equivalently
\[
 F(tZ+(1-t)Y) \leq t\, F(Z) + (1-t)\, F(Y) - \frac{\lambda}{2} t(1-t) \|Z-Y\|^2.
\]
This exactly means that $F$ is $\lambda$-convex on $C(X)$.

Now, let us assume that for every $X$ such that $X{_\#}\mathbb P\ll\Leb$, $F$ is {$\lambda$-}convex on $C(X)$ and show that $\F$ is {$\lambda$-}convex along generalized geodesics of absolutely continuous base point. Fix $\mu\in\mathcal P_2(\R^d)$ such that $\mu\ll\Leb$. Take $X$ the optimal map from the Lebesgue measure on $[0,1]^d$ to $\mu$. We have $X\in\mathcal{H}$ and $X{_\#}\mathbb{P}=\mu$. Let $\Pi\in \mathcal P_2((\R^d)^3)$ a three-plan such that $\pi_1{_\#}\Pi=\mu$ and such that $(\pi_1,\pi_2){_\#}\Pi$ and $(\pi_1,\pi_3){_\#}\Pi$ are optimal between their marginals. Since $\mu\ll\Leb$ the Brenier theorem ensures that the plans $(\pi_1,\pi_2){_\#}\Pi$ and $(\pi_1,\pi_3){_\#}\Pi$ are concentrated on the graphs of $T,R$, which are gradients of convex functions. Letting $Y=T(X)$ and $Z=R(X)$, $Y,Z\in C(X)$ and $\Pi=(X,Y,Z){_\#}\mathbb{P}$. Then the $\lambda$-convexity of $F$ on $C(X)$ implies that for all $t \in [0,1]$,
\[
F(tZ+(1-t)Y) \leq t\, F(Z) + (1-t)\, F(Y) - \frac{\lambda}{2} t(1-t) \|Z-Y\|^2,
\]
or equivalently
\[
\mathcal{F}((t\pi_3 + (1-t)\pi_2){_\#} \Pi) \leq t\, \mathcal{F}(\pi_3{_\#}\Pi) + (1-t)\, \mathcal{F}(\pi_2{_\#}\Pi) - \frac{\lambda}{2} t(1-t) \int |y - z|^2 \, d(\pi_2,\pi_3){_\#}\Pi(y,z).
\]
Hence, $\mathcal F$ is $\lambda$-convex along generalized geodesics of base point $\mu$.
\end{proof}
\subsubsection{Discrete EVI}
One of the most important consequences of the convexity of $\mathcal F$ along generalized geodesics is the discrete Evolution Variational Inequality (EVI), a stability inequality that will play a crucial role in our work.
\begin{theorem}[Discrete EVI, Ambrosio, Gigli, Savaré \cite{ambrosio2005gradient}] \label{discrete E.V.I}
Let $\tau>0$. Let $\F$ be $\lambda$-convex along generalized geodesics and $W_2$-l.s.c. For every $\mu \in \mathcal{P}_2(\mathbb{R}^d)$ such that~\eqref{JKO} admits a minimizer $J^0_\tau(\mu)$ and every $\rho \in \mathcal P_2(\R^d)$, we have:
\[
\frac{1}{2\tau} \left( W_2^2(\rho, J_\tau^0(\mu)) - W_2^2(\rho, \mu) \right)
\leq \mathcal{F}(\rho) - \mathcal{F}(J_\tau^0(\mu)) - \frac{1}{2\tau} W_2^2(J_\tau^0(\mu), \mu)
{- \frac{\lambda}{2} W_2^2(\rho, J_\tau^0(\mu))}.
\]
\end{theorem}
\begin{remark}
    In Theorem~\ref{welldefined JKO}, we provide conditions for the existence of $J_\tau^0(\mu).$
\end{remark}
This inequality can be seen as an easy consequence of the same inequality at the level of $\mathcal H$, which can be stated as follows. As usual, we only state it in the case of absolutely continuous measures.
\begin{proposition}\label{prop:e.V.i at level of H}
Let $\tau>0$. Let $\F$ be $\lambda$-convex along generalized geodesics and $W_2$-l.s.c. Let $X\in\mathcal{H}$, be such that $X{_\#}\mathbb P\ll \Leb$ and such that~\eqref{Prox OP} admits a minimizer $E_\tau(X)$. Then for every $V \in C(X)$, the following inequality holds:
\[\frac{1}{2\tau} \left( \|V - E_\tau(X)\|^2_{L^2} - \|V - X\|^2_{L^2} \right)
\leq F(V) - F(E_\tau(X)) - \frac{1}{2\tau} \|X - E_\tau(X)\|^2_{L^2} {-\frac{\lambda}{2} \|V - E_\tau(X)\|^2_{L^2}}.
\]
\end{proposition}
\begin{proof}[Proof of Proposition \ref{prop:e.V.i at level of H}]
Let $X\in\mathcal H$ be such that $X{_\#}\mathbb P\ll \Leb$. Assume that $\mathcal{F}$ is $\lambda$-convex along generalized geodesics and $\tau<\frac{1}{\lambda_-}$. Then, by Proposition \ref{equiv cnx}, the penalized functional 
\[
F_\tau^X(Y) := \frac{\|X - Y\|^2_{L^2}}{2\tau} + F(Y)
\tag{Prox map}\label{Prox map}\]
is $\frac{1 + \lambda\tau}{\tau}$-convex on $C(X)$. If $\F$ is $W_2$-l.s.c,  then $E_\tau(X)$ exists, see Theorem \ref{welldefined JKO}. If $F$ is differentiable, then for every $V \in C(X)$, we have
\[
F_\tau^X(V) \geq F_\tau^X(E_\tau(X)) + \underbrace{\left\langle \nabla F_\tau^X(E_\tau(X)), V - E_\tau(X) \right\rangle}_{= 0} + \frac{{1 + \lambda\tau}}{2\tau} \|E_\tau(X) - V\|^2_{L^2}.
\]
In fact the convexity of $F_\tau^X$ on $C(X)$ is enough to establish this inequality. By standard results in convex analysis (see \cite{BauschkeCombettes2017} Definition 6.38, Theorem 16.3 and Example 16.13), the point $E_\tau(X)$ is a minimizer of $F_\tau^X$ on the convex set $C(X)$ if and only if 
$$0\in\partial (F_\tau^X+\iota_{C(X)})(E_\tau(X))=\partial F_\tau^X(E_\tau(X))+N_{C(x)}(E_\tau(X))$$ 
where  the subdifferential of $F_\tau^X$ is defined by 
\begin{equation*}
\partial F_\tau^X(Y):=\left\{U \in \mathcal H  \, | \, \forall V\in C(X),\,F_\tau^X(V)\geq\left\langle U,V-Y\right\rangle+F_\tau^X(Y)\right\},
\end{equation*}
and the convex indicator function $\iota_{C(X)}$ and the normal cone $N_{C(X)}$ are defined by 
\[\iota_{C(X)}(Y):=\left\{\begin{aligned}
    &+\infty &&\text{if $Y\notin C(X)$},\\
   &0 &&\text{if $Y\in C(X)$},
\end{aligned} 
\qquad\text{and}\qquad N_{C(X)}(Y):=\{U \in \mathcal H \, | \, \forall V\in C(X), \, \left\langle U,V-Y\right\rangle\leq0\}.\right.\] 
This means that there is a point $U$ in the subdifferential of $F_\tau^X$ at point $E_\tau(X)$ such that for every $V\in C(X)$ the following holds:
\[
\left\langle U, V - E_\tau(X) \right\rangle\geq 0.
\]
Then by definition of the subdifferential we obtain:
\begin{align*}
F_\tau^X(V) &\geq F_\tau^X(E_\tau(X)) + \left\langle U, V - E_\tau(X) \right\rangle + \frac{{1 + \lambda\tau}}{2\tau} \|E_\tau(X) - V\|^2_{L^2}\\
&\geq F_\tau^X(E_\tau(X))  + \frac{{1 + \lambda\tau}}{2\tau} \|E_\tau(X) - V\|^2_{L^2},
\end{align*}
 which is the desired result.
\end{proof}

This discrete-time inequality has a continuous counterpart obtained in the limit $\tau \to 0$. When the solution of the PDE~\eqref{eq:gradient_flow} is well defined, being a solution is equivalent to verifying this continuous-time inequality. However, this inequality still makes sense when the solution of the PDE \eqref{eq:gradient_flow} is not well-defined. Therefore it can be used to define what is a gradient flow of a functional that is $\lambda$-convex along generalized geodesics, see \cite{ambrosio2005gradient}.
\begin{definition}
    Let $\F$ be $\lambda$-convex along generalized geodesics and $W_2$-l.s.c. A curve in $\mathcal P_2(\R^d)$ is called a gradient flow of $\F$ in the Wasserstein space if for all $\rho\in\mathcal P_2(\R^d)$ and all $t\geq0$, the following inequality holds:
\[
\frac{1}{2} \frac{\mathrm{d}}{\mathrm{d}t} W_2^2(\rho_t, \rho)
\leq \mathcal{F}(\rho) - \mathcal{F}(\rho_t)
{- \frac{\lambda}{2} W_2^2(\rho_t, \rho)}.
\]
\end{definition}
\begin{theorem}[EVI Characterization of Gradient Flows]\label{E.V.I}
Let $\F$ be $\lambda$-convex along generalized geodesics and $W_2$-l.s.c. The limiting curve $(\rho_t)_{t \geq 0}$ defined by equation~\eqref{eq:def_rho0} is the only gradient flow of $\mathcal{F}$ in the Wasserstein space starting from $\mu_0$.
\end{theorem}

\subsection{Existence of minimizers along the schemes}\label{Welldefined scheme}
\subsubsection{Existence and uniqueness along the JKO scheme}
The well-posedness of the JKO scheme for functionals that are convex along generalized geodesics has been established in \cite{ambrosio2005gradient}. For the sake of completeness, we briefly revisit the arguments, using the framework introduced in the previous section. 

Surprinsingly, in terms of topology, this framework allows us to establish the result for functionals \(\mathcal{F}\) that are only lower semicontinuous with respect to $W_2$. This is weaker than requiring \(\mathcal{F}\) to be lower semicontinuous in the sense of Hypothesis~\ref{hypothesis}, and hence the direct method does not apply straightforwardly. Since in this work, we only deal with measures that are absolutely continuous with respect to the Lebesgue measure, we restrict ourselves to proving the existence of the JKO scheme in this situation.
\begin{theorem}\label{welldefined JKO}
     If $\F$ is $W_2$-l.s.c and $\lambda$-convex along generalized geodesics, if $\tau<\frac{1}{\lambda_-}$, and if there exists $\nu \in \mathcal P_2(\R^d)$ such that $\F(\nu)<+\infty$, then for all $\mu\in\mathcal P_2(\R^d)$ such that $\mu\ll\Leb$ and for all $X\in\mathcal H$ such that $X{_\#}\mathbb P=\mu$, $J_\tau^0(\mu)$ and $E_\tau(X)$ are well defined.
     \end{theorem}
     Due to this theorem, the JKO scheme can be defined iteratively.
\begin{corollary}\label{welldefined iterative JKO}
     If $\F$ is $\lambda$-convex along generalized geodesics, $W_2$-lower semicontinuous and $\tau<\frac{1}{\lambda_-}$, then there exists a unique sequence $(J_{k,\tau}^0(\mu))_{k\geq0}$ satisfying the induction relation $J_{k,\tau}^0(\mu)=J_\tau^0(J_{k-1,\tau}^0(\mu))$, for all $k \in \N^*$.
\end{corollary}
\begin{remark}
    The proof of Corollary~\ref{welldefined iterative JKO} is done in \cite{ambrosio2005gradient}. Here Theorem~\ref{welldefined JKO} only implies it in the case when $J_{k,\tau}^0(\mu)\ll\Leb$ for all $k \in \N$. This will be verified is the rest of the article, as a consequence of Hypothesis~\ref{hypothesis}, see Proposition~\ref{prop:linear increasing of entropy along JKO alpha neq0}.
\end{remark}
 In order to prove Theorem~\ref{welldefined JKO}, we will need the following preliminary results.
\begin{proposition}
\label{prop:equiv_lsc}
    The functional $\F:\mathcal{P}_2(\R^d)\mapsto\R$ is l.s.c with respect to $W_2$ if and only if $F$ defined in~\eqref{F} is l.s.c with respect to the strong topology of $L^2$.
\end{proposition}
\begin{proof}
    We start by showing that if $\F$ is l.s.c with respect to $W_2$, then $F$ is l.s.c with respect to the strong topology of $L^2$. Let $(X_n)_n\in \mathcal H^\N$ be a sequence strongly converging in $L^2$ to $X$. Then 
    \begin{equation*}
        W_2(X_n{_\#}\mathbb{P},X{_\#}\mathbb{P})\leq\|X_n -X\|_{L^2}\xrightarrow[n\to+\infty]{} 0.
    \end{equation*}
     But $\F$ is l.s.c with respect to $W_2$, so 
     \begin{equation*}
         F(X) = \F(X_\#\mathbb{P}) \leq \liminf\limits_{n\to+\infty}\F(X_n{_\#}\mathbb{P}) = \liminf\limits_{n\to+\infty}F(X_n),
     \end{equation*}
   and $F$ is l.s.c for the strong topology of $L^2$.
    
    Now, we show the inverse implication i.e that if $F$ is l.s.c with respect to the strong topology of $L^2$, then $\F$ is l.s.c with respect to $W_2$. Let $(\rho_n)_n\in\mathcal P_2(\R^d)^\N$ and $\rho\in\mathcal P_2(\R^d)$ be such that $(\rho_n)_n$ is converging to $\rho$ in $W_2$. Then  $(\rho_n)_n$ is converging to $\rho$ for the narrow topology and the second moment converges i.e 
    \begin{equation}
    \label{eq:cv_moments}
    \int |x|^2\mathrm{d}\rho_n\xrightarrow[n\to+\infty]{}\int |x|^2\diff\rho,
    \end{equation}
    see~\cite{santambrogio2015optimal}. By the Skorokhod theorem, there exists a sequence $(X_n)_n$ and a limit point $X$ such that for all $n \in \N$, $X_n{_\#}\mathbb{P}=\rho_n$, $X{_\#}\mathbb{P=\rho}$ and $X_n$ converges almost everywhere to $X$. Moreover, the convergence of the moments~\eqref{eq:cv_moments} can be read as 
    \begin{equation*}
    \|X_n\|_2\xrightarrow[n\to+\infty]{}\|X\|_2.
    \end{equation*}
    Therefore, the Brezis-Lieb lemma implies that $(X_n)$ converges to $X$ in the strong topology of $L^2$. But $F$ is l.s.c with respect to the strong topology of $L^2$, so 
    \begin{equation*}
        \F(\rho) = F(X) \leq \liminf\limits_{n\to+\infty} F(X_n) = \liminf\limits_{n\to+\infty} \F(\rho_n),
    \end{equation*}
  and $\F$ is l.s.c with respect to $W_2$.
\end{proof}
Let $\mu\in\mathcal P_2(\R^d)$ with $\mu\ll\Leb$ and $X\in\mathcal H$ such that $X{_\#}\mathbb P=\mu$. We recall that one step of the JKO scheme $J^0_\tau(\mu)$ is well defined if and only if the functional $F_\tau^X$ defined by~\eqref{Prox map} has a unique minimizer, see Theorem \ref{equiv scheme}. Therefore, the questions of existence and uniqueness of $J^0_\tau(\mu)$ reduce to proving the existence of a unique minimizer of a strictly convex, lower semicontinuous functional on a Hilbert space. First, let us use strict convexity to guarantee coercivity of $F^X_\tau$.

\begin{proposition}\label{level set bounded}
   Let $X$ be such that $X{_\#}\mathbb P\ll\Leb$. If the restriction of $F$ to $C(X)$ is $\lambda$-convex and $\tau<\frac{1}{\lambda_-}$, then the sub-levels of the restriction of $F_\tau^X$ to $C(X)$ are bounded in $C(X)$.
\end{proposition}
\begin{proof}
    Let $M\in\R$. Let us show that $\{Y\in C(X) \mbox{ such that }F_\tau^X(Y)\leq M\}$ is bounded in $L^2$. Assume by contradiction that there exists $(Y_n)_n$ a sequence in this set such that $\|Y_n\| \rightarrow+\infty$. For a given $n \in \N^*$, the convexity inequality given by Proposition \ref{equiv cnx} leads for all $t \in [0,1]$ to:
    \begin{align}
    \notag F_\tau^X(tY_n+(1-t)Y_0)&\leq tF_\tau^X(Y_n)+(1-t)F_\tau^X(Y_0)-\frac{1+\lambda\tau}{2\tau}t(1-t)\|Y_0-Y_n\|^2\\
    \label{eq:convexity F}&\leq M -\frac{1+\lambda\tau}{2\tau}t(1-t)\|Y_0-Y_n\|^2.
    \end{align}
    Let $t_n:=\|Y_0-Y_n\|^{-\frac{3}{2}}$ and $Z_n:=t_nY_n+(1-t_n)Y_0$. First, $\|Y_0-Y_n\|\geq\|Y_n\|-\|Y_0\|\rightarrow+\infty$, and $t_n\to 0$. Therefore, 
    \begin{equation*}
    t_n(1-t_n)\|Y_0-Y_n\|^2=\|Y_0-Y_n\|^{\frac{1}{2}}(1-t_n)\xrightarrow[n\to+\infty]{}+\infty,
    \end{equation*}
    so that plugged in~\eqref{eq:convexity F}, we find
    \[
    F_\tau^X(Z_n)\xrightarrow[n\to+\infty]{}-\infty.
    \]
    Also, as $t_n \to 0$, $(Z_n)$ converges to $Y_0$ in the strong topology of $L^2$. Since $F_\tau^X$ is lower semicontinuous for the strong topology of $L^2$, we find:
\[
F_\tau^X(Y_0)\leq \liminf_{n \to + \infty} F_\tau^X (Z_n)=-\infty.
\]
This is absurd, and the claim follows.
\end{proof}
Finally, convexity allows to deduce weak lower semicontinuity out of strong lower semicontinuity.
\begin{proposition}
\label{prop:strong_to_weak}
    If $\F$ is $W_2$-l.s.c, $\lambda$-convex along generalized geodesics, and if $\tau<\frac{1}{\lambda_-}$, then for all $X\in\mathcal{ H}$ such that $X{_\#}\mathbb P\ll\Leb$, $F_\tau^{X}$ defined by \eqref{Prox map} is l.s.c on $C(X)$ for the weak topology of $L^2$.
\end{proposition}
\begin{proof}
 By Proposition~\ref{prop:equiv_lsc}, $F^X_\tau$ is l.s.c for the strong topology of $L^2$. But according to the assumptions of this proposition and Proposition~\ref{equiv cnx}, the restriction of $F^X_\tau$ is convex in the convex set $C(X)$. The result follows from \cite[corollary 3.9]{brezis2011functional}
\end{proof}
Now, we have all the requirements to attack the proof of Theorem \ref{welldefined JKO}.
\begin{proof}[Proof of Theorem \ref{welldefined JKO}]
    Let $X \in \mathcal H$. Because of Proposition \ref{equiv scheme}, we just need to show that $E_\tau(X)$ is well defined. By assumption, there exists a competitor $X_0$ such that $F(X_0)<+\infty$, so we can restrict the search for a competitor to the set 
    \begin{equation*}
    \left\{Y\in C(X) \mbox{ such that } \frac{\|X-Y\|^2_2}{2\tau}+F(Y)\leq \frac{\|X-X_0\|^2_2}{2\tau}+F(X_0)\right\}.
    \end{equation*}
    By Proposition \ref{level set bounded}, this set is bounded. Therefore, it is compact for the weak topology of $L^2$. But in virtue of Proposition~\ref{prop:strong_to_weak}, $F_\tau^X$ is l.s.c for the weak topology of $L^2$ on $C(X)$, and hence $F_\tau^X$ admits a minimizer. Moreover the strict convexity of $F_\tau^X$ implies uniqueness of this minimizer. So $E_\tau(X)$ is well defined and so is $J_\tau^0(X{_\#}\mathbb P)$.
\end{proof}

\subsubsection{Existence along the Entropic JKO scheme}

In this paragraph, we show that the entropic JKO scheme has minimizers. However, we will not be able to prove uniqueness in general, since there is no analogue of the discrete E.V.I inequality of Theorem~\ref{discrete E.V.I} in the entropic setting. A list of cases where we are able to prove uniqueness is given in Proposition \ref{uniqueness Ent JKO}. 
\begin{theorem}\label{existence Ent JKO}
    Let $\F$ be $\lambda$-convex along generalized geodesics, lower semicontinuous in the sense of Hypothesis \ref{hypothesis}, $\tau<\frac{1}{\lambda_-}$. We assume that there exists $\nu_0 \in \mathcal P_2(\R^d)$ with $\F(\nu_0)$ and $H(\nu_0)$ finite. Then, for all $\mu \in \mathcal P_2(\R^d)$ with finite entropy there exists a minimizer $J_\tau^\alpha(\mu) \in \mathcal P_2(\R^d)$ in~\eqref{Ent JKO}, and it has finite entropy.
\end{theorem}
An easy induction shows that:
\begin{corollary}
     If $\F$ is $\lambda$-convex along generalized geodesics, lower semicontinuous in the sense of Hypothesis \ref{hypothesis}, $\tau<\frac{1}{\lambda_-}$ and $\mu \in \mathcal P_2(\R^d)$ has finite entropy and satisfies $\F(\mu)<+\infty$, then there exists a sequence $(J_{k,\tau}^\alpha(\mu))_{k\geq0}$ satisfying the induction relation $J_{k,\tau}^\alpha(\mu)=J_\tau^\alpha(J_{k-1,\tau}^\alpha(\mu))$, for all $k \in \N^*$.
\end{corollary}
The main ingredient in the proof of Theorem~\ref{existence Ent JKO} is the following proposition.
\begin{proposition}\label{sous niv moment borne}
    If $\F$ is $\lambda$-convex along generalized geodesics and $\tau<\frac{1}{\lambda_-}$, then for every $\mu\in\mathcal{P}_2(\R^d)$ with finite entropy, the sublevels of $\frac{\Sch^{\alpha\tau}(\mu,\cdot)}{\tau}+\F(\cdot)$ have uniformly bounded second moment and entropy.
\end{proposition}
\begin{proof}
    Let us consider $\mu \in \mathcal P_2(\R^d)$ with $H(\mu)$ and $\F(\mu)$ finite, and let $M\in\R$. Let us show that the set
    \begin{equation*}
        \left\{\nu\in\mathcal P_2(\R^d) \mbox{ such that }\frac{\Sch^{\alpha\tau}(\mu,\nu)}{\tau}+\F(\nu)\leq M\right\}
    \end{equation*} 
    has uniformly bounded second moment and entropy. This will be done in three steps. First, we derive a bound from below for the entropy, then a bound for the second moment, and finally we deduce an upper bound for the entropy using Proposition~\ref{below bound ent}. So let us consider $\nu$ in the sublevel above. Proposition \ref{below bound sch by W-2}, implies:
    \begin{equation}
    \frac{\Sch^{\alpha\tau}(\mu,\nu)}{\tau}+\F(\nu)\geq \frac{W_2^2(\mu,\nu)}{2\tau}+\alpha\frac{H(\mu)+H(\nu)}{2}+\F(\nu).\label{eq:cout Ent JKO > cout JKO + ent}
    \end{equation}
    As $\F$ is $\lambda$-convex along generalized geodesics and $\tau<\frac{1}{\lambda_-}$, by Theorem \ref{welldefined JKO}, the minimizer $J^0_\tau(\mu)$ of the classical JKO scheme~\eqref{JKO} is well defined and:
    \[
    M\geq\frac{\Sch^{\alpha\tau}(\mu,\nu)}{\tau}+\F(\nu)\geq \min_\rho\left\{\frac{W_2^2(\mu,\rho)}{2\tau}+\F(\rho)\right\}+\alpha\frac{H(\nu)}{2}+\alpha\frac{H(\mu)}{2},
    \]
    which provides the following uniform upper bound for the entropy:
    \[
    H(\nu)\leq \frac{2}{\alpha}\left[M-\min_\rho\left\{\frac{W_2^2(\mu,\rho)}{2\tau}+\F(\rho)\right\}\right] - H(\mu).
    \]
   
     Now, let us derive a uniform bound for the second moment of $\nu$. As this second moment equals the distance $W_2^2(\delta_0,\nu)$ to $\delta_0$, the Dirac mass at $0$, by the triangle inequality, we just have to show a uniform bound for the distance from $\nu$ to some measure in $\mathcal P_2(\R^d)$, independent of $\nu$. We will bound the distance from $\nu$ to
    \[
    \mathcal{J}_\tau^\alpha(\mu):=\argmin_\rho \frac{W_2^2(\mu,\rho)}{2\tau}+\alpha\frac{H(\rho)}{2}+\F(\rho).
    \]
    This measure is well defined in virtue of Theorem~\ref{welldefined JKO}, because $H$ is convex along generalized geodesics (Proposition~\ref{classic functional light ver}), and so $\F + \frac{\alpha}{2} H$ is $\lambda$-convex along generalized geodesics. Also, we can apply the discrete E.V.I of Theorem \ref{discrete E.V.I} to $\F + \frac{\alpha}{2} H$ and obtain:
    \[(1+\lambda\tau)\frac{W_2^2(\mathcal{J}_\tau^\alpha(\mu),\nu)}{2\tau}\leq \frac{W_2^2(\mu,\nu)}{2\tau}+\F(\nu)+\frac{\alpha}{2}H(\nu)-\left(\frac{W_2^2(\mu,\mathcal{J}_\tau^\alpha(\mu))}{2\tau}+\F(\mathcal{J}_\tau^\alpha(\mu))+\frac{\alpha}{2}H(\mathcal{J}_\tau^\alpha(\mu))\right).
    \]
    Using the bound~\eqref{eq:cout Ent JKO > cout JKO + ent}, we obtain:
    \begin{align*}
     (1+\lambda\tau)\frac{W_2^2(\mathcal{J}_\tau^\alpha(\mu),\nu)}{2\tau}&\leq \frac{\Sch(\mu,\nu)}{\tau}+\F(\nu)-\frac{\alpha}{2}H(\mu)-\left(\frac{W_2^2(\mu,\mathcal{J}_\tau^\alpha(\mu))}{2\tau}+\F(\mathcal{J}_\tau^\alpha(\mu))+\frac{\alpha}{2}H(\mathcal{J}_\tau^\alpha(\mu))\right)\\
     &\leq M-\frac{\alpha}{2}H(\mu)-\left(\frac{W_2^2(\mu,\mathcal{J}_\tau^\alpha(\mu))}{2\tau}+\F(\mathcal{J}_\tau^\alpha(\mu))+\frac{\alpha}{2}H(\mathcal{J}_\tau^\alpha(\mu))\right),
    \end{align*}
    and the claim follows.
    
Finally, a uniform bound for the entropy is directly obtained using Proposition~\ref{below bound ent}.
    \end{proof}
Now, we can start the demonstration of Theorem \ref{existence Ent JKO}.
\begin{proof}[Proof of Theorem \ref{existence Ent JKO}]
Let $\nu_0 \in \mathcal P_2(\R^d)$ be such that both $\F(\nu_0)$ and $H(\nu_0)$ are finite. Then, as explained below Definition~\ref{def:sch cost}, $\frac{\Sch^{\alpha\tau}(\mu,\nu_0)}{\tau}+\F(\nu_0)<+\infty$, so we can consider a minimizing sequence $(\rho_n)_n$ such that for all $n \in \N$,
\begin{equation*}
\frac{\Sch^{\alpha\tau}(\mu,\rho_n)}{\tau}+\F(\rho_n)\leq\frac{\Sch^{\alpha\tau}(\mu,\nu_0)}{\tau}+\F(\nu_0).
\end{equation*}
Then, by Proposition \ref{sous niv moment borne}, the second moment and the entropy of the sequence $(\rho_n)_n$ are uniformly bounded. In particular, the sequence is tight, so we can extract a subsequence converging narrowly to some $\rho \in \mathcal{P}(\R^d)$. Because of the uniform bound of the second moments of $(\rho_n)$, and since $\F$ is l.s.c in sense of Hypothesis~\ref{hypothesis}, we have
\begin{equation*}
    \F(\rho)\leq\liminf_{n\to +\infty}\F(\rho_n).
\end{equation*}
Because of the uniform bounds on the second moments and entropy of $(\rho_n)$, and by the lower semicontinuity of $\Sch$ stated in~\cite[Lemma~2.4]{Carlier2017}, we also have 
\begin{equation*}
    \frac{\Sch^{\alpha\tau}(\mu,\rho)}{\tau}\leq\liminf\limits_{n\to +\infty}\frac{\Sch^{\alpha\tau}(\mu,\rho_n)}{\tau}.
\end{equation*}
All in all
\begin{equation*}
\frac{\Sch^{\alpha\tau}(\mu,\rho)}{\tau}+\F(\rho)\leq\liminf\limits_{n\to +\infty}\frac{\Sch^{\alpha\tau}(\mu,\rho_n)}{\tau}+\F(\rho_n),
\end{equation*}
and so $\rho$ is a minimizer.
\end{proof}
\subsection{Cases of uniqueness in the entropic JKO scheme}
\label{subsec:uniqueness}
Once again, uniqueness is not as straightforward as in the classical case. We can only prove it in the following cases:
\begin{proposition}\label{uniqueness Ent JKO}
   Let us assume that $\F$ is of the form \(\F(\rho)=\int V\rho+\frac{1}{2}\int W*\rho\rho+\int f(\rho)\) and that one of the next statements holds true: 
    \begin{itemize}
        \item \(\widehat{W}\geq 0\) and \(f\) is convex,  where $\widehat{W}$ is the Fourier transform of $W$;
        \item $V$ and $W$ are $\lambda$-convex and {$f=0$}. 
    \end{itemize}
   Then, there is at most one minimizer in~\eqref{Ent JKO}.
\end{proposition}
\begin{proof}
    In the first case, uniqueness is a direct consequence of the strict convexity of $\frac{\Sch^{\alpha\tau}}{\tau}(\mu,\cdot)+\F(\cdot)$ along interpolations of the form $t\mapsto t\nu_1+(1-t)\nu_0$. Let us prove separately the convexity and strict convexity of $\F$ and $\Sch^{\alpha \tau}(\mu, \cdot)$ respectively along these interpolations in the two following lemmas.
   
    \begin{lemma}\label{Lem: cnv F hor}
        Let $\mathcal{F}$ be of the form $ \F(\rho) = \int V\diff\rho +\int W*\rho\diff\rho+\int f(\rho(x))\diff x$, with $f$ convex and $\widehat{W}\geq0$. Let $\nu_0,\nu_1\in\mathcal{P}_2(\R^d)$, such that $\F(\nu_1)<+\infty$ and $\F(\nu_2)<+\infty$, then for all $t\in[0,1]$, 
        \begin{equation*}
        \F(t\nu_1+(1-t)\nu_0)\leq t\F(\nu_1)+(1-t)\F(\nu_0).
        \end{equation*}
    \end{lemma}
     \begin{lemma}\label{Lem: cnv Sch hor}
        Let $\mu\in\mathcal{P}_2(\R^d)$ be fixed. Consider $\nu_0,\nu_1\in\mathcal{P}_2(\R^d)$ such that $H(\nu_0)<+\infty$, $H(\nu_1)<+\infty$ and $\nu_0\neq\nu_1$. Then for all $t\in(0,1)$, $\Sch(\mu,t\nu_1+(1-t)\nu_0)< t\Sch(\mu,\nu_1)+(1-t)\Sch(\mu,\nu_0)$.
    \end{lemma}
    \begin{proof}[Proof of Lemma \ref{Lem: cnv F hor}]
    Recall that there are three parts in $\F$, as for all $t \in [0,1]$,
        \[
        \F(t\nu_1+(1-t)\nu_0)=\int V \diff(t\nu_1+(1-t)\nu_0)+\int W*(t\nu_1+(1-t)\nu_0)\diff(t\nu_1+(1-t)\nu_0)+\int f(t\nu_1+(1-t)\nu_0).
        \]
        The part concerning $V$ can be managed as follows
        \begin{equation}\label{eq: cnv V part hor}
        \int V \diff(t\nu_1+(1-t)\nu_0)=t\int V\diff \nu_1+(1-t)\int V\diff\nu_0.
        \end{equation}
        For the part involving $W$, given $\rho \in \mathcal P_2(\R^d)$, let us start by rewriting the functional using the Plancherel identity:
        \begin{equation}\label{eq: formule W fourier}
            \int W*\rho\diff \rho=\int\widehat{W*\rho} \Bar{\widehat{\rho}}=\int\widehat{W}|\widehat\rho|^2.
        \end{equation}
        Hence, for all $t \in [0,1]$,
        \[
        \int W*(t\nu_1+(1-t)\nu_0)\diff(t\nu_1+(1-t)\nu_0)=\int \widehat{W}|t\widehat\nu_1+(1-t)\widehat\nu_0|^2.
        \]
        As $|t\widehat\nu_1+(1-t)\widehat\nu_0|^2=t|\widehat\nu_1|^2+(1-t)|\widehat\nu_0|^2-t(1-t)|\widehat\nu_1-\widehat\nu_0|^2$, we get
        \[
        \int W*(t\nu_1+(1-t)\nu_0)\diff(t\nu_1+(1-t)\nu_0)=\int \left(t\widehat{W}|\widehat\nu_1|^2+(1-t) \widehat{W}|\widehat\nu_0|^2-t(1-t)\widehat{W}|\widehat\nu_1-\widehat\nu_0|^2\right).
        \]
       But $\widehat{W}\geq0$ and so $-t(1-t)\widehat{W}|\widehat\nu_1-\widehat\nu_0|^2\leq0$, so that
        \[
        \int W*(t\nu_1+(1-t)\nu_0)\diff(t\nu_1+(1-t)\nu_0)\leq\int \left(t\widehat{W}|\widehat\nu_1|^2+(1-t) \widehat{W}|\widehat\nu_0|^2\right).
        \]
        Finally, using equation~\eqref{eq: formule W fourier} backwards, we obtain:
        \begin{equation}\label{eq: cnv W part hor}
            \int W*(t\nu_1+(1-t)\nu_0)\diff(t\nu_1+(1-t)\nu_0)\leq t\int W*\nu_1\diff\nu_1+(1-t)\int W*\nu_0\diff\nu_0.
        \end{equation}
        Since $\F(\nu_0)<+\infty$ and $\F(\nu_1)<+\infty$, either $f=0$ and there nothing more to show or $\nu_0,\nu_1$ have densities $\nu_0(x),\nu_1(x)$. Then, by convexity of $f$,
        \[
        f(t\nu_1(x)+(1-t)\nu_0(x))\leq t f(\nu_1(x))+(1-t) f(\nu_0(x)),
        \]
        and hence
        \begin{equation}\label{eq: cnv f part hor}
        \int f(t\nu_1(x)+(1-t)\nu_0(x))\diff x\leq t\int f(\nu_1(x))\diff x+(1-t) \int f(\nu_0(x))\diff x.
        \end{equation}
        Adding equation~\eqref{eq: cnv V part hor}, equation~\eqref{eq: cnv W part hor} and equation~\eqref{eq: cnv f part hor}, we obtain the Lemma \ref{Lem: cnv F hor}.
    \end{proof}
   Let us proceed to the proof of the second lemma.
    \begin{proof}[Proof of Lemma \ref{Lem: cnv Sch hor}]
        By Definition \ref{def:sch cost} the Schrodinger cost can be expressed as:
        \[\Sch^{\alpha\tau}(\mu,\nu_i)=\min_{\gamma\in\Pi(\mu,\nu_i)} \left\{ \int \frac{|x-y|^2}{2}\diff\gamma(x,y)+\alpha \tau H(\gamma)+\frac{\alpha\tau d}{2} \ln(2\pi\alpha\tau)\right\}, \quad i = 0,1. \]
        Let $\gamma_i$, $i=0,1$ realize this minimum, and $t \in (0,1)$. Then, choosing $\gamma_t=t\gamma_1+(1-t)\gamma_0$ as a competitor for the Schrödinger cost between $\mu$ and $t\nu_1+(1-t)\nu_0$, we obtain:
        \begin{equation}\label{eq: cnv comp sch hor}
            \Sch(\mu,t\nu_1+(1-t)\nu_0)\leq  \int \frac{|x-y|^2}{2}\diff\gamma_t(x,y)+\alpha \tau H(\gamma_t)+\frac{\alpha\tau d}{2} \ln(2\pi\alpha\tau).
        \end{equation}
        The first term verifies:
        \begin{equation}\label{eq: cnv cout hor}
        \int \frac{|x-y|^2}{2}\diff\gamma_1(x,y)=t\int \frac{|x-y|^2}{2}\diff\gamma_1(x,y)+(1-t)\int \frac{|x-y|^2}{2}\diff\gamma_0(x,y),
        \end{equation}
        and since the function $h:s\mapsto s\ln(s)$ is strictly convex, then for every $(x,y) \in (\R^d)^2$ such that $\gamma_0(x,y)\neq\gamma_1(x,y)$, we obtain:
        \begin{equation}\label{eq: str cnv xlnx}
        \gamma_t(x,y)\ln(\gamma_t(x,y))<t\gamma_1(x,y)\ln(\gamma_1(x,y))+(1-t)\gamma_0(x,y)\ln(\gamma_0(x,y)),
        \end{equation}
        while for every $(x,y)$ such that $\gamma_0(x,y)=\gamma_1(x,y)$, we have:
        \[
        \gamma_t(x,y)\ln(\gamma_t(x,y))\leq t\gamma_1(x,y)\ln(\gamma_1(x,y))+(1-t)\gamma_0(x,y)\ln(\gamma_0(x,y)).
        \]
        If we assume by contradiction that $\gamma_0(x,y)=\gamma_1(x,y)$ almost everywhere, then $\pi_2{_\#}\gamma_0=\pi_2{_\#}\gamma_1$, so $\nu_0=\nu_1$ and we have excluded this case. Hence, there is a non negligible set such that the previous inequality~\eqref{eq: str cnv xlnx} holds, and then, by integration,
         \begin{equation}\label{eq: cnv H hor}
        H(\gamma_t)=\int\ln(\gamma_t)\diff\gamma_t< t\int \ln(\gamma_1)\diff\gamma_1+(1-t)\int\ln(\gamma_0)\diff\gamma_0
        = t H(\gamma_1)+(1-t) H(\gamma_0).
        \end{equation}
        The result follows from plugging \eqref{eq: cnv cout hor} and \eqref{eq: cnv H hor} into \eqref{eq: cnv comp sch hor}.
    \end{proof}
    Uniqueness in the second case is also a consequence of the convexity of $\frac{\Sch^{\alpha\tau}}{\tau}(\mu,\cdot)+\F(\cdot)$ along a well-chosen interpolation. We construct the interpolation as follows; let $\nu_0$ and $\nu_1$ be two measures, consider the Schrodinger plans $\gamma_0,\gamma_1$  from $\mu$ to $\nu_0$ and from $\mu$ to $\nu_1$ respectively. Now, let us disintegrate our plans with respect to their first marginal, in order to get two collections of measures defined for $\mu$-almost every $x$, $(\nu_0^x)_x$ and $(\nu_1^x)_x$. For a fixed $x$ such that $\nu_0^x$ and $\nu_1^x$ are defined, let us consider the optimal transport plan between these two measures, and call it $\gamma^x$. The interpolation between $\nu_0$ and $\nu_1$ that we will consider is defined for all $t \in [0,1]$ by $\nu_t:=\left(t\pi_3+(1-t)\pi_2\right){_\#}\Pi$, where $\Pi$ is a the three plan  of marginals $\mu$, $\nu_0$ and $\nu_1$ defined by $\Pi=\mu\otimes\gamma^x$, that is, the unique measure such that for all $\varphi\in C_c(\R^{3d})$, 
    \begin{equation*}
    \int \varphi(x,y,z)\diff \Pi(x,y,z)=\int\varphi(x,y,z)\diff\gamma^x(y,z)\diff\mu(x).
    \end{equation*}
    Since $f=0$, $\F$ is convex along this interpolation. This is a consequence of~\cite[Proposition 9.3.2, Proposition 9.3.5]{ambrosio2005gradient} whose proof we have reproduced to prove the second point of Proposition~\ref{Prop:proof classic verify} for the part of the functional concerning $W$. Thus, it is enough to show that $t \mapsto\frac{\Sch^{\alpha\tau}(\mu,\nu_t)}{\tau}$ is strictly convex. This is the purpose of the following lemma, which easily concludes the proof. 
    \end{proof}
    \begin{lemma} With the notations of the proof of Proposition~\ref{uniqueness Ent JKO}, for all $t \in [0,1]$, we have:
        \[
        \Sch^{\alpha\tau}(\mu,\nu_t)\leq t\Sch^{\alpha\tau}(\mu,\nu_0)+(1-t)\Sch^{\alpha\tau}(\mu,\nu_1)-t(1-t)\int\frac{|x-y|^2}{2}\diff (\pi_2,\pi_3){_\#}\Pi.
        \]
    \end{lemma}
    \begin{proof}
        Let $t \in [0,1]$. Using $(\pi_1,t\pi_3+(1-t)\pi_2){_\#}\Pi$ as competitor in Definition~\ref{def:sch cost} of the Schr\"odinger cost, we find
\begin{multline}
\label{eq: maj sch lat}
\Sch^{\alpha\tau}(\mu,\nu_t)\leq  \int \frac{|x-y|^2}{2}\,\diff(\pi_1,t\pi_3+(1-t)\pi_2){_\#}\Pi(x,y)\\
+\alpha \tau\,H((\pi_1,t\pi_3+(1-t)\pi_2){_\#}\Pi)+\frac{\alpha\tau d}{2} \ln(2\pi\alpha\tau).
\end{multline}
Concerning the distance part, we have
        \begin{align}
                 \notag \int &\frac{|x-y|^2}{2}\diff(\pi_1,t\pi_3+(1-t)\pi_2){_\#}\Pi(x,y)=\int \frac{|x-(tz+(1-t)y)|^2}{2}\diff \Pi(x,y,z)\\
                 \notag &=t\int \frac{|x-z|^2}{2}\diff\Pi(x,y,z)+(1-t)\int \frac{|x-y|^2}{2}\diff\Pi(x,y,z)-t(1-t)\int \frac{|z-y|^2}{2}\diff\Pi(x,y,z)\\
                 &=t\int \frac{|x-z|^2}{2}\diff\gamma_1(x,z)+(1-t)\int \frac{|x-y|^2}{2}\diff\gamma_0(x,y)-t(1-t)\int \frac{|z-y|^2}{2}\diff(\pi_2,\pi_3){_\#}\Pi(y,z).  \label{eq: cnv cout lat}
        \end{align}
        Now let us take care about the entropy term. First, let us remind that $\Pi=\mu\otimes\gamma^x$ where $\gamma^x$ is the optimal transport plan between $\nu_0^x$ and $\nu_1^x$, so using Proposition~\ref{prop:addivity of entropy}, we find:
        \[
        H((\pi_1,t\pi_3+(1-t)\pi_2){_\#}\Pi)=H(\pi_1{_\#}\Pi)+\int H((t\pi_3+(1-t)\pi_2){_\#}\gamma^x)\diff\mu(x).
        \]
        But since for $\mu$-almost all $x$, $\gamma^x$ is an optimal transort plan, $t\mapsto(t\pi_3+(1-t)\pi_2){_\#}\gamma^x$ is a Wasserstein geodesic. As the entropy is convex along Wasserstein geodesics (for example, it verifies the McCann criterion, see Proposition \ref{classic functional light ver}), we have
        \[
        H((t\pi_3+(1-t)\pi_2){_\#}\gamma^x)\leq t H(\pi_3{_\#}\gamma^x)+(1-t)H(\pi_2{_\#}\gamma^x)=tH(\nu_1^x)+(1-t)H(\nu_0^x).
        \]
        Hence,
        \[
        H((\pi_1,t\pi_3+(1-t)\pi_2){_\#}\Pi)\leq t\left(H(\mu)+\int H(\nu_1^x))\diff\mu(x)\right)+(1-t)\left(H(\mu)+\int H(\nu_0^x)\diff\mu(x)\right).
        \]
        Using once again the additivity of the entropy, we end up with:
        \begin{equation}\label{eq: cnv H plan lat}
        H((\pi_1,t\pi_3+(1-t)\pi_2){_\#}\Pi)\leq tH(\gamma_1)+(1-t)H(\gamma_0).
        \end{equation}
       The result follows from gathering formulas~\eqref{eq: maj sch lat}, \eqref{eq: cnv cout lat} and \eqref{eq: cnv H plan lat}.
    \end{proof}

\subsection{Validity of Hypothesis \ref{hypothesis} in usual cases}\label{proof classic verify}
Hypothesis~\ref{hypothesis} covers a large variety of functionals~$\F$ that are commonly used in the literature. In particular, the following proposition holds. The reader can find a lighter version of this proposition in Proposition~\ref{classic functional light ver}.
\begin{proposition}\label{Prop:proof classic verify}
Let $\mathcal F$ be of the form~\eqref{eq:explicit_F}, that is,
\begin{equation*}
\mathcal F : \rho \in \mathcal P_2(\R^d) \longmapsto \int_{\mathbb{R}^d} V(x)\, \rho(x)\diff x 
+ \frac{1}{2} \int_{\mathbb{R}^d} (W * \rho)(x)\, \rho(x)\diff x 
+ \int_{\mathbb{R}^d} f(\rho(x))\diff x,
\end{equation*}
for some functions $V,W\in C^0(\R^d,\R)$ and $f\in C^0(\R_+,\R)$, where $\mathcal F$ is set to $+\infty$ if $\rho$ is not absolutely continuous with respect to the Lebesgue measure.
\begin{enumerate}
    \item  Let us assume that:
    \begin{itemize}
     \item  $\frac{V_-(x)}{|x|^2}\xrightarrow[|x|\rightarrow +\infty]{}0,$
     \item $\frac{W_-(x)}{|x|^2}\xrightarrow[|x|\rightarrow +\infty]{}0,$
     \item $f$ is convex and superlinear and there exist $q>\frac{d}{d+2}$ and two positive constants $c_1,c_2$ such that $f(0)=0$ and for all $ s \in [0,+\infty) \quad f_{-}(s) \leq c_{1}s + c_{2}s^{q} $.
 \end{itemize}
 Then $\F$ is well defined and l.s.c in the sense of Hypothesis \ref{hypothesis}. In other terms, it verifies the first point of Hypothesis~\ref{hypothesis}.
    \item Let $\F$ satisfy (1), and let us further assume that:
    \begin{itemize}
        \item $V$ is $\lambda_1$-convex,
        \item $W$ is symmetric and $\lambda_2$-convex, with $\lambda_2\leq0$,
        \item $s \mapsto s^df(s^{-d})$ is convex and non increasing on $(0,+\infty)$.
    \end{itemize}
    Then $\mathcal{F}$ is $\lambda$-convex along generalized geodesics for $\lambda=\lambda_1 + \lambda_2$. In other terms, it verifies the second point of Hypothesis~\ref{hypothesis}.
    \item Let $\F$ satisfy (1), and let us further assume:
     \begin{itemize}
        \item $V$ is $\lambda_1$-convex and $\Delta V \leq K_1$ in the distributional sense,
        \item W is $\lambda_2$-convex, symmetric and $\Delta W\leq K_2$ in the distributional sense,
        \item $f$ is convex.
    \end{itemize}
    Then for all $\mu \in \mathcal P_2(\R^d)$ and $t \geq 0$, $\F(\mu*\sigma_{t})-\F(\mu)\leq K\frac{t}{2}$ with $K:=K_1+K_2$. In other terms, it verifies the third point of Hypothesis~\ref{hypothesis}.
\end{enumerate}
In particular, if $\F$ verifies the point $(1),(2)$ and $(3)$ right above, then $\mathcal{F}$ satisfies Hypothesis \ref{hypothesis} for $\lambda:=\lambda_1+ \lambda_2$ and $K:=K_1+K_2$.
\end{proposition}
\begin{remark}
\label{rem:C11}
    If $V$ is $\lambda_1$-convex and $\Delta V \leq K_1$, then $V$ has to be $C^1(\R^d,\R)$ with globally Lipchitz derivatives. The same holds for $W$.
\end{remark}

We will do the proof of each point separately.
\begin{proof}[Proof of Proposition~\ref{Prop:proof classic verify} point $(1)$]
    The proof is based on the fact that each functional: $\rho\mapsto\int V\diff\rho$, $\rho\mapsto\frac{1}{2}\int W*\rho\diff\rho$ and $\rho\mapsto\int f(\rho)$ are lower semicontinuous in sense of Hypothesis~\ref{hypothesis}. For the part $\rho\mapsto\int f(\rho)$, taking $\varepsilon$ small enough such that $q>\frac{d}{d+(2-\varepsilon)}$ then following \cite[Remark 9.3.8]{ambrosio2005gradient}, we obtain that $\rho\mapsto\int f(\rho)$ is $W_{2-\varepsilon}$- l.s.c which is stronger than the lower semicontinuity in the sense of Hypothesis~\ref{hypothesis}. Here, we will only do the proof for the functional associated to $V$ and $W$ in the two next lemmas for which the semicontinuity in the sense of Hypothesis~\ref{hypothesis} is not standard.
    \end{proof}
    \begin{lemma}\label{lower semicontinuity V}
        If $\frac{V_-(x)}{|x|^2}\xrightarrow[|x|\to+\infty]{}0$,
        then $\rho\mapsto\int V\diff\rho$ is l.s.c in the sense of Hypothesis \ref{hypothesis}.
    \end{lemma}
    \begin{proof}
        Consider $(\rho_n)_n\in\mathcal{P}_2(\R^d)$ such that $(\rho_n)_n$ has uniformly bounded second moment and converges for the narrow topology to $\rho$. We introduce $S:=\sup_n \int|x|^2\diff\rho_n(x)$. Let us quickly treat the positive part which is well known (see \cite{ambrosio2005gradient} for instance), and does not require any assumption on the moments of $(\rho_n)_n$. Let $M\in\R_+$, and $V_+^M: x \mapsto \max\{V_+(x),M\}$. We have 
        \begin{equation*}
            \liminf_{n \to + \infty} \int V_+ \diff \rho_n \geq \lim_{n \to + \infty} \int V_+^M \diff \rho_n = \int V_+^M \diff \rho,
        \end{equation*}
        where the second equality follows from the fact that $V_+^M$ is continuous and bounded. We get the desired lower semicontinuity by letting $M$ tend to $+ \infty$ and using the monotone convergence theorem.        
        
        Now, we have to prove that
        \begin{equation*}
            \limsup_{n \to + \infty} \int V_- \diff \rho_n \leq \int V_- \diff \rho.
        \end{equation*}
        Let $\chi_R \in C_c(\R^d)$ be a function taking values in $[0,1]$, uniformly equal to $1$ on the ball of center $0$ and radius $R$, and of support inside the ball of center $0$ and radius $R+1$. For all $R \in \R_+^*$ and $n \in \N$,
        \begin{align*}
            \int V_- \diff \rho_n(x) &= \int V_- \chi_R \diff \rho_n + \int  V_- (1 - \chi_R) \diff \rho_n \\
            & \leq \int V_-  \chi_R \diff \rho_n + \int |x|^2 \diff \rho_n(x) \sup_{|x| \geq R} \frac{V_-(x)}{|x|^2}.
        \end{align*}
        In the last line, the first term converges because $V_- \chi_R$ is continuous with compact support, and its limit is smaller or equal to $\int V_- \diff \rho$. By assumption, the second term converges to $0$ uniformly in $n$ as $R \to + \infty$. The result follows easily.
        \end{proof}

\begin{lemma}
    If $\frac{W_-(x)}{|x|^2}\xrightarrow[|x|\to+\infty]{}0$ then $\rho\mapsto\int W*\rho\diff\rho$ is l.s.c in the sense of Hypothesis \ref{hypothesis}.
\end{lemma}
\begin{proof}
    Consider $(\rho_n)_n\in\mathcal{P}_2(\R^d)$ such that $(\rho_n)_n$ have uniformly bounded second moment and converges for the narrow topology to $\rho$. As in the previous Lemma \ref{lower semicontinuity V}, the proof for the positive part is already known in the literature, see for instance \cite{ambrosio2005gradient}, and does not require any assumption on the second moment. 
     Consider once again $\chi_R \in C_c(\R^d)$ a function taking values in $[0,1]$, uniformly equal to $1$ on the ball of center $0$ and radius $R$, and of support inside the ball of center $0$ and radius $R+1$.

Let $W_+^R:=W_+\chi^R$. This function is continuous and compactly supported, hence, uniformly continuous. Therefore, $(W_+^R * \rho_n)$ is uniformly equicontinuous, hence relatively compact for the topology of locally uniform convergence thanks to Ascoli's Theorem, and its limit clearly appears to be $W_+^R * \rho$. Finally, as $(\rho_n)_{n \in \N}$ is tight, $W_+^R * \rho_n(x)$ converges to $0$ as $x \to + \infty$, uniformly in $n \in \N$. So the locally uniform convergence is actually a uniform convergence.

 Moreover, $(\rho_n)_n$ is converging for the narrow topology, which is the weak-* topology on $\mathcal P(\R^d)$, seen as a subset of the dual of continuous and bounded functions endowed with the topology of uniform convergence. It follows that,
\begin{equation}\label{eq:sci W conv tronc}
\lim_{n\to+\infty}\int W_+^R*\rho_n\diff\rho_n=\int W_+^R*\rho\diff\rho.
\end{equation}
   Therefore,
    \[
    \int W_+^R*\rho\diff\rho=\lim_{n\to+\infty}\int W_+^R*\rho_n\diff\rho_n\leq \liminf_{n\to+\infty}\int W_+*\rho_n\diff\rho_n,
    \]
    and we get the result by letting $R$ tend to $ + \infty$ on the left hand side.
    
    Let us now treat the negative part. We need to show that 
    \begin{equation*}
        \limsup_{n \to + \infty} \int W_- * \rho_n \diff \rho_n \leq \int W_- * \rho \diff \rho.
    \end{equation*}
    For all $R>0$, let us define $W_-^R=W_-\chi^R$, where $\chi^R$ is the truncation function defined in the first part of the proof. Up to replacing $W_+$ by $W_-$, the proof made to obtain equation~\eqref{eq:sci W conv tronc} provides
\begin{equation*}
\lim_{n\to+\infty}\int W_-^R*\rho_n\diff\rho_n=\int W_-^R*\rho\diff\rho \leq \int W_- * \rho \diff \rho.
\end{equation*}
To conclude, it remains to show that for all $\varepsilon >0$, there exists $R>0$ such that for all $n \in \N$,
\begin{equation*}
    \int W_- * \rho_n \diff \rho_n \leq  \int W_-^R * \rho_n \diff \rho_n + \varepsilon.
\end{equation*}
Let $\varepsilon>0$. By assumption, there exists $R>0$ such that for all $x$, if $|x|>R$, then $W_-(x)\leq\varepsilon |x|^2$. We have
\begin{align*}
\int (W_--W_-^R)*\rho_n\diff\rho_n&=\int_{x,y:|x-y|>R}(W_-(x-y)-W_-^R(x-y))\diff\rho_n(x)\diff\rho_n(y) \\
&\leq \varepsilon\int_{x,y:|x-y|>R}|x-y|^2\diff\rho_n(x)\diff\rho_n(y) \\
&\leq 2\varepsilon \int\left( |x|^2+|y|^2\right)\diff\rho_n(x)\diff\rho_n(y).
\end{align*}
Calling $S:=\sup\limits_n\int |x|^2\diff\rho_n(x)$, which is finite by assumption, we obtain:
\begin{equation*}
  \int (W_--W_-^R)*\rho_n\diff\rho_n \leq 4\varepsilon S,
\end{equation*}
and the result follows replacing $\varepsilon$ by $4\varepsilon S$ in our claim.
\end{proof}

\begin{proof}[Proof of Proposition~\ref{Prop:proof classic verify} point (2)]
For $V$, see \cite[Proposition 9.3.2]{ambrosio2005gradient}, for $f$ see \cite[Proposition 9.3.9]{ambrosio2005gradient}, for $W$ and $\lambda=0$ see \cite[Proposition 9.3.5]{ambrosio2005gradient}. The only case left is for $W$ and $\lambda<0$.
    Consider $\gamma\in\mathcal{P}_2((\R^d)^2$ and for $t\in[0,1]$, let $\rho_t=(t\pi_2+(1-t)\pi_1){_\#}\gamma$, where $\pi_i$ are the canonical projection. We start by writing the formula in term of $\gamma$
    \[
        \frac12\int W(z-\tilde z) \diff \rho_t(z)\diff\rho_t(\tilde z)=\frac12\int W(t(y-\tilde y)+(1-t)(x-\tilde x))\diff \gamma(x,y)\diff\gamma(\tilde x,\tilde y).
    \]  
Using the $\lambda$-convexity of $W$, we obtain
\begin{align*}
\frac{1}{2} \int W(z-\tilde{z}) \, \mathrm{d}\rho_t(z)\, \mathrm{d}\rho_t(\tilde{z})
\le t\,\frac{1}{2}\int W(y-\tilde{y}) \, \mathrm{d}\gamma(x,y)\, \mathrm{d}\gamma(\tilde{x},\tilde{y})
+(1-t)\,\frac{1}{2}\int W(x-\tilde{x}) \, \mathrm{d}\gamma(x,y)\, \mathrm{d}\gamma(\tilde{x},\tilde{y}) \\
\quad - t(1-t)\,\frac{\lambda}{2}\int \frac{\lvert y-\tilde{y}-x+\tilde{x}\rvert^{2}}{2}\,
\mathrm{d}\gamma(x,y)\, \mathrm{d}\gamma(\tilde{x},\tilde{y}) .
\end{align*}
But
    \begin{align*}       
    \int\frac{|y-\tilde y-x+\tilde x|^2}{2}&\diff \gamma(x,y)\diff\gamma(\tilde x,\tilde y)\\
    &=\int\frac{|y-x|^2}{2}\diff \gamma(x,y) +\int\frac{|\tilde y-\tilde x|^2}{2}\diff \gamma(\tilde x,\tilde y)-2\int(y-x)\cdot(\tilde y-\tilde x)\diff \gamma(x,y)\diff\gamma(\tilde x,\tilde y)
    \\&=2\int\frac{|y-x|^2}{2}\diff \gamma(x,y)-2\left(\int (y-x) \diff \gamma(x,y)\right)^2\\
    &\leq 2\int\frac{|y-x|^2}{2}\diff \gamma(x,y).
    \end{align*}
    Since $\lambda<0$,
  \begin{align*}
\frac{1}{2} \int W(z-\tilde{z}) \, \mathrm{d}\rho_t(z)\, \mathrm{d}\rho_t(\tilde{z})
\le t\,\frac{1}{2}\int W(y-\tilde{y}) \, \mathrm{d}\gamma(x,y)\, \mathrm{d}\gamma(\tilde{x},\tilde{y})
+(1-t)\,\frac{1}{2}\int W(x-\tilde{x}) \, \mathrm{d}\gamma(x,y)\, \mathrm{d}\gamma(\tilde{x},\tilde{y}) \\
\quad - t(1-t)\,\lambda \int \frac{\lvert y-x\rvert^{2}}{2}\, \mathrm{d}\gamma(x,y) .
\end{align*}
    which concludes the proof.
    \end{proof}
\begin{proof}[Proof of Proposition~\ref{Prop:proof classic verify} point $(3)$]
In the whole proof, we fix $\mu \in \mathcal P_2(\R^d)$, and for all $t \geq 0$, we define $\rho_t := \mu * \sigma_t$. We will use the following classical property of the heat flow: for all $t\geq 0$,
\begin{equation}
\label{eq:second_moment_rhot}
    \frac{1}{2}\int |x|^2 \diff \rho_t \leq \frac{1}{2}\int |x|^2 \diff \mu + dt.
\end{equation}
First, let us show that if $V$ is $\lambda_1$-convex and $\Delta V \leq K_1$, then for all $t \geq 0$,
\begin{equation}
\label{eq:bound_V_heat_flow}
    \int V \diff \rho_t \leq \int V \diff \mu + \frac12K_1 t. 
\end{equation}
First, if $\bar V \in C^\infty_c(\R^d)$, then $t \mapsto \int \bar V \rho_t$ is clearly continuous, and its distributional derivative is $ t \mapsto
    \frac{1}{2} \int \Delta \bar V \rho_t$. Therefore, 
\begin{equation}
\label{eq:ineg_approx}
    \int \bar V \diff \rho_t = \int \bar V \diff \mu + \frac{1}{2} \int_0^t \hspace{-5pt} \int \Delta \bar V \diff \rho_t \diff t.
\end{equation}
We need to replace $\bar V \in C^{\infty}_c(\R^d)$ in~\eqref{eq:ineg_approx} by the potential $V$ given in the statement of Proposition~\ref{Prop:proof classic verify}. Notice that by Remark~\ref{rem:C11}, $V$ is necessarily in $C^{1,1}$, and so it grows at most quadratically and its gradient grows at most linearly. In other words, there exists $C>0$ such that for all $x \in \R^d$, 
\begin{equation}
\label{eq:growth_V}
|V(x)| \leq C(1 + |x|^2) \qquad \mbox{and} \qquad |\nabla V(x)| \leq C(1 + |x|).
\end{equation}
By convolution, we can replace the condition $\bar V \in C_c^\infty(\R^d)$ in~\eqref{eq:ineg_approx} by $\bar V \in C^{1,1}_c(\R^d)$, and we just need to relax the fact that $\bar V$ has compact support. 

Given $R>0$, let $\chi_R := \chi(\cdot / R)$, where $\chi$ is a smooth function with value in $[0,1]$, uniformly equal to~$1$ in the ball of center $0$ and radius $1$, and with support in the ball of center $0$ and radius $2$. We will apply~\eqref{eq:ineg_approx} to $\bar V_R:= V \chi_R$. For all $R\geq 1$ and $x \in \R^d$, 
\begin{align*}
\left|\Delta V_R(x) - (\Delta V(x) \chi_R(x)) \right| &= \left| \nabla V(x) \cdot \frac{\nabla \chi(x/R)}{R}\right| + \left|  V(x) \frac{\Delta \chi(x/R)}{R^2}\right| \\
&\leq C 1_{R\leq |x| \leq 2R} \left( \| \nabla \chi \|_\infty  \frac{1 + 2R}{R} + \| \Delta \chi \|_\infty\frac{1 + 4R^2}{R^2} \right) \\
&\leq A 1_{R\leq |x|},
\end{align*}
where $A$ is chosen sufficiently large, and where to get the second line, we used~\eqref{eq:growth_V} and the fact that the supports of $\nabla \chi_R$ and $\Delta \chi_R$ are included in the annulus of center $0$ and radiuses $R$ and $2R$. Therefore, writing \eqref{eq:ineg_approx} to $\bar V_R$, we find
\begin{equation*}
    \int  V \chi_R \diff \rho_t \leq \int V \chi_R \diff \mu + \frac{1}{2} \int_0^t \int(\Delta V)\chi_R \diff \rho_s \diff s + A \int_0^t \rho_s(\{ x \in \R^d \mbox{ s.t.\ }|x| \geq R \} ) \diff s.
\end{equation*}
Letting $R$ tend to $+\infty$ with the help of~\eqref{eq:second_moment_rhot},~\eqref{eq:growth_V} and the Markov inequality, we deduce
\begin{equation*}
    \int  V  \diff \rho_t \leq \int V  \diff \mu + \frac{1}{2} \int_0^t \int (\Delta V)\diff \rho_s \diff s,
\end{equation*}
and~\eqref{eq:bound_V_heat_flow} follows bounding $\Delta V$ by $K_1$ from above.

Now, let us proceed to the proof of the bound for the $W$ part.
\begin{align*}
\frac{\diff }{\diff t}\frac{1}{2}\int W*\rho_t\diff\rho_t&=\left.\frac{\diff }{\diff s}\frac{1}{2}\int W*\rho_t\diff\rho_s\right|_{s=t}+\left.\frac{\diff }{\diff t}\frac{1}{2}\int W*\rho_s\diff\rho_t\right|_{s=t}\\
&=\left.\frac{\diff }{\diff s}\int W*\rho_t\diff\rho_s\right|_{s=t}.
\end{align*}
But if $W$ is $\lambda_2$-convex, then so does $W*\rho_t$, and if $\Delta W\leq K_2$, then $\Delta(W*\rho_t)\leq K_2$. So applying the the previous computation for $V=W*\rho_t$ we obtain 
\begin{equation}\label{eq: borne heat W}
    \frac{\diff}{\diff t}\frac{1}{2}\int W*\rho_t\diff\rho_t\leq \frac{1}{2} K_2.
\end{equation}

The only part remaining in the one on $f$. By the Jensen inequality, for all $t>0$ and $x \in \R^d$,
\[
f(\rho_t(x))=f\left(\int\mu(x-y)\diff\sigma_t(y))\right)\leq\int f(\mu(x-y))\diff\sigma_t(y),
\]
where the last integral is well defined in $\R \cup\{ + \infty \}$ thanks to the hypothesis made on $f_-$, which are made to ensure that for all absolutely continuous measure $\mu\in\mathcal P_2(\R^d)$, $\int f_-(\mu)<+\infty$ see~\cite{ambrosio2005gradient}. Integrating this inequality, using the Fubini theorem for the negative part and the Fubini-Tonelli theorem for the positive part, and then making the change of variable $x=x-y$ in the second integral, we obtain
\begin{equation}\label{eq: borne heat f}
    \int f(\rho_t)\leq\int f(\mu).
\end{equation}
The result follows from equations~\eqref{eq:bound_V_heat_flow}, \eqref{eq: borne heat W} and \eqref{eq: borne heat f}.
\end{proof}

\section{Proof of the Main Result}\label{proof of theorem}
The purpose of this Section is to prove Theorem \ref{the theorem}. We first provide a sketch of the proof.
\subsection{Sketch of proof}
The proof proceeds iteratively, i.e.\ for all $k\geq 0$, by comparing the distance at stage $k+1$, that is $W_2^2(J_{k+1,\tau}^0(\mu_0), J_{k+1,\tau}^\alpha(\mu_0))$, with the one of stage $k$, $W_2^2(J_{k,\tau}^0(\mu_0), J_{k,\tau}^\alpha(\mu_0))$. Rewriting $W_2^2(J_{k+1,\tau}^0(\mu_0), J_{k+1,\tau}^\alpha(\mu_0))$ as $W_2^2(J_\tau^0(J_{k,\tau}^0(\mu_0)), J_\tau^\alpha(J_{k,\tau}^\alpha(\mu_0)))$, we need to compare one increment of two different schemes starting from two different measures. Our strategy is to use the following decomposition to treat the facts that the starting measures are different and that the schemes are different separately:
    \[
W_2^2(J_{k+1,\tau}^0(\mu_0), J_{k+1,\tau}^\alpha(\mu_0)) \leq 
\left(\underbrace{W_2(J_\tau^0(J_{k,\tau}^0(\mu_0)), J_\tau^0(J_{k,\tau}^\alpha(\mu_0)))}_{\text{(I)}}
+ \underbrace{W_2(J_\tau^0(J_{k,\tau}^\alpha(\mu_0)), J_\tau^\alpha(J_{k,\tau}^\alpha(\mu_0)))}_{\text{(II)}}\right)^2.
\]
The term (I) is the distance between two iterates of the classic JKO scheme starting from different measures, and the term (II) is the distance between two increments of different schemes starting from the same measure. Notably, the first one is already estimated in \cite{ambrosio2005gradient}, see Theorem~\ref{contraction AGS}. Our main contribution is an estimate of the second part. Here we will see the entropic JKO scheme as a perturbation of the classic JKO scheme, thus reframing our question as a stability question: \textit{Why does this perturbation yield to a close solution?} In fact, the stability of the JKO scheme is contained in the discrete E.V.I. Indeed, under our $\lambda$-convexity assumption, Theorem~\ref{discrete E.V.I} implies for all $\mu \in \mathcal P_2(\R^d)$ admissible for both schemes and $\tau>0$:
   \[
\frac{1}{2\tau}W_2^2(J_\tau^0(\mu), J_\tau^\alpha(\mu)) \leq \frac{1}{1 + \lambda \tau} \left( \frac{W_2^2(\mu, J_\tau^\alpha(\mu))}{2\tau} + \mathcal{F}(J_\tau^\alpha(\mu)) - \left( \frac{W_2^2(\mu, J_\tau^0(\mu))}{2\tau} + \mathcal{F}(J_\tau^0(\mu)) \right) \right)
\]
Hence, to show that $J_\tau^\alpha(\mu)$ is close to $J_\tau^0(\mu)$, it suffices to show that it is a good competitor for the problem of which $J_\tau^0(\mu)$ is a minimizer. Since the Schrödinger cost is a perturbation of the Wasserstein distance, standard inequalities allow to replace the Wasserstein distance with the Schrödinger cost up to error terms. Up to this change, estimating the distance between $J_\tau^0(\mu)$ and $J_\tau^\alpha(\mu)$ reduces to estimating the difference between the optimal values of the classic and entropic problems. We estimate this difference by constructing a good competitor for the entropic JKO scheme by perturbing the minimizer of the classic JKO scheme. In order to do so, we will follow the heuristic idea that, for short times, the flow of \(\mathcal{F} + \frac{\alpha}{2}H\) can be obtained by following the flow of \(\mathcal{F}\) and then following the flow of \(\frac{\alpha}{2}H\). In other words, we will take as a competitor \(J_\tau(\mu) * \sigma_{\alpha\tau}\) and obtain a sharp bound.

Let us now enter the details of the proof.
\subsection{Beginning of the proof}

As already said, with the notations of the statement of the theorem, given $k \in \N$, we aim at estimating
\[
W_2^2(J_{k+1,\tau}^0(\mu_0), J_{k+1,\tau}^\alpha(\mu_0)).
\]
Because of the triangle inequality,
\begin{equation}\label{eq:splitting (I) and (II)}
W_2^2(J_{k+1,\tau}^0(\mu_0), J_{k+1,\tau}^\alpha(\mu_0)) \leq 
\left(\underbrace{W_2(J_\tau^0(J_{k,\tau}^0(\mu_0)), J_\tau^0(J_{k,\tau}^\alpha(\mu_0)))}_{\text{(I)}}
+ \underbrace{ W_2(J_\tau^0(J_{k,\tau}^\alpha(\mu_0)), J_\tau^\alpha(J_{k,\tau}^\alpha(\mu_0)))}_{\text{(II)}}\right)^2.
\end{equation}
We will estimate the terms (I) and (II) separately. Indeed, (I) is related to the contraction property of the classic JKO scheme. The term (II) is related to the stability of the scheme through perturbation.

\subsection{Bounding term (I)}

Let us start by considering the following contraction property of the classic JKO scheme, proven in~\cite{ambrosio2005gradient}.
\begin{theorem}[Contraction property of the JKO scheme \cite{ambrosio2005gradient}]\label{contraction AGS}
    Let $\F$ be $\lambda$-convex along generalized geodesics and $\tau < \frac{1}{\lambda_-}$. Then, for all $\mu,\nu\in\mathcal{P}_2(\R^d)$ such that~\eqref{JKO} admit minimizers $J^0_\tau(\mu)$ and $J^0_\tau(\nu)$, 
    \[
    W_2^2(J_\tau^0(\mu),J_\tau^0(\nu))\leq \left(\frac{1}{1+\lambda\tau} \right)^2W_2^2(\mu, \nu) + \frac{R(\tau)}{1+\lambda\tau}
\]
where $R(\tau)=2\tau\left(\mathcal{F}\left(\mu\right)-\mathcal{F}\left(J_{\tau}^0\left(\mu
\right)\right)\right)$. 
\end{theorem}
\begin{remark}
In virtue of Theorem~\ref{welldefined JKO}, existence of $J^0_\tau(\mu)$ and $J^0_\tau(\nu)$ is guaranteed as soon as $\F(\mu)< + \infty$ and $\mu$ and $\nu$ are absolutely continuous.
\end{remark}
Applying this theorem to our case, with $\mu = J^0_{k,\tau}(\mu_0)$ and $\nu = J^\alpha_{k,\tau}(\mu_0)$, since $\mathcal F$ is $\lambda$-convex along generalized geodesics, we find:
\begin{multline}
\label{eq:bound (I)}
\text{(I)}^2 \leq \left(\frac{1}{1+\lambda\tau} \right)^2W_2^2(J_{k,\tau}^0(\mu_0), J_{k,\tau}^\alpha(\mu_0)) + \frac{R_k(\tau)}{1+\lambda\tau}, \\ \mbox{where} \quad R_k(\tau)=2\tau\left(\mathcal{F}\left(J_{k,\tau}^0\left(\mu_0\right)\right)-\mathcal{F}\left(J_{k+1,\tau}^0\left(\mu_0\right)\right)\right).
\end{multline}

For the reader's convenience, let us reprove Theorem~\ref{contraction AGS}. 
\begin{proof}[Proof of Theorem~\ref{contraction AGS}]
Taking $\mu=\mu$ and $\rho=J_\tau^0(\nu)$ in the discrete E.V.I of Theorem \ref{discrete E.V.I}, we obtain following bound:
\[
\frac{1}{2\tau} \left( W_2^2(J_\tau^0(\nu), J_\tau^0(\mu)) - W_2^2(J_\tau^0(\nu), \mu) \right)
\leq \mathcal{F}(J_\tau^0(\nu)) - \mathcal{F}(J_\tau^0(\mu)) - \frac{1}{2\tau} W_2^2(J_\tau^0(\mu), \mu)
{- \frac{\lambda}{2} W_2^2(J_\tau^0(\nu), J_\tau^0(\mu))}.
\]
Doing the same for $\mu=\nu$ and $\rho=\mu$ we get:
\[
\frac{1}{2\tau} \left( W_2^2(\mu, J_\tau^0(\nu)) - W_2^2(\mu, \nu) \right)
\leq \mathcal{F}(\mu) - \mathcal{F}(J_\tau^0(\nu)) - \frac{1}{2\tau} W_2^2(J_\tau^0(\nu), \nu)
{- \frac{\lambda}{2} W_2^2(\mu, J_\tau^0(\nu))}.
\]
Summing these two inequalities, we find, up to rearranging the terms:
\begin{multline}
\label{eq:proof_contraction_intermediary_step}
(1+\lambda\tau)\frac{W_2^2(J_\tau^0(\nu), J_\tau^0(\mu))}{2\tau}-\frac{W_2^2(\mu, \nu)}{2\tau}\\\leq \F(\mu)-\F(J_\tau^0(\mu))-\frac{1}{2\tau}W_2^2(J_\tau^0(\mu),\mu)- \frac{1}{2\tau} W_2^2(J_\tau^0(\nu), \nu)
{- \frac{\lambda}{2} W_2^2(\mu, J_\tau^0(\nu))}.
\end{multline}
We easily conclude using the following lemma.
\end{proof}
\begin{lemma}\label{squared-triangle inequality}
    For all $\mu,\nu,\rho\in\mathcal{P}_2(\R^d)$ and for all $\lambda \in \R$ and $\tau < \frac{1}{\lambda_-}$, the following inequality holds:
    \[
    \frac{\lambda\tau}{1+\lambda\tau}{W_2^2(\mu,\nu)}\leq{\lambda\tau}W_2^2(\mu,\rho)+{W_2^2(\nu,\rho)}
    \]
\end{lemma}
Applying this lemma to $\rho = J^0_\tau(\nu)$, we find
    \[
    \frac{\lambda\tau}{1+\lambda\tau}\frac{W_2^2(\mu,\nu)}{2\tau}\leq\frac{\lambda}{2}W_2^2(\mu,J_\tau^0(\nu))+\frac{W_2^2(\nu,J_\tau^0(\nu))}{2\tau},
    \]
which, plugged into~\eqref{eq:proof_contraction_intermediary_step}, provides
\[
(1+\lambda\tau)\frac{W_2^2(J_\tau^0(\nu), J_\tau^0(\mu))}{2\tau}-\frac{W_2^2(\mu, \nu)}{2\tau}\leq \F(\mu)-\F(J_\tau^0(\mu))-\frac{1}{2\tau}W_2^2(J_\tau^0(\mu),\mu)-\frac{\lambda\tau}{1+\lambda\tau}\frac{W_2^2(\mu,\nu)}{2\tau}.
\]
Forgetting the nonnegative term $W_2^2(J^0_\tau(\mu),\mu)$ and rearranging the terms leads to Theorem~\ref{contraction AGS}.

Let us close this part of the proof with the proof of Lemma~\ref{squared-triangle inequality} (the case $\lambda<0$ can be found in~\cite{ambrosio2005gradient}). 

\begin{proof}[Proof of lemma \ref{squared-triangle inequality}]
    For $\lambda=0$ there is nothing to show. Otherwise let us distinguish the cases $\lambda>0$ and $\lambda<0$. \\
    \underline{\textbf{Case $\lambda>0$.}} The triangle inequality gives that : $W_2(\mu,\nu)\leq W_2(\mu,\rho)+W_2(\rho,\nu)$. We will use the following classic inequality $(a+b)^2\leq pa^2+p^*b^2$ where $a,b\in\R$, $p\in(1,+\infty)$ and $\frac{1}{p}+\frac{1}{p^*}=1$. Since $\lambda>0$ then $1+\lambda\tau>1$, so we can apply the inequality for $p=1+\lambda\tau$ and $p^*=\frac{1+\lambda\tau}{\lambda\tau}$. We obtain:
    \[
    W_2^2(\mu,\nu)\leq (1+\lambda\tau)W_2^2(\mu,\rho)+\frac{1+\lambda\tau}{\lambda\tau}W_2^2(\rho,\nu).
    \]
    Multiplying by $\frac{\lambda\tau}{1+\lambda\tau}>0$, we get the lemma for $\lambda>0$.\\
    \underline{\textbf{Case $\lambda<0$.}}  The triangle inequality gives that : $W_2(\mu,\rho)\leq W_2(\mu,\nu)+W_2(\nu,\rho)$ since $\frac{-1}{\tau}<\lambda<0$, then $\frac{1}{1+\lambda\tau}>1$ so as previously, we can apply the classic inequality for $p=\frac{1}{1+\lambda\tau}$ and $p^*=\frac{-1}{\lambda\tau}$ and obtain:
    \[
    W_2^2(\mu,\rho)\leq \frac{1}{1+\lambda\tau}W_2^2(\mu,\nu)+\frac{-1}{\lambda\tau}W_2^2(\nu,\rho).
    \]
    Multiplying by $\lambda\tau<0$, we get:
    \[
    \lambda\tau W_2^2(\mu,\rho)\geq \frac{\lambda\tau}{1+\lambda\tau}W_2^2(\mu,\nu)-W_2^2(\nu,\rho),
    \]
    which is the lemma for $\lambda<0$.
\end{proof}

\subsection{Bounding Term (II)} 
Let us now estimate term (II) which is the main novelty of the proof.
In this section, we want to show the following bound:
    \begin{equation}
        \label{eq:bound_II}
    \text{(II)}^2 \leq \frac{\Tilde{R}_k(\tau,\alpha)}{1+\lambda\tau}\quad \mbox{where} \quad \Tilde{R}_k(\tau,\alpha)={K\alpha \tau^2 } + {\alpha}\tau\left({H( J_{k,\tau}^\alpha(\mu_0))-H(J_{k+1,\tau}^\alpha(\mu_0))}\right).
    \end{equation}
In order to lighten the notations, let us denote:
\[
\mu = J_{k,\tau}^\alpha(\mu_0), \quad \nu^0 = J_\tau^0(\mu), \quad \nu^\alpha = J_\tau^\alpha(\mu).
\]
The first step consists in applying the discrete E.V.I of Theorem \ref{discrete E.V.I} to $\mu$ and $\nu^\alpha$, leading thanks to the $\lambda$-convexity of $\mathcal F$ to:
   \[
\frac{1}{2\tau}W_2^2(\nu^0, \nu^\alpha) \leq \frac{1}{1 + \lambda \tau} \left( \frac{W_2^2(\mu, \nu^\alpha)}{2\tau} + \mathcal{F}(\nu^\alpha) - \left( \frac{W_2^2(\mu, \nu^0)}{2\tau} + \mathcal{F}(\nu^0) \right) \right).
\]
From Proposition~\ref{below bound sch by W-2}, we have for all $\mu,\nu \in \mathcal P_2(\R^d)$,
    \[
\frac{\Sch^{\alpha\tau}(\mu,\nu)}{\tau} \geq \frac{\alpha}{2} (H(\mu) + H(\nu))+\frac{W_2^2(\mu,\nu)}{2\tau}.
\]
Therefore, defining the cost associated with the JKO scheme and the entropic JKO scheme as
    \[
C(\tau, \alpha) = \frac{\Sch^{\alpha\tau}}{\tau}(\mu, \nu^\alpha) + \mathcal{F}(\nu^\alpha)
\quad \text{and} \quad
C(\tau, 0) = \frac{W_2^2}{2\tau}(\mu, \nu^0) + \mathcal{F}(\nu^0),
\]
we find
\[
\frac{1}{2\tau}W_2^2(\nu^0, \nu^\alpha) \leq \frac{1}{1 + \lambda \tau} \left(  C(\tau, \alpha)-C(\tau, 0)-\alpha\frac{H(\mu)+H(\nu^\alpha)}{2}\right).
\]
Then, the last step consists in proving the following bound between the different costs:
\begin{equation}
\label{eq:comparison_costs}
C(\tau, \alpha) - C(\tau, 0) \leq \alpha H(\mu) + K\frac{\alpha}{2}\tau.
\end{equation}
Indeed, plugging this inequality into the previous line directly leads to
$$ \frac{1}{2\tau}W_2^2(\nu^0, \nu)\leq \frac{K\alpha \tau }{2(1+\lambda\tau)} + \frac{\alpha}{1+\lambda\tau}\frac{H(\mu)-H(\nu^\alpha)}{2}, $$
which is a rewriting of equation~\eqref{eq:bound_II}.

Our last task is to prove inequality~\eqref{eq:comparison_costs}.
The argument to compare the two costs is to construct a competitor of the entropic problem using the minimizer of the non entropic one; for this, we can follow the idea suggested by the heuristic remark made in Subsection \ref{comments hyp}, that starting from the same measure $\mu$ the solution of the gradient flow and the regularized gradient flow $\rho^0$ and $\rho^\alpha$ verify 
\begin{equation*}
\frac{\diff }{\diff t}\left(\rho^0*\sigma_{\alpha t}\right)\Big|_{t=0}=\frac{\diff }{\diff t}\rho^\alpha\Big|_{t=0}.
\end{equation*}
Hence, being $\nu^0$ an approximation of $\rho^0(\tau)$ and $\nu^\alpha$ an approximation of $\rho^\alpha(\tau)$ we expect $\nu^\alpha$ to be close to $\nu^0*\sigma_{\alpha \tau}$. So let us consider this last measure as a competitor for the entropic JKO scheme.

Let $(\rho, c)$ be the geodesic and its associated velocity between $\mu$ and $\nu^0$, defined in Definition \ref{Sch Benamou Brenier}. Define:
\[
\tilde{\rho}_t = \rho_t * \sigma_{\alpha t}, \quad \tilde{c}_t = \frac{(\rho_tc_t) * \sigma_{\alpha t}}{\tilde{\rho}_t}\quad\text{and}\quad\tilde m_t:=\tilde\rho\tilde c=(\rho_tc_t) * \sigma_{\alpha t}.
\]
Then:
\[
\partial_t \tilde{\rho}_t + \operatorname{div}(\tilde{\rho}\tilde{c}_t) = \frac{\alpha}{2} \Delta \tilde{\rho}_t.
\]
Hence, it is a competitor in the formulation of the Schrodinger cost from Definition \ref{Sch Benamou Brenier}. Therefore,
\[
C(\tau, \alpha) \leq \alpha H(\mu) + \int_0^\tau \int \frac{|\tilde{c}_t|^2}{2\tau} \, \diff\tilde{\rho}_t \diff t + \mathcal{F}(\nu^0 * \sigma_{\alpha\tau}).
\]
Using the convexity of $J: (\rho, m) \mapsto \int \frac{|m|^2}{2\rho}$ and Jensen’s inequality:
\[
J(\tilde{\rho}, \tilde{m}) \leq J(\rho, m) = \frac{W_2^2}{2\tau}(\mu, \nu^0).
\]
Therefore:
\[
C(\tau, \alpha) \leq \alpha H(\mu) + C(\tau, 0) + \mathcal{F}(\nu^0 * \sigma_{\alpha\tau}) - \mathcal{F}(\nu^0).
\]
Since $\F$ verifies Hypothesis~\ref{hypothesis}, in particular the last point implies that:
\[
C(\tau, \alpha) \leq \alpha H(\mu) + C(\tau, 0) + K\frac{\alpha}{2}\tau,
\]
which is nothing but inequality~\eqref{eq:comparison_costs}.

We are now in position to prove Theorem \ref{the theorem}.

\subsection{Conclusion of the result}
In view  of equations~\eqref{eq:bound (I)} and ~\eqref{eq:bound_II} we have almost enough to conclude. The last ingredient is the following technical proposition.

\begin{proposition}[Squared discrete Gronwall lemma]\label{square discrete gronwall}
    Let $\lambda\in\R$, $\tau<\frac{1}{\lambda_-}$, and $(a_k)$ and $(b_k)$ be two non-negative sequences. If $(u_k)$ is a sequence verifying the following inequality:
    \begin{equation}\label{eq: squared disc grow hypo}
    u_0 = 0 \qquad \mbox{and} \qquad u_{k+1}\leq \sqrt{\left(\frac{u_k}{1+\lambda\tau}\right)^2 + a_{k+1}}+b_{k+1},
    \end{equation}
    then for all $n\in \N$ we have:
    \begin{equation*}
    u_n (1+\lambda\tau)^{n}\leq \sqrt{\sum_{k=1}^n \left(1+\lambda\tau\right)^{2k} a_k}+ \sum_{k=1}^n \left(1+\lambda\tau\right)^{k} b_k.
    \end{equation*}
\end{proposition}
\begin{proof}
    At step $k\in\N$, multiplying inequality \eqref{eq: squared disc grow hypo} by $(1+\lambda\tau)^{2(k+1)}$ we obtain:
    $$    (1+\lambda\tau)^{(k+1)}u_{k+1}\leq \sqrt{\left({(1+\lambda\tau)^{k}u_k}\right)^2 + (1+\lambda\tau)^{2(k+1)}a_{k+1}}+(1+\lambda\tau)^{(k+1)}b_{k+1}.
    $$
    Up to replacing $u_k$ by $(1+\lambda\tau)^{k}u_k$, $a_{k+1}$ by $(1+\lambda\tau)^{2(k+1)}a_{k+1}$ and $b_{k+1}$ by $(1+\lambda\tau)^{k+1}b_{k+1}$, we can assume that $\lambda=0$. Now let us introduce for all $n\in\N$, $A_n=\sum\limits_{k=1}^na_k$, $B_n=\sum\limits_{k=1}^nb_k$ and $(v_{k})$ the sequence defined by $v_0=0$ and the following iterative scheme:
    \[
    \forall k\in\N, \quad v_{k+1}= \sqrt{v_k^2 + a_{k+1}}+b_{k+1}.
    \]
    An easy induction shows that for all $k \in \N$, $u_k\leq v_k$. Moreover, for all $k\in\N$;
    \begin{align*}
    \left(v_{k+1}-B_{k+1}\right)^2&=\left(\sqrt{v_k^2 + a_{k+1}}+b_{k+1}-B_{k+1}\right)^2\\
    &=\left(\sqrt{v_k^2 + a_{k+1}}-B_{k}\right)^2=v_k^2 + a_{k+1}-2\sqrt{v_k^2 + a_{k+1}}B_k+B_k^2.
    \end{align*}
    As this stage we simply use
    \[
    -2\sqrt{v_k^2 + a_{k+1}}B_k\leq -2v_kB_k
    \]
    to find 
    \[
    \left(v_{k+1}-B_{k+1}\right)^2\leq v_k^2 -2v_kB_k+B_k^2+ a_{k+1}=\left(v_{k}-B_{k}\right)^2+ a_{k+1}.
    \]
    Summing these inequalities, for all $n\in\N^*$,
    \[
    \left(v_{n}-B_{n}\right)^2\leq A_n.
    \]
    Thus,
    \[
    v_{n}-B_{n}\leq|v_{n}-B_{n}|\leq \sqrt{A_n}
    \]
    and so $u_n\leq v_n\leq B_n+\sqrt{A_n}$, as claimed.
\end{proof}
Let us now use this proposition to conclude the proof of Theorem~\ref{the theorem}. 
Putting the inequalities obtained in equations~\eqref{eq:bound (I)} and~\eqref{eq:bound_II} into~\eqref{eq:splitting (I) and (II)}, we get:
\begin{align*}
W_2^2(J_{k+1,\tau}^0(\mu_0), J_{k+1,\tau}^\alpha(\mu_0)) &\leq \left(\text{(I)}+\text{(II)}\right)^2\\
&\leq \left(\sqrt{\left(\frac{W_2(J_{k,\tau}^0(\mu_0), J_{k,\tau}^\alpha(\mu_0))}{1+\lambda\tau}\right)^2+\frac{R_k(\tau)}{1+\lambda\tau}}+\sqrt{\frac{\Tilde{R}_k(\alpha,\tau)}{1+\lambda\tau}}\right)^2 .
\end{align*}
Then, applying Lemma \ref{square discrete gronwall} with, for all $k \in \N$:
\[
u_k=W_2(J_{k,\tau}^0(\mu_0), J_{k,\tau}^\alpha(\mu_0)), \quad a_{k+1}=\frac{R_k(\tau)}{1+\lambda\tau} \quad \text{and}\quad b_{k+1}=\sqrt{\frac{\Tilde{R}_k(\alpha,\tau)}{1+\lambda\tau}} ,
\]
we obtain for all $n\in\N^*$:
\begin{equation}\label{eq:bound bf CS}
(1+\lambda\tau)^{n}W_2(J_{n,\tau}^0(\mu_0), J_{n,\tau}^\alpha(\mu_0))\leq \sqrt{\sum_{k=0}^{n-1} \left(1+\lambda\tau\right)^{2k} \frac{R_k(\tau)}{1+\lambda\tau}}+ \sum_{k=0}^{n-1}\left(1+\lambda\tau\right)^{k} \sqrt{\frac{\Tilde{R}_k(\alpha,\tau)}{1+\lambda\tau}}.
\end{equation}
From now on, we fix $n\in \N$. Using the Cauchy-Schwartz inequality, we deduce the following bound:
\begin{equation}\label{eq: bond bet scheme after CS}
(1+\lambda\tau)^{n}W_2(J_{n,\tau}^0(\mu_0), J_{n,\tau}^\alpha(\mu_0))\leq \sqrt{\sum_{k=0}^{n-1} \left(1+\lambda\tau\right)^{2k} \frac{R_k(\tau)}{1+\lambda\tau}}+\sqrt{\frac{1}{1+\lambda\tau}\sum_{k=0}^{n-1}\left(1+\lambda\tau\right)^{2k}\sum_{k=0}^{n-1} \Tilde{R}_k(\alpha,\tau)}.
\end{equation}
From now on, we will treat the cases $\lambda=0$ and $\lambda\neq0$ separately.\\
\underline{\textbf{Case $\lambda=0$}}.
In this case, equation~\eqref{eq: bond bet scheme after CS} rewrites as:
\[
W_2(J_{n,\tau}^0(\mu_0), J_{n,\tau}^\alpha(\mu_0))\leq  \sqrt{\sum_{k=0}^{n-1} R_k(\tau)}+\sqrt{n\sum_{k=0}^{n-1} \Tilde{R}_k(\alpha,\tau)},
\]
where, by definitions of $R_k(\tau)$ and $\Tilde{R}_k(\alpha,\tau)$ in equations~\eqref{eq:bound (I)} and~\eqref{eq:bound_II},
\[
\sum_{k=0}^{n-1}R_k(\tau)=2\tau(\F(\mu_0)-\F(J_{n,\tau}^0(\mu_0))\quad\text{and}\quad \sum_{k=0}^{n-1} \Tilde{R}_k(\alpha,\tau)=\alpha\tau\left(H(\mu_0)-H(J_{n,\tau}^\alpha(\mu_0))+K n\tau\right).
\]
Finally,
\[
W_2(J_{n,\tau}^0(\mu_0), J_{n,\tau}^\alpha(\mu_0))\leq \sqrt{2\tau(\F(\mu_0)-\F(J_{n,\tau}^0(\mu_0))}+\sqrt{n\tau\alpha\left(H(\mu_0)-H(J_{n,\tau}^\alpha(\mu_0))+K n\tau\right)},
\]
which is the desired bound when $\lambda=0$.\\
\underline{\textbf{Case $\lambda\neq0$}}.
Using that for all $k\leq n-1$, $$(1+\lambda\tau)^{2k}\leq\max\{1,(1+\lambda\tau)^{2(n-1)}\} \quad \mbox{and} \quad \sum\limits_{k=0}^{n-1}\left(1+\lambda\tau\right)^{2k}=\frac{(1+\lambda\tau)^{2n}-1}{2\lambda\tau(1+\frac{\lambda\tau}{2})},$$ it follows:
\begin{multline*}
(1+\lambda\tau)^{n}W_2(J_{n,\tau}^0(\mu_0), J_{n,\tau}^\alpha(\mu_0))\\
\leq \sqrt{\frac{\max\{1,(1+\lambda\tau)^{2(n-1)}\}}{1+\lambda\tau}} \sqrt{\sum_{k=0}^{n-1}R_k(\tau)}+\sqrt{\frac{(1+\lambda\tau)^{2n}-1}{2\lambda\tau}}\sqrt{\sum_{k=0}^{n-1} \frac{\Tilde{R}_k(\alpha,\tau)}{(1+\lambda\tau)(1+\frac{\lambda\tau}{2})} },
\end{multline*}
where the following still holds true:
\[
\sum_{k=0}^{n-1}R_k(\tau)=2\tau(\F(\mu_0)-\F(J_{n,\tau}^0(\mu_0))\quad\text{and}\quad \sum_{k=0}^{n-1} \Tilde{R}_k(\alpha,\tau)=\alpha\tau\left(H(\mu_0)-H(J_{n,\tau}^\alpha(\mu_0))+K n\tau\right).
\]
Finally,
\begin{multline*}
W_2(J_{n,\tau}^0(\mu_0), J_{n,\tau}^\alpha(\mu_0)) 
\leq \sqrt{
    2  \frac{ \max\{(1+\lambda\tau)^{-2n+2},\, 1\} }{ (1+\lambda\tau)^3 }
}  \sqrt{\tau} 
\sqrt{ \mathcal{F}(\mu_0) - \mathcal{F}(J_{n,\tau}^0(\mu_0)) } \\
\quad + \sqrt{ 
    \frac{1 - (1+\lambda\tau)^{-2n}}{2\lambda} 
}  
\sqrt{ 
    \frac{ \alpha \left( H(\mu_0) - H(J_{n,\tau}^\alpha(\mu_0)) + K n \tau \right) }
         { (1+\lambda\tau)\left(1 + \frac{\lambda\tau}{2} \right) } 
}.
\end{multline*}
Now, observe the following identity:
\begin{equation*}
    \max\{(1+\lambda\tau)^{-2n+2},1\}=(1-\lambda_-\tau)^{-2n+2}.
\end{equation*} 
Second, if $\tau<\frac{1}{2\lambda_-}$, then we have:
\[
(1-\lambda_-\tau)^2\leq 1\quad \text{and} \quad\frac{1}{1+\lambda\tau}\leq 1+4\lambda_-\tau.
\]
Finally, the Mean value theorem provides:
\[
\frac{1}{\sqrt{ (1+\lambda\tau)(1+\frac{\lambda\tau}{2})} }\leq 1+3\lambda_-\tau.
\]
So we obtain:
\begin{multline*}
  W_2(J_{n,\tau}^0(\mu_0), J_{n,\tau}^\alpha(\mu_0))\leq \sqrt2(1+4\lambda_-\tau)^{\frac{3}{2}} (1-\lambda_-\tau)^{-n}\sqrt{\tau}\sqrt{(\F(\mu_0)-\F(J_{n,\tau}^0(\mu_0))}\\
+(1+3\lambda_-\tau)\sqrt{\frac{1-(1+\lambda\tau)^{-2n}}{2\lambda}}\sqrt{\alpha\left(H(\mu_0)-H(J_{n,\tau}^\alpha(\mu_0))+K n\tau\right)}.
\end{multline*}
With this, we conclude the proof of the main theorem.

\section{Proof of Corollary \ref{cv rate entropic JKO}}
\label{sec:proof_cor}
The purpose of this section is to prove Corollary~\ref{cv rate entropic JKO}. The main idea is to apply the bound found in Theorem~\ref{the theorem} and the following convergence rates depending on the value of $\lambda$, proved by Ambrosio, Gigli and Savaré in \cite[Theorems 4.0.7, 4.0.9 and 4.0.10]{ambrosio2005gradient}:

\noindent\textbf{(i) If $\lambda=0$.}
For all $t>0$, we have for all $n\in\mathbb N^*$:
\begin{equation*}
W_2^2\!\left(\rho^0(t),\,J_{n,t/n}^0(\mu_0)\right)
\le
\frac{t}{n}\Big(\mathcal F(\mu_0)-\mathcal F(J^0_{t/n}(\mu_0))\Big).
\end{equation*}

\medskip
\noindent\textbf{(ii) If $\lambda<0$.}
For all $t>0$, we have for all $n \in \N^*$ such that $\tau=t/n<(-\lambda)^{-1}$ (or equivalently $n>|\lambda|t$):
\begin{multline}\label{eq:lambdaneg}
W_2^2\!\left(\rho^0(t),\,J_{n,t/n}^0(\mu_0)\right)
\le
c_n(t)\,\frac{t}{n}\,
\Bigl(\mathcal F(\mu_0)-\inf_{n' > - \lambda t}\inf_{k \leq n'}\mathcal F (J^0_{k,t/n'}(\mu_0)\Bigr)\,
e^{-2\lambda t},\\
\mbox{where }
c_n(t):=\left(1+\sqrt{\tfrac{4}{3}|\lambda|\,\tfrac{t}{n}}\right)^2 .
\end{multline}

\medskip
\noindent\textbf{(iii) If $\lambda>0$.}
For all $\theta>0$, let us define
\[
\lambda_{\theta}:=\frac{\ln\!\bigl(1+\theta \lambda\bigr)}{\theta}.
\]
Then for all $t>0$, we have for all $n\in\mathbb N^*$:
\begin{multline*}
W_2^2\!\left(\rho^0(t),\,J_{n,t/n}^0(\mu_0)\right)
\le
c_n(t)\,\frac{t}{n}\,
\Bigl(\mathcal F(\mu_0)-\inf_{\mathcal P_2(\mathbb R^d)}\mathcal F\Bigr)\,
e^{-2\lambda_{t/n} t},
\\ \mbox{where }
c_n(t):=\left(1+\lambda\tfrac{t}{n}\right)\left(1+\sqrt{2\lambda t}\right)^4 .
\end{multline*}
\begin{remark}
\label{rem:below_bound_F}
    In fact, in~\cite{ambrosio2005gradient}, when $\lambda < 0$, the estimate~\eqref{eq:lambdaneg} is written with $\inf \mathcal F$ in place of 
    \begin{equation*}
         \inf_{n' > - \lambda t}\inf_{k \leq n'} \mathcal F (J^0_{k, t/n'}(\mu_0)).
    \end{equation*}
    This could be a problem since when $\lambda \leq 0$, $\F$ is not necessarily below bounded. But for a given $t\geq 0$, let us call
    \begin{equation*}
        M_t:=  \inf_{n>\lambda_-t}\inf_{k\leq n} \F(J_{k,t/n}^0(\mu_0)).
    \end{equation*}
    Using iteratively Proposition~\ref{prop:lower bound along JKO implies wlog F lower bounded} from the appendix, it appears that replacing the functional $\mathcal F$ by $\F^{M_t}:=\max\{\F,M_t\}$ does not affect the points $(J^0_{k,t/n}(\mu_0))_{0 \leq k \leq n}$ reached by the JKO scheme up to step $k=n$. Letting $n$ tend to $+\infty$, $\rho^0(t)$ is therefore the evaluation at time $t$ of the gradient flow of both functionals $\F$ and $\F^{M_t}$ starting from $\mu_0$, and our bound~\eqref{eq:lambdaneg} follows from using $\F^{M_t}$ instead of $\F$ in the right hand side.
\end{remark}

By comparing the bound that we want to prove in Corollary~\ref{cv rate entropic JKO} with the bounds obtained in Theorem~\ref{the theorem} and the bounds just stated, we can see that we only need to derive estimates to show that the following quantities:
\begin{equation*}
    \F(\mu_0) - \F(J_{k,t/n}^0(\mu_0)) \qquad \mbox{and} \qquad H(\mu_0) -H(J_{n,t/n}^\alpha(\mu_0))
\end{equation*}
do not tend to $+ \infty$ as $n \to + \infty$.
This is the purpose of the two next propositions, which easily imply Corollary~\ref{cv rate entropic JKO}.

 \begin{proposition}\label{prop:lower bound F alonf jko scheme}
       Let $\F$ be $\lambda$-convex along generalized geodesics and $W_2$-l.s.c. Let $\mu_0\in \mathcal P_2(\R^d)$ be absolutely continuous and such that $\F(\mu_0) < + \infty$. There exists $c$ only depending on $\mu_0$ and $\F$, such that for all $t>0,$ all $n>4\lambda_-t$, and all $k\leq n$,
       \begin{itemize}
           \item if $\lambda=0$,
           \[
       \F(\mu_0)-\F(J_{k,t/n}^0(\mu_0))\leq c\left(1+\frac kn t\right),
       \]
       \item if $\lambda<0$,
       \[
       \F(\mu_0)-\F(J_{k,t/n}^0(\mu_0))\leq c\left(\frac{n}{n+2\lambda t}\right)^k,
       \]
       \item if $\lambda>0$,
       \[
       \F(\mu_0)-\F(J_{k,t/n}^0(\mu_0))\leq c.
       \]
       \end{itemize}
       
   \end{proposition}
   \begin{proposition}\label{prop:lower bound H along entropic jko scheme}
    Let $\mathcal F$ satisfy Hypothesis~\ref{hypothesis}. Let $\mu_0$ be such that $\F(\mu_0)< + \infty$ and $H(\mu_0)< + \infty$. Then for all $\alpha_0>0$, there exists $C$ only depending on $d,\mu_0,\F,\alpha_0, K$ such that for every $t>0$, $\alpha\leq\alpha_0$ and $n>16\lambda_-t$,    
    \begin{itemize}
        \item if $\lambda=0$,
        \[
    H(\mu_0) -H(J_{n,t/n}^\alpha(\mu_0))\leq C(1 + \ln( t)),
    \]
    \item if $\lambda<0$,
    \[
    H(\mu_0) -H(J_{n,t/n}^\alpha(\mu_0))\leq C(1+t),
    \]
    \item if $\lambda>0$,
    \[
    H(\mu_0) -H(J_{n,t/n}^\alpha(\mu_0))\leq C.
    \]
    \end{itemize}
    \end{proposition}

 We will now prove these propositions. Both of them can be deduced from an upper bound on the Wasserstein distances along our schemes. 
    
    Firstly, we prove Proposition \ref{prop:lower bound F alonf jko scheme}. The starting point is the following proposition, which relates an upper bound on the Wasserstein distance to a lower bound on $\F$. In the Hilbertian framework developed in Subsection~\ref{general cnx geo explain}, this proposition follows from the $\lambda$-convexity of $\F$ along generalized geodesics and the Hahn-Banach theorem applied in $L^2$, and we decided to skip the proof.
\begin{proposition}\label{prop:lower bound on F by the wasserstein distance}
    Let $\F$ be $\lambda$-convex along generalized geodesics and $W_2$-l.s.c, and let $\mu_0 \in \mathcal P_2(\R^d)$ be such that $\F(\mu_0)< + \infty$. Then there exist $c_0,c_1>0$ depending only on $\mu_0$ and $\F$ such that for all $\rho\in\mathcal P_2(\R^d)$,
    \[
        \F(\rho)\geq \F(\mu_0)-c_0-c_1W_2(\rho,\mu_0)+\frac\lambda2W_2^2(\rho,\mu_0).
    \]
    \end{proposition}
    
We will show Proposition \ref{prop:lower bound F alonf jko scheme} by using the discrete E.V.I of Theorem~\ref{discrete E.V.I}, Proposition~\ref{square discrete gronwall} and the previous Proposition~\ref{prop:lower bound on F by the wasserstein distance}.
  
   \begin{proof}[Proof of Proposition \ref{prop:lower bound F alonf jko scheme}]
   \noindent The case $\lambda>0$ is a direct consequence of Proposition~\ref{prop:lower bound on F by the wasserstein distance}.
   In the following we are assuming $\lambda\leq0$.
   
   For now, consider any $0<\tau<\frac{1}{4\lambda_-}$ and any $k\in\N$. (Ultimately, we will obviously choose $\tau=\frac tn$ and $k\leq n$, but this is not necessary for the following and will simplify the notation).
   
   Applying the discrete E.V.I of Theorem \ref{discrete E.V.I} to $\rho=\mu_0$ and $\mu=J_{k,\tau}^0(\mu_0)$ provides, forgetting the nonnegative term $W_2^2(J_{k,\tau}(\mu_0),J_{k+1,\tau}(\mu_0))$:
    \[    \frac1{2\tau}\left(W_2^2\left(\mu_0,J_{k+1,\tau}^0(\mu_0)\right)-W_2^2(\mu_0,J_{k,\tau}^0(\mu_0)\right)\leq \F(\mu_0)-\F(J_{k+1,\tau}^0(\mu_0))-\frac\lambda2W_2^2(\mu_0,J_{k+1,\tau}^0(\mu_0)).
    \]
    Proposition \ref{prop:lower bound on F by the wasserstein distance} gives:
    \[    \frac1{2\tau}\left(W_2^2\left(\mu_0,J_{k+1,\tau}^0(\mu_0)\right)-W_2^2(\mu_0,J_{k,\tau}^0(\mu_0)\right)\leq c_0+c_1W_2(\mu_0,J_{k+1,\tau}^0(\mu_0))-\lambda W_2^2(\mu_0,J_{k+1,\tau}^0(\mu_0)).
    \]
    Then, rearranging the terms, 
    \[    (1+2\lambda\tau)W_2^2(\mu_0,J_{k+1,\tau}^0(\mu_0))-2\tau c_1 W_2(\mu_0,J_{k+1,\tau}^0(\mu_0))-\left(2\tau c_0+W_2^2(\mu_0,J_{k,\tau}^0(\mu_0)\right)\leq0.
    \]
    Letting $u_k=\frac{W_2(\mu_0,J_{k,\tau}^0(\mu_0))}{\max\{\sqrt{c_0},c_1\}}$, then 
    \[    (1+2\lambda\tau)u_k^2-2\tau  u_k-\left(2\tau +u_{k-1}^2\right)\leq0.
    \]
    Remind that we chose $\tau$ small enough to obtain $1+2\lambda\tau>\frac{1}{2}>0$, which implies that $u_k$ is between the two square roots of the polynomial: $(1+2\lambda\tau)X^2-2\tau X-\left(2\tau +u_{k-1}\right)$. It follows:
    \[
    u_{k+1}\leq \frac{\tau}{1+2\lambda\tau}+\sqrt{\frac{\tau^2}{(1+2\lambda\tau)^2}+\frac{2\tau}{1+2\lambda\tau}+\frac{u_k^2}{1+2\lambda\tau}}.
    \]
    Since, $1+2\lambda\tau>0$, we can replace in Proposition \ref{square discrete gronwall}, $1+\lambda\tau$ by $\sqrt{1+2\lambda\tau}$, and obtain for all $k\in\N$,
    \[  u_k\sqrt{1+2\lambda\tau}^k\leq\sqrt{\sum_{i=1}^k(1+2\lambda\tau)^i\left(\frac{\tau^2}{(1+2\lambda\tau)^2}+\frac{2\tau}{1+2\lambda\tau}\right)}+\frac{\tau}{1+2\lambda\tau}\sum_{i=1}^k\sqrt{1+2\lambda\tau}^i.
    \]
    Dividing by $\sqrt{1+2\lambda\tau}^k$ and making the change of variable $i=k-i$ in the sums, we obtain:
    \begin{equation}\label{eq:bound wasserstein distance for lower F before computing the sums}
    u_k \leq\sqrt{\sum_{i=0}^{k-1}(1+2\lambda\tau)^{-i}\left(\frac{\tau^2}{(1+2\lambda\tau)^2}+\frac{2\tau}{1+2\lambda\tau}\right)}+\frac{\tau}{1+2\lambda\tau}\sum_{i=0}^{k-1}\sqrt{1+2\lambda\tau}^{-i}.
    \end{equation}
    In order to continue the discussion, we need to distinguish the case $\lambda<0$ and $\lambda=0$.\\
    \underline{\textbf{Case $\lambda<0$}}. Computing the geometric sums leads to
    \[
    u_k\leq \sqrt{\frac{1}{-2\lambda}\left[\left(\frac{1}{1+2\lambda\tau}\right)^k-1\right]}\sqrt{\frac{\tau}{1+2\lambda\tau}+{2}}+\left[\left(\frac{1}{\sqrt{1+2\lambda\tau}}\right)^k-1\right]\frac{\tau}{\sqrt{1+2\lambda\tau}-(1+2\lambda\tau)}.
    \]
    Using that $1+2\lambda\tau>\frac12$ and forgetting the $-1$, we obtain:
    \[
    u_k\leq \sqrt{\left(\frac{1}{1+2\lambda\tau}\right)^k}\left(\sqrt{\frac{1+\tau}{|\lambda|}}+\frac{\tau}{\sqrt{1+2\lambda\tau}(1-\sqrt{1+2\lambda\tau})}\right).
    \]
    But for every $x\in[-\frac{1}{2},0],$ $1-\sqrt{1+x}\geq \frac1{\sqrt6}|x|$, so that
    \[
    u_k\leq \sqrt{\left(\frac{1}{1+2\lambda\tau}\right)^k}\left(\sqrt{\frac{1+\tau}{|\lambda|}}+\frac{\sqrt3}{|\lambda|}\right).
    \]
    Thus, $W_2^2(\mu_0,J_{n,\tau}^0(\mu_0))\leq C\left(\frac{1}{1+2\lambda\tau}\right)^k$
    where $C$ depends only on $\F$ and $\mu_0$. Using Proposition \ref{prop:lower bound on F by the wasserstein distance}, we obtain that:
    \[
    \F(\mu_0)-\F(J_{k,\tau}^0(\mu_0))\leq c_0+c_1W_2(\mu_0,J_{k,\tau}^0(\mu_0))-\frac\lambda2W_2^2(\mu_0,J_{k,\tau}^0(\mu_0)).
    \]
    As $c_1W_2(\mu_0,J_{k,\tau}^0(\mu_0))\leq \frac{c_1^2}{2}+\frac12W_2^2(\mu_0,J_{k,\tau}^0(\mu_0))$ and $1\leq\left(\frac{1}{1+2\lambda\tau}\right)^k$, there exists a real number still called $C$ depending only on $\mu_0$ and {$\F$}, such that:
    \[
    \F(\mu_0)-\F(J_{k,\tau}^0(\mu_0))\leq C \left(\frac{1}{1+2\lambda\tau}\right)^k.
    \]
    Taking $\tau=\frac{t}{n}$, we conclude the lemma for $\lambda<0$.
    \\
    \underline{\textbf{Case $\lambda=0$.}} This time, equation~\eqref{eq:bound wasserstein distance for lower F before computing the sums} leads to
    \[
    u_k\leq \sqrt{k\tau^2+2k\tau}+k\tau\leq \sqrt{2k\tau}+2k\tau.
    \]
    Thus, there exists $C$ only depending on $\mu_0$ and $\F$ such that
    \[
    W_2(\mu_0,J_{k,\tau}^0(\mu_0))\leq C(\sqrt{k\tau}+k\tau).
    \]
    Using Proposition \ref{prop:lower bound on F by the wasserstein distance}, we obtain that:
    \[
    \F(\mu_0)-\F(J_{k,\tau}^0(\mu_0))\leq c_0+c_1W_2(\mu_0,J_{k,\tau}^0(\mu_0)).
    \]
    It follows that there exists a real number still called $C$, depending only on $\mu_0$ and $\F$, such that:
    \[
    \F(\mu_0)-\F(J_{n,\tau}^0(\mu_0))\leq C (1+n\tau).
    \]
     We conclude the proof of the proposition by taking $\tau=\frac{t}{n}$.\end{proof}
The next step is to prove Proposition \ref{prop:lower bound H along entropic jko scheme} by obtaining a lower bound on $H$ along the entropic JKO scheme. In view of Proposition~\ref{below bound ent} and of the following observation, holding for all $\mu,\rho \in \mathcal P_2(\R^d)$: 
\begin{equation*}
\sqrt{\int|x|^2\diff \rho(x)}=W_2(\rho,\delta_0)\leq W_2(\rho,\mu)+W_2(\mu,\delta_0),
\end{equation*}
we see that a below bound on the entropy along the entropic JKO scheme is equivalent to an upper bound on the Wasserstein distance between the initial measure and the iterates of the scheme. Such an estimate is proved in Proposition~\ref{prop:upper wasserstein alon ent JKo scheme} below.

Unfortunately, we are not aware of an analogue of the discrete E.V.I for the entropic JKO scheme. Therefore, it will not be possible to obtain an upper bound on the Wasserstein distance simply by adapting the previous argument to the entropic JKO scheme. However, we can straightforwardly adapt the argument in \cite[lemma 3.2.2]{ambrosio2005gradient} designed to obtain a bound on the Wasserstein distance along the JKO scheme for functionals for which the discrete E.V.I is not available. We will need the following consequence of Proposition~\ref{prop:lower bound on F by the wasserstein distance}:
\begin{proposition}\label{prop:lower bound on F by the second moment}
    Let $\mathcal F$ be $\lambda$-convex along generalized geodesics, $W_2$-l.s.c, and $\mu_0$ be such that $\F(\mu_0)$ and 
    $H(\mu_0)$ are finite. 
            \begin{itemize}
            \item If $\lambda<0$, for all $\alpha_0\in\R_+$, there exists $c_2>0$ only depending only on $\mu_0,\F,\alpha_0,d$ such that for all $\alpha\leq\alpha_0$ and $\rho \in \mathcal P_2(\R^d)$,
    \[
    \F(\rho)+\frac\alpha 2H(\rho)\geq\F(\mu_0)+\frac\alpha 2H(\mu_0)- c_2+\lambda W_2^2(\rho,\mu_0).
    \]
    \item If $\lambda=0$, there exists $c_2,c_3>0$ only depending on $\mu_0,\F,\alpha_0,d$ such that for all $\alpha\leq\alpha_0$ and $\rho \in \mathcal P_2(\R^d)$,
    \[
    \F(\rho)+\frac\alpha 2H(\rho)\geq\F(\mu_0)+\frac\alpha 2H(\mu_0)- c_2-c_3W_2(\rho,\mu_0).
    \]
        \end{itemize}
\end{proposition}
\begin{proof}
    Let $\rho\in\mathcal P_2(\R^d)$, then because of Proposition~\ref{prop:lower bound on F by the wasserstein distance}, there exist $c_0,c_1$ depending only on $\mu_0,\F$ such that for all $\rho\in\mathcal P_2(\R^d)$,
    \[
        \F(\rho)\geq \F(\mu_0) -c_0-c_1W_2(\rho,\mu_0)+\frac\lambda2W_2^2(\rho,\mu_0).
    \]
    Also, because of Proposition~\ref{prop:lower bound on F by the wasserstein distance} applied to $H$, there exist $\tilde c_0,\tilde c_1$ depending only on $\mu_0,d$ such that for all $\rho\in\mathcal P_2(\R^d)$,
    \[
        H(\rho)\geq H(\mu_0) -\tilde c_0-\tilde c_1W_2(\rho,\mu_0).
    \]
    Combining these two inequalities, we obtain that 
    \[
    \F(\rho)+\frac{\alpha}{2}H(\rho)\geq \F(\mu_0)+\frac{\alpha}{2}H(\mu_0) -\left(c_0+\frac{\alpha}{2}\tilde c_0\right)-\left(c_1+\frac{\alpha}{2}\tilde c_1\right)W_2(\rho,\mu_0)+\frac\lambda2W_2^2(\rho,\mu_0).
    \]
    If $\lambda=0$, there is nothing more to show. If $\lambda<0$, the inequality 
    \begin{equation*}
    \left(c_1+\frac{\alpha}{2}\tilde c_1\right)W_2(\rho,\mu_0)\leq \frac{1}{-2\lambda}\left(c_1+\frac{\alpha}{2}\tilde c_1\right)^2 -\frac{\lambda}{2}W_2^2(\rho,\mu_0)
    \end{equation*}
    concludes the proof.
\end{proof}
Now, we have all the ingredients to obtain a bound on the Wasserstein distance, and so to conclude the proof of Proposition~\ref{prop:lower bound H along entropic jko scheme}.
\begin{proposition}\label{prop:upper wasserstein alon ent JKo scheme}
    In the context of Proposition~\ref{prop:lower bound H along entropic jko scheme}, for all $\alpha_0>0$, there exists $c$ only depending on $d,\mu_0,\F,\alpha_0$ such that for all $t>0$, every $n>16\lambda_-t$ and for all $\alpha\leq \alpha_0$,
    \begin{itemize}
        \item if $\lambda=0$,
        \[
    W_2^2(J_{n,t/n}^\alpha(\mu_0),\mu_0)\leq  c \left(\sqrt{t}+t+K\alpha t^2\right),
    \]
    \item if $\lambda<0$,
    \[
    W_2^2(J_{n,t/n}^\alpha(\mu_0),\mu_0)\leq  c(1+K\alpha t) \left(\frac{n}{n-8|\lambda|t}\right)^n,
    \]
    \item if $\lambda>0$,
    \[
    W_2^2(J_{n,t/n}^\alpha(\mu_0),\mu_0)\leq  c.
    \]
    \end{itemize}
    \end{proposition}   
    \begin{proof}
         Consider any $0<\tau<\frac{1}{4\lambda_-}$ and any $k\in\N$. (Ultimately, we will obviously choose $\tau=\frac tn$ and $k\leq n$, but this is not necessary for the following and will simplify the notations.) We have for all $\alpha>0$,
        \[\frac{1}{2} W_2^2(J_{k,\tau}^\alpha(\mu_0),\mu_0) =\frac12\sum_{i=1}^k\left[W_2^2(J_{i,\tau}^\alpha(\mu_0),\mu_0)-W_2^2(J_{i-1,\tau}^\alpha(\mu_0),\mu_0)\right],\]
        while for all $1\leq i\leq k$, 
        $$\frac12 \left( W_2^2(J_{i,\tau}^\alpha(\mu_0),\mu_0)-W_2^2(J_{i-1,\tau}^\alpha(\mu_0),\mu_0)\right) \leq W_2(J_{i,\tau}^\alpha(\mu_0),J_{i-1,\tau}^\alpha(\mu_0))W_2(J_{i,\tau}^\alpha(\mu_0),\mu_0).$$
        Indeed, this inequality is trivial if $W_2(J_{i-1,\tau}^\alpha(\mu_0),\mu_0)\geq W_2(J_{i,\tau}^\alpha(\mu_0),\mu_0)$ and otherwise:
        \begin{align*}
        \frac12(W_2^2(J_{i,\tau}^\alpha(\mu_0),\mu_0)&-W_2^2(J_{i-1,\tau}^\alpha(\mu_0),\mu_0))\\&=\left(W_2(J_{i,\tau}^\alpha(\mu_0),\mu_0)-W_2(J_{i-1,\tau}^\alpha(\mu_0),\mu_0)\right)\left(\frac{W_2(J_{i,\tau}^\alpha(\mu_0),\mu_0)+W_2(J_{i-1,\tau}^\alpha(\mu_0),\mu_0)}{2}\right)\\
        &\leq W_2(J_{i,\tau}^\alpha(\mu_0),J_{i-1,\tau}^\alpha(\mu_0))W_2(J_{i,\tau}^\alpha(\mu_0),\mu_0).
        \end{align*}
        So we end up with
        \[
            \frac{1}{2}W_2^2(J_{k,\tau}^\alpha(\mu_0),\mu_0)\leq\sum_{i=1}^k W_2(J_{i,\tau}^\alpha(\mu_0),J_{i-1,\tau}^\alpha(\mu_0))W_2(J_{i,\tau}^\alpha(\mu_0),\mu_0).\]
            Using the Cauchy-Schwartz inequality, we obtain
            \begin{equation}\label{eq:second moment leq sum W_2^2}
            \frac{1}{2}W_2^2(J_{k,\tau}^\alpha(\mu_0),\mu_0)\leq \sqrt{\sum_{i=1}^k W_2^2(J_{i,\tau}^\alpha(\mu_0),J_{i-1,\tau}^\alpha(\mu_0))}\sqrt{\sum_{i=1}^kW_2^2(J_{i,\tau}^\alpha(\mu_0),\mu_0)}.
        \end{equation}
        Moreover, using Proposition~\ref{below bound sch by W-2}, for all $i = 1\dots, k$,
    \begin{equation}\label{eq: w_2 leq sch *alpha h}
    W_2^2(J_{i,\tau}^\alpha(\mu_0),J_{i-1,\tau}^\alpha(\mu_0))\leq 2\tau\left(\frac{\Sch^{\alpha\tau}(J_{i,\tau}^\alpha(\mu_0),J_{i-1,\tau}^\alpha(\mu_0))}{\tau}-\alpha\frac{H(J_{i,\tau}^\alpha(\mu_0))+H(J_{i-1,\tau}^\alpha(\mu_0))}{2}\right).
    \end{equation}
    By definition of $J_{i,\tau}^\alpha(\mu_0)$ in \eqref{Ent JKO}, testing as a competitor $J^\alpha_{i-1,\tau}(\mu_0) * \sigma_{\alpha \tau}$,
    \begin{equation}\label{eq:sch one step standard inequality}
    \frac{\Sch^{\alpha\tau}(J_{i,\tau}^\alpha(\mu_0),J_{i-1,\tau}^\alpha(\mu_0))}{\tau} +\F(J_{i-1,\tau}^\alpha(\mu_0))\leq \frac{\Sch(J_{i-1,\tau}^\alpha(\mu_0)*\sigma_{\alpha\tau},J_{i-1,\tau}^\alpha(\mu_0))}{\tau}+\F(J_{i-1,\tau}^\alpha(\mu_0)*\sigma_{\alpha\tau}).
    \end{equation}
    Since $\F$ verifies the third point of Hypothesis \ref{hypothesis}, then
    \begin{equation}\label{eq: hyp 1.3 apllycation}
    \F(J_{i-1,\tau}^\alpha(\mu_0)*\sigma_{\alpha\tau})\leq\F(J_{i-1,\tau}^\alpha(\mu_0))+K\frac{\alpha\tau}{2}.
    \end{equation}
    Finally, as an easy consequence of Proposition~\ref{Sch Benamou Brenier},
    \begin{equation}\label{eq:sch leq H}
    \frac{\Sch(J_{i-1,\tau}^\alpha(\mu_0)*\sigma_{\alpha\tau},J_{i-1,\tau}^\alpha(\mu_0))}{\tau}\leq \alpha H(J_{i-1,\tau}^\alpha(\mu_0))
    \end{equation}
    Gathering equations \eqref{eq: w_2 leq sch *alpha h}, \eqref{eq:sch one step standard inequality}, \eqref{eq: hyp 1.3 apllycation} and \eqref{eq:sch leq H}, we obtain:
    \begin{multline*}
     W_2^2(J_{i,\tau}^\alpha(\mu_0),J_{i-1,\tau}^\alpha(\mu_0)) \\ \leq 2\tau\left(\F(J_{i-1,\tau}^\alpha(\mu_0))-\F(J_{i,\tau}^\alpha(\mu_0))+\frac{\alpha}{2} H(J_{i-1,\tau}^\alpha(\mu_0))-\frac{\alpha}{2}H(J_{i,\tau}^\alpha(\mu_0))+K\frac{\alpha}{2}\tau\right).
    \end{multline*}
    Plugging this into equation~\eqref{eq:second moment leq sum W_2^2}, we obtain:
    \begin{multline*}
        \frac{1}{2}W_2^2(J_{k,\tau}^\alpha(\mu_0),\mu_0) \\
        \leq \sqrt{2\tau\left(\F(\mu_0)+\frac\alpha 2H(\mu_0)-\F(J_{k,\tau}^\alpha(\mu_0))-\frac \alpha2H(J_{k,\tau}^\alpha(\mu_0))+K\frac\alpha 2k\tau\right)}\sqrt{\sum_{i=1}^kW_2^2(J_{i,\tau}^\alpha(\mu_0),\mu_0)}.
    \end{multline*}
    From now on, we will need to distinguish the cases $\lambda=0$ and $\lambda<0$.\\
    \underline{\textbf{Case $\lambda=0$.}}
    Using Proposition~\ref{prop:lower bound on F by the second moment}, there exists $c_2,c_3$ only depending on $\mu_0,d,\F,\alpha_0$ such that for all $\alpha\leq\alpha_0$:
    \begin{equation*}
        \frac{1}{2}W_2^2(J_{k,\tau}^\alpha(\mu_0),\mu_0)\leq \sqrt{2\tau\left(c_2+c_3 W_2(J_{k,\tau}^\alpha(\mu_0),\mu_0)+K\frac\alpha2k\tau\right)}\sqrt{\sum_{i=1}^kW_2^2(J_{i,\tau}^\alpha(\mu_0),\mu_0)}.
    \end{equation*}
Now, consider the following lemma.
\begin{lemma}\label{lem:bound iterative sequence}
   Let $\tau\in\R_+$ Let $(w_n)_n\in\R^\N$ such that, $w_0=0$ and for all $n\in\N^*$, $$ w_n^4\leq \tau(1+w_n)\left(\sum_{k=1}^{n}w_k^2\right).$$
    Then for all $n\in\N$, $$w_n\leq \sqrt{n\tau} +n\tau.$$
\end{lemma}
Let us show how this lemma allows to conclude, and postpone its proof to the end of the section. Define $\tilde w_k=W_2(J_{k,\tau}^\alpha(\mu_0),\mu_0)$, $k \in \N^*$. Then, for all $k \in \N^*$, 
    \[
    \tilde w_k^4\leq \tau(8c_2+K4\alpha k\tau)+c_1\tilde w_k)\left(\sum_{i=1}^k \tilde w_i^2\right)
    \]
    Thus, calling $m_k=\max\left\{\sqrt{8c_2+K4\alpha k\tau},c_1\right\}$ and $w_k=\frac{\tilde w_k}{m_k}$, we get:
    \[
    w_k^4\leq \tau\left(\frac{8c_2+K4\alpha k\tau}{m_k^2}+\frac{c_1}{m_k}\frac{\tilde w_k}{m_k}\right)\left(\sum_{i=1}^k\frac{\tilde w_i^2}{m_k^2}\right),
    \]
    and thus,
    \[
    w_k^4\leq \tau\left(\frac{8c_2+K4\alpha k\tau}{m_k^2}+\frac{c_1}{m_k}w_k\right)\left(\sum_{i=1}^kw_i^2\right).
    \]
    Since $\frac{8c_2+K4\alpha k\tau}{m_k^2}\leq 1$ and $\frac{c_1}{m_k}\leq 1$, we obtain
     \[
    w_k^4\leq \tau\left(1+w_k\right)\left(\sum_{i=1}^kw_i^2\right).
    \]
    Then, applying Lemma \ref{lem:bound iterative sequence} we obtain $\tilde w_k\leq m_k(k\tau+\sqrt{k\tau})$, which concludes the proposition for $\lambda=0$.
    \\
    \underline{\textbf{Case $\lambda<0$.}}Using Proposition~\ref{prop:lower bound on F by the second moment}, there exists a constant $c_2$ only depending on $\mu_0,\F,\alpha_0,d$ such that for all $\alpha \leq \alpha_0$:
    \begin{equation}\label{eq:relation inductive second moment}
        \frac{1}{2}W_2^2(J_{k,\tau}^\alpha(\mu_0),\mu_0)\leq \sqrt{2\tau\left(c_2-\lambda W_2^2(J_{k,\tau}^\alpha(\mu_0),\mu_0)+K\frac\alpha2k\tau\right)}\sqrt{\sum_{i=1}^kW_2^2(J_{i,\tau}^\alpha(\mu_0),\mu_0)}.
    \end{equation}
    We will use the following lemma, already used in \cite{ambrosio2005gradient}:
    \begin{lemma}\label{lem:discrte Grownwall}
    Consider a non negative sequence $(w_n)_{n \in \N^*}$ and two positive numbers $C_0,C_1$ with $C_1<1$, such that for all $n \in \N^*$,
        \[
        w_n^2\leq C_0+C_1\sum_{k=1}^nw_k^2.
        \]
        Then for all $n \in \N^*$, $$w_n^2\leq C_0\left(\frac{1}{1-C_1}\right)^{n}.$$
    \end{lemma}

    Once again, we postpone the proof of this lemma to the end of this section.
    From equation~\eqref{eq:relation inductive second moment} and the inequality $ab\leq\frac{\varepsilon a} {4\tau}+\frac{\tau b}{\varepsilon}$, it follows that,
    \[
    \frac{1}{2}W_2^2(J_{n,\tau}^\alpha(\mu_0),\mu_0)\leq {\frac\varepsilon2\left(c_2+K\frac{\alpha}{2}k\tau-\lambda W_2^2(J_{n,\tau}^\alpha(\mu_0),\mu_0)\right)}+\frac\tau\varepsilon{\sum_{k=1}^nW_2^2(J_{k,\tau}^\alpha(\mu_0),\mu_0)}.
    \]
    Thus
    \[    (1+\varepsilon\lambda)W_2^2(J_{n,\tau}^\alpha(\mu_0),\mu_0)\leq \varepsilon\left(c_2+K\frac{\alpha}{2}k\tau\right)+\frac{2\tau}\varepsilon{\sum_{k=1}^nW_2^2(J_{k,\tau}^\alpha(\mu_0),\mu_0)},
    \]
    and hence
    \[
    W_2^2(J_{n,\tau}^\alpha(\mu_0),\mu_0)\leq {\frac\varepsilon{(1+\varepsilon\lambda)}\left(c_2+K\frac{\alpha}{2}k\tau\right)}+\frac{2\tau}{\varepsilon(1+\varepsilon\lambda)}{\sum_{k=1}^nW_2^2(J_{k,\tau}^\alpha(\mu_0),\mu_0)}.
    \]
    Taking $\varepsilon=-\frac1{2\lambda}=\frac{1}{2|\lambda|}$, we obtain that
    \[
    W_2^2(J_{n,\tau}^\alpha(\mu_0),\mu_0)\leq \frac1{|\lambda|} \left(c_2+K\frac{\alpha}{2}k\tau\right)+8|\lambda|\tau{\sum_{k=1}^nW_2^2(J_{k,\tau}^\alpha(\mu_0),\mu_0)}.
    \]
    We can then apply the Lemma~\ref{lem:discrte Grownwall} and obtain 
    \[    W_2^2(J_{n,\tau}^\alpha(\mu_0),\mu_0)\leq\frac1{|\lambda|} \left(c_2+K\frac{\alpha}{2}k\tau\right)\left(\frac{1}{1-8|\lambda|\tau}\right)^n.
    \]
   Taking $\tau=\frac{t}{n}$ concludes the proof for $\lambda<0$.
    \end{proof}
Along the proof, we used Lemma~\ref{lem:bound iterative sequence} and Lemma~\ref{lem:discrte Grownwall}, that we prove now.
\begin{proof}[Proof of Lemma~\ref{lem:bound iterative sequence}]
 We do the proof by induction. For $n=0$ the result is trivial. We assume now that the result holds for all $k\leq n-1$. Then using the induction hypothesis, we find
    \begin{align*}     
    \sum_{k=1}^{n-1}w_k^2&\leq\sum_{k=1}^{n-1}\left(k\tau+\sqrt{k\tau}\right)^2\leq \sum_{k=1}^{n-1}\left(n\tau+\sqrt{n\tau}\right)^2\\
    &\leq (n-1)(\sqrt{n\tau}+n\tau)^2.
    \end{align*}
    Substituting this into the inductive formula verified by $w_n$, yields to:
    \[
    w_n^4\leq\tau(1+w_n)\left(w_n^2+(n-1)(n\tau+\sqrt{n\tau})^2\right)= (n-1)\tau(n\tau+\sqrt{n\tau})^2+(n-1)\tau(n\tau+\sqrt{n\tau})^2w_n+\tau w_n^2+\tau w_n^3.
    \]
    Now, either $w_n\leq n\tau+\sqrt{n\tau}$ and there is nothing more to show, or
    \begin{align*}
    w_n&\leq (n-1)\tau(n\tau+\sqrt{n\tau})^2\frac{1}{w_n^3}+(n-1)\tau(n\tau+\sqrt{n\tau})^2\frac{1}{w_n^2}+\tau \frac{1}{w_n}+\tau\\
    &\leq \frac{(n-1)\tau}{n\tau+\sqrt{n\tau}} +(n-1)\tau+\frac{\tau}{n\tau+\sqrt{n\tau}}+\tau\\
    &=n\tau+\frac{n\tau}{n\tau+\sqrt{n\tau}}\leq n\tau+\sqrt{n\tau}.
    \end{align*}
    This shows that $w_n\leq n\tau+\sqrt{n\tau}$ and concludes the proof.
    \end{proof}

        \begin{proof}[Proof of Lemma~\ref{lem:discrte Grownwall}]
        Consider the sequence $(u_n)_{n \in \N^*}$ defined inductively for all $n \in \N^*$ by
        \[
        (1-C_1)u_n= C_0+C_1\sum_{k=1}^{n-1} u_k.
        \]
        Then an easy induction shows that for all $n \in \N^*$, $w_n^2\leq u_n$. Moreover, if we let $U_0=0$ and for all $n \in \N^*$, $U_n=\sum_{k=1}^nu_k$, then for all $n \in \N^*$, 
        \begin{equation*}
        U_n-U_{n-1}=u_n=C_0+C_1\sum_{k=1}^nu_k=C_0+C_1U_n.
        \end{equation*}
        Solving this iterative scheme yields:
        \[
U_n = \frac{C_0}{C_1}\left(\left(\frac{1}{1-C_1}\right)^n - 1\right).
\]
Thus, we can deduce the expression of $u_n$
\[
u_n=U_n-U_{n-1}=C_0\left(\frac{1}{1-C_1}\right)^{n}.
\] This concludes the proof of the lemma.
    \end{proof}
    
\section{Optimality in $\alpha$ of the inequalities}\label{Sec:optimality in alpha}
The purpose of this section is to investigate the optimality of the bounds obtained in Theorems \ref{the theorem} and \ref{bound continuous case}, and to explicit the link between the two. Our main conclusions are:
\begin{enumerate}
    \item The first inequalities in formulas \eqref{eq:bound continuous} and \eqref{eq:bound continuous lambda} are sharp. At least, there are models for which the inequalities are in fact equalities.
    \item Similarly, we will exhibit a model for which one of the inequalities obtained along the proof of Theorem~\ref{the theorem}, that we reproduce at equations \eqref{eq:precise bound lambda=0}, \eqref{eq:precise bound lambda neq 0} of Theorem~\ref{thm:precise bound} below, coincides with the first inequalities in~\eqref{eq:bound continuous} and~\eqref{eq:bound continuous lambda} up to a term going to $0$ when $\tau$ goes to $0$. Remarkably, this implies that the lack of optimality in $\alpha$ of Theorem \ref{the theorem} is neither due to our splitting argument~\eqref{eq:splitting (I) and (II)}, nor to the suboptimality of the competitor built in the proof of equation~\eqref{eq:comparison_costs}, nor of the use of squared discrete Grönwall lemma, Proposition~\ref{square discrete gronwall}.
    \item In fact, we expect these two bounds (the first inequalities in equation \eqref{eq:bound continuous} and \eqref{eq:bound continuous lambda} on the one hand, and equations \eqref{eq:precise bound lambda=0}, \eqref{eq:precise bound lambda neq 0} of Theorem~\ref{thm:precise bound} on the other hand) to always correspond to each other; letting $\tau\to0$ in equations \eqref{eq:precise bound lambda=0}, \eqref{eq:precise bound lambda neq 0}, we show that we recover formally the first inequalities of equations~\eqref{eq:bound continuous} and~\eqref{eq:bound continuous lambda} in the general case. 
\end{enumerate} 
Our toy model, set once for all in the whole section, consists in considering the very simple energy, corresponding to a parameter $\lambda \in \R$:
\begin{equation}
\label{eq:linear_F}
\F:\rho\in\mathcal P_2(\R^d)\rightarrow\int V\diff\rho, \quad V:x\in\R^d\rightarrow \lambda\frac{|x|^2}{2}.
\end{equation}
Note that this choice of $\mathcal F$ fulfills points (2) and (3) of Hypothesis~\ref{hypothesis} with the same value of $\lambda$ and $K = \lambda d$.

With this choice of $\F$ and centered gaussian initial conditions, we observe that the iterates of the JKO and entropic JKO schemes, as well as the solution of equations \eqref{eq:gradient_flow} and \eqref{eq:gradient_flow_reg}, which rewrite in this case
\begin{equation}\label{eq: gradient flows for V}
\partial_t\rho-\operatorname{div}(\rho\nabla V)=0\quad\text{and}\quad\partial_t\rho-\operatorname{div}(\rho\nabla V)=\frac{\alpha}{2}\Delta\rho
\end{equation}
respectively, remain centered gaussian. Furthermore, the variance of these Gaussians can be computed explicitly in each case, enabling us to prove our claims. 

In the next subsection, we state the evolution of the variance of the Gaussian solutions along the JKO and entropic JKO schemes, and of the corresponding limiting PDEs~\eqref{eq: gradient flows for V}. In Subsection \ref{optimality continuous level}, we prove the optimality of the first inequality in equation~\eqref{eq:bound continuous} and \eqref{eq:bound continuous lambda}. In Subsection \ref{Optimality in alpha at the discrete level}, we prove that, in our gaussian settings, equations \eqref{eq:precise bound lambda=0}, \eqref{eq:precise bound lambda neq 0} of Theorem~\ref{thm:precise bound} converge to the first inequality in equation~\eqref{eq:bound continuous} and~\eqref{eq:bound continuous lambda}. We conclude this section with Subsection~\ref{remark formal computation optimality in alpha}, where we remark that this convergence is in fact formally expected in general. 
\subsection{Preliminaries}
 In the computations to come, we will need a more convenient notation for the variance of our gaussian. So let us replace the term $\sigma_t$ in this section with the following definition.
\begin{definition}
    We note for all $s\in\R^d$, identifying a measure with its density with respect to the Lebesgue measure:
    \begin{equation}
    \label{eq:formula_gaussian}
    \mathcal N(s):=\frac{1}{\sqrt{2\pi s}^d}e^{-\frac{|x|^2}{2s}}=\sigma_s.
    \end{equation}
\end{definition}
The following quantities of interest are easy to compute.
\begin{proposition}\label{prop:computation fisher entropy wassertein gaussian}
    Let $s,u \in \R_+^*$ with $s \geq u$, $\rho=\mathcal N(s)$ and $\mu=\mathcal N(u)$. We have 
\begin{gather}\label{eq:expression fisher gaussian}
    \int_{\R^d}|\nabla\ln(\rho)|^2\diff\rho=\frac{d}{s},\\
   \label{eq:expression wasserstein gaussian}
        W_2(\rho,\mu)=\sqrt d\left(\sqrt{s}-\sqrt{u}\right),\\
    \label{eq:computaion entropy gaussian}
    H(\rho)=-\frac{d}{2}\Big(1 + \ln(2\pi s)\Big).
    \end{gather}
\end{proposition}
\begin{proof}
    \noindent \underline{First formula}. If $\rho=\mathcal N(s)$, then in in view of~\eqref{eq:formula_gaussian} $\ln(\rho)=-\frac{|x|^2}{2s}-\frac{d}{2}\ln(2\pi s)$ and $\nabla\ln(\rho)=-\frac{x}{s}$. Then, the Fisher information satisfies:
    \[
    \int |\nabla\ln(\rho)|^2\diff\rho=\int_{\R^d}\frac{|x|^2}{s^2}\diff\rho=\frac{ds}{s^2}=\frac{d}{s}.
    \]

\noindent \underline{Second formula}.
   Consider $T:x\mapsto\sqrt{\frac{u}{s}}x$. Then $T{_\#}\rho=\mu$ and $T$ is the gradient of the convex function $x\mapsto\sqrt{\frac{u}{s}}\frac{|x|^2}{2}$. Thus $T$ is the Brenier map, see Subsection~\ref{Wasser}. Therefore, the Wasserstein distance is equal to:
   \begin{align*}
   W_2(\rho,\mu)^2&=\int|x-T(x)|^2\diff\rho(x)=\int\left|x-\sqrt{\frac{u}{s}}x\right|^2\diff\rho=\left(1-\sqrt{\frac{u}{s}}\right)^2\int|x|^2\diff\rho\\
   &=\left(1-\sqrt{\frac{u}{s}}\right)^2ds=d\left(\sqrt s-\sqrt{u}\right)^2.
   \end{align*}

\noindent \underline{Third formula}.
    Once again $\ln(\rho)=-\frac{d}{2}\ln(2\pi s)-\frac{|x|^2}{2s}$. Therefore:
    \[
    H(\rho)=\int \left( -\frac{d}{2}\ln(2\pi s)-\frac{|x|^2}{2s}\right)\diff \rho(x)=-\frac{d}{2}\ln(2\pi s)-\frac{sd}{2s},   \]
    as anounced.
\end{proof}
As mentioned previously, it is possible to compute the solution of the PDEs and of the different schemes explicitly in this setting. In the case of the PDEs, the formulas write as follows.
\begin{proposition}\label{explicit computation continuous}
    Let $a \geq 0$, $\mu_0=\mathcal N(a)$ and $\alpha >0$. The solutions $\rho^0$ and $\rho^\alpha$ of equation~\eqref{eq: gradient flows for V} starting from $\mu_0$, satisfy for all $t \geq 0$:
    \begin{equation*}
 \rho^0(t)=
 \left\{ \begin{aligned}
 &\mathcal{N}\left(a \right), &&\mbox{if }\lambda = 0,\\
 &\mathcal{N}\left(a e^{-2\lambda t}\right),  && \mbox{if }\lambda \neq 0,
 \end{aligned} \right.
 \qquad \mbox{and} \qquad 
 \rho^\alpha(t) =
 \left\{ \begin{aligned}
 &\mathcal{N}\left(a +\alpha t\right), &&\mbox{if }\lambda = 0,\\
 &\mathcal{N}\left(a e^{-2\lambda t}+\alpha\frac{1-e^{-2\lambda t}}{2\lambda}\right),  && \mbox{if }\lambda \neq 0.
 \end{aligned} \right.
\end{equation*} 
\end{proposition}
We omit the proof of this proposition which only consists in plugging the different formulas into the PDEs~\eqref{eq: gradient flows for V}. The iterates of each schemes can also be computed, thanks to the following proposition.
\begin{proposition}\label{explicit computation discrete}
    Let $a\geq 0$, $\mu=\mathcal N(a)$, $\tau<\frac{1}{\lambda_-}$ and $\alpha >0$. Then, we have 
        \begin{equation*} 
        J_\tau^0(\mu)=  \left\{ \begin{aligned}
 &\mathcal{N}\left(a \right), &&\mbox{if }\lambda = 0,\\
 & \mathcal{N}\left(\frac{a}{(1+\lambda\tau)^2}\right),  && \mbox{if }\lambda \neq 0,
 \end{aligned} \right.
        \qquad \mbox{and} \qquad
            J_\tau^\alpha(\mu)=
             \left\{ \begin{aligned}
 &\mathcal{N}\left(a + \alpha \tau \right), &&\mbox{if }\lambda = 0,\\
 &  \mathcal{N}\left(\frac{a}{(1+\lambda\tau)^2}+\frac{\alpha\tau}{1+\lambda\tau}\right),  && \mbox{if }\lambda \neq 0.
 \end{aligned} \right.            
        \end{equation*}
        Consequently, the iterates satisfy for all $k \in \N$:
        \begin{gather*} 
        J_{k,\tau}^0(\mu)=  \left\{ \begin{aligned}
 &\mathcal{N}\left(a \right), &&\mbox{if }\lambda = 0,\\
 & \mathcal{N}\left(\frac{a}{(1+\lambda\tau)^{2k}}\right),  && \mbox{if }\lambda \neq 0,
 \end{aligned} \right.
      \\
            J_{k,\tau}^\alpha(\mu)=
             \left\{ \begin{aligned}
 &\mathcal{N}\left(a + k \alpha \tau \right), &&\mbox{if }\lambda = 0,\\
 &  \mathcal{N}\left(\frac{a}{(1+\lambda\tau)^{2k}}+\frac{\alpha}{\lambda}\left(1-\frac{1}{(1+\lambda\tau)^{2k}}\right)\frac{1+\lambda\tau}{2+\lambda\tau}\right),  && \mbox{if }\lambda \neq 0.
 \end{aligned} \right.            
        \end{gather*}
\end{proposition}
The formulas for the iterative scheme can easily be deduced from those obtained for one step. Therefore, we only prove the two first formulas.
\begin{proof}
First, we show the formula for the JKO scheme.
    We have:
    \begin{align*}
    \inf_{\rho\in\mathcal P_2(\R^d)}\frac{W_2^2(\mu,\rho)}{2\tau}+\lambda\int \frac{|x|^2}{2}\diff\rho&=\inf_{\substack{\gamma\in\mathcal P_2(\R^d\times\R^d)\\
\pi_1{_\#}\gamma=\mu}}\int \frac{|x-y|^2+\lambda\tau|y|^2}{2\tau}\diff\gamma(x,y)\\
&\geq \int\inf_{z\in\R^d} \frac{|x-z|^2+\lambda\tau|z|^2}{2\tau} \diff\mu(x).
    \end{align*}
    Since for all $x \in \R^d$, $\argmin_z \frac{|x-z|^2+\lambda\tau|z|^2}{2\tau}=\frac{x}{1+\lambda\tau}$, the last inequality is an equality if and only if for $\gamma$ almost every $(x,y)$, there holds $y=\frac{x}{1+\lambda\tau}$. Therefore, the only minimizer on the right-hand side of the first line is $\gamma=(I_d,\frac{1}{1+\lambda\tau}I_d){_\#}\mu$ . In particular, $J_\tau^0(\mu)=\frac{1}{1+\lambda\tau}I_d{_\#}\mu$, and so $J_\tau^0(\mu)=\mathcal N\left(\frac{a}{(1+\lambda\tau)^2}\right)$.
    
    Now, we prove the formula for the entropic JKO scheme. Here as well, $J^\alpha_\tau(\mu)$ is the second marginal of the minimizer of a minimization problem (see Definition~\ref{def:sch cost}), which is:
    \[   \inf_{\substack{\gamma\in\mathcal P_2(\R^d\times\R^d)\\
\pi_1{_\#}\gamma=\mu}} \alpha H(\gamma\|R_{\alpha\tau})+\frac{\lambda}{2}\int|y|^2\diff\gamma(x,y).\]
Using the additivity of the entropy, (see Proposition~\ref{prop:addivity of entropy}, slightly adapted to the case when $R$ is not a probability measure), for each $\gamma\in\mathcal{P}_2(\R^d\times\R^d)$ of first marginal $\mu$, calling $(\nu^x)_x$ the family, well defined for $\mu$ almost every $x$, obtained by disintegrating $\gamma$ with respect to the first projection, we find: 
\begin{align*}\alpha H(\gamma\|R_{\alpha\tau})+\frac{\lambda}{2}\int|y|^2\diff\gamma(x,y)&= \alpha H(\mu)+\int \left(H(\nu^x\|R_{\alpha\tau}^x)+\int\frac{\lambda}{2}|y|^2\diff\nu^x(y)\right)\diff\mu(x)\\
&\geq \alpha H(\mu)+\int \inf_{\rho\in\mathcal P_2(\R^d)}  \left\{H(\rho\|R_{\alpha\tau}^x)+\int\frac{\lambda}{2}|y|^2\diff\rho(y)\right\}\diff\mu(x)
    \end{align*}
    where for all $x,y \in \R^d$,
    $R_{\alpha\tau}^x(y)=\frac{1}{\sqrt{2\pi\alpha\tau}^d}\exp\left(-\frac{|x-y|^2}{2\alpha\tau}\right).$
    Let $x\in\R^d$, then
    \[
    \alpha H(\rho\|R_{\alpha\tau}^x)+\frac{\lambda}{2}\int|y|^2\diff\rho=\alpha\int\ln\left(\frac{\rho}{e^{-\lambda\frac{|y|^2}{2\alpha}}R_{\alpha\tau}^x}\right)\diff\rho.
    \]
By the Jensen inequality, this quantity is minimized for $\rho =: \rho^x $ of the form
\begin{equation*}
    \rho^x(y) = \frac{1}{Z^x} \exp\left(-\lambda\frac{|y|^2}{2\alpha}\right) R_{\alpha\tau}^x(y) = \frac{1}{Z^x} \exp\left(-\frac{(1+\lambda\tau)\left|y-\frac{x}{1+\lambda\tau}\right|^2}{2\alpha\tau}\right) , \qquad y \in \R^d,
\end{equation*}
where $Z^x$ is a normalizing constant allowed to change in each equality, and where we used the identity $|x-y|^2+\lambda\tau|y|^2=(1+\lambda\tau)\left|y-\frac{x}{1+\lambda\tau}\right|^2+\frac{\lambda\tau}{1+\lambda\tau}|x|^2$, holding for all $x,y \in \R^d$. Note that in the last identity, $Z^x$ is in fact independent of $x$.

 Thus, $\gamma$ is a minimizer if and only if for $\mu$ almost every $x$, $\nu^x=\rho^x$. In particular, $J^\alpha_\tau(\mu)=\int \rho^x\diff\mu(x)$. But for all $y \in \R^d$,
    \[
    J^\alpha_\tau(\mu)(y)=\frac{1}{Z}\int \exp\left(-\frac{(1+\lambda\tau)\left|y-\frac{x}{1+\lambda\tau}\right|^2}{2\alpha\tau}\right)\exp\left(-\frac{|x|^2}{2a}\right)\diff x,
    \]
    where $Z$ is a normalizing constant allowed to change in the following computations.
    Since for all $x,y\in \R^d$,
    \begin{equation*}
    a(1+\lambda\tau)\left|y-\frac{x}{1+\lambda\tau}\right|^2+\alpha\tau|x|^2=\left(\frac{a}{1+\lambda\tau}+\alpha\tau\right)\left|x-\frac{a(1+\lambda\tau)}{a+(1+\lambda\tau)\alpha\tau}y\right|^2+\frac{a(1+\lambda\tau)^2\alpha\tau}{a+(1+\lambda\tau)\alpha\tau}|y|^2,
    \end{equation*}
    we find
 \[
    J^\alpha_\tau(\mu)(y)=\frac{1}{Z}\int \exp\left(-\frac{(\frac{a}{1+\lambda\tau}+\alpha\tau)|x-\frac{a(1+\lambda\tau)}{a+(1+\lambda\tau)\alpha\tau}y|^2}{2\alpha\tau a}\right)\diff x\times\exp\left(-\frac{(1+\lambda\tau)^2}{2(a+(1+\lambda\tau)\alpha\tau)}|y|^2\right).
    \]
The integral in this last formula is independent of $y$, and hence:
    \[
    J_\tau^\alpha(\mu)(y)=\frac{1}{Z}\exp\left(-\frac{(1+\lambda\tau)^2}{2(a+(1+\lambda\tau)\alpha\tau)}|y|^2\right).
    \]
    That is, $J_\tau^\alpha(\mu)=\mathcal N\left(\frac{a}{(1+\lambda\tau)^2}+\frac{\alpha\tau}{1+\lambda\tau}\right)$, as announced.
\end{proof}
\subsection{Optimality in $\alpha$ at the continuous level}\label{optimality continuous level}

The purpose of this subsection is to compare the solutions of the two continuous equations~\eqref{eq: gradient flows for V}, when $V$ is defined as in~\eqref{eq:linear_F} for some $\lambda \in \R$ and $\alpha>0$, starting from the same initial measure $\mu_0=\mathcal N(a)$, where $a>0$.
\begin{proposition}
    Let $\mu_0=\mathcal N(a)$ and $\rho^0,\rho^\alpha$ be the associated solutions of~\eqref{eq: gradient flows for V} given by Proposition~\ref{explicit computation continuous}. We have for all $t \geq 0$:
    \[
    W_2(\rho^0(t),\rho^\alpha(t))=\frac{\alpha}{2}\int_0^t e^{\lambda (s-t)} \sqrt{\int |\nabla\ln(\rho^\alpha_s)|^2\mathrm{d}\rho^\alpha_s}\mathrm{d}s.
    \]
\end{proposition}
\begin{proof}
Let us compute both sides of this equality explicitly, treating the case $\lambda = 0$ and $\lambda \neq 0$ separately.

\noindent \underline{\textbf{Case $\lambda=0$}}. For all $s \geq 0$, using Proposition~\ref{explicit computation continuous} and equation~\eqref{eq:expression fisher gaussian} of Proposition~\ref{prop:computation fisher entropy wassertein gaussian}, we obtain
\begin{equation*}
\sqrt{\int|\nabla\ln(\rho_s^\alpha)|^2\diff\rho_s^\alpha}=\frac{\sqrt{d}}{\sqrt{a +\alpha s}},
\end{equation*}
and hence for all $t \geq 0$
\begin{equation*}
\frac{\alpha}{2}\int_0^t \sqrt{\int|\nabla\ln(\rho_s^\alpha)|^2\diff\rho_s^\alpha}\diff s
=\frac{\sqrt{d}\alpha}{2}\int_0^t \frac{1}{\sqrt{a +\alpha s}}\diff s = \sqrt{d}\left(\sqrt{a+\alpha t}-\sqrt{a}\right).
\end{equation*}
Using Proposition~\ref{explicit computation continuous} and equation~\eqref{eq:expression wasserstein gaussian} of Proposition~\ref{prop:computation fisher entropy wassertein gaussian}, for all $t \geq 0$, we observe that this quantity coincides with the Wasserstein distance between the solutions $
W_2(\rho_t^\alpha,\rho_t^0)$, 
as anounced.

\noindent \underline{\textbf{Case $\lambda\neq0$}}. This time, for all $s \geq 0$, we find
\begin{equation*}
\sqrt{\int|\nabla\ln(\rho_s^\alpha)|^2\diff\rho_s^\alpha}=\frac{\sqrt{d}}{\sqrt{a e^{-2\lambda s}+\alpha\frac{1-e^{-2\lambda s}}{2\lambda}}}.
\end{equation*}
Thus, for all $t\geq 0$, we have:
\begin{equation*}
J(t):=\int_0^t e^{\lambda s}\,\sqrt{\int|\nabla\ln(\rho_s^\alpha)|^2\diff\rho_s^\alpha}\diff s
=\sqrt{d}\int_0^t \frac{e^{\lambda s}}{\sqrt{a e^{-2\lambda s}+\alpha\frac{1-e^{-2\lambda s}}{2\lambda}}}\diff s.
\end{equation*}
As $\lambda\neq0$, changing the variables according to $u=e^{\lambda s}$ leads to
\[
J(t)=\frac{\sqrt{d}}{\lambda}\int_{1}^{e^{\lambda t}}
\frac{\diff u}{\sqrt{a u^{-2}+\alpha\frac{1-u^{-2}}{2\lambda}}}
=\frac{\sqrt{d}}{\lambda}\int_{1}^{e^{\lambda t}}\frac{u\diff u}{\sqrt{a +\alpha\frac{u^2-1}{2\lambda}}}.
\]
Changing once again the variables according to $w=a +\alpha\frac{u^2-1}{2\lambda}$, we end up with
\[
J(t)=\frac{\sqrt{d}}{\alpha}
\int_{a}^{a + \alpha \frac{e^{2\lambda t} - 1}{2 \lambda}}\frac{\diff w}{\sqrt{w}}
=\frac{2\sqrt{d}}{\alpha}\left(\sqrt{a+\alpha\frac{e^{2\lambda t}-1}{2\lambda}}-\sqrt{a}\right).
\]
All in all,
\begin{equation*}
\frac{\alpha}{2}\int_0^t e^{\lambda (s-t)} \sqrt{\int |\nabla\ln(\rho^\alpha_s)|^2\mathrm{d}\rho^\alpha_s}\mathrm{d}s= e^{-\lambda t} J(t) =\sqrt{d}\left(\sqrt{ae^{-2\lambda t}+\alpha\frac{1-e^{-2\lambda t}}{2\lambda}}-\sqrt{ae^{-2\lambda t}}\right).
\end{equation*}
Once again, using Proposition~\ref{explicit computation continuous} and equation~\eqref{eq:expression wasserstein gaussian} of Proposition~\ref{prop:computation fisher entropy wassertein gaussian},  we observe that this quantity coincides with the Wasserstein distance $W_2(\rho_t^\alpha,\rho_t^0)$ between the solutions at time $t \geq 0$,
and the result follows.
\end{proof}

\subsection{Optimality in $\alpha$ at the discrete level}\label{Optimality in alpha at the discrete level}
When proving Theorem~\ref{the theorem}, we always neglected the Fisher information term of Proposition~\ref{below bound sch by W-2}.
If we keep this term, if we do not use the third point of Hypothesis~\ref{hypothesis} and finally if we apply the Cauchy-Schwartz inequality in the estimate corresponding to~\eqref{eq:bound bf CS} only to the term which does not depend on $\alpha$, we obtain the following refinement of estimate~\eqref{eq: bond bet scheme after CS}.
\begin{theorem}[A more precise bound]\label{thm:precise bound}
Let $\mathcal{F}$ satisfy Hypothesis~\ref{hypothesis}. Let $\mu_0 \in \mathcal P_2(\R^d)$ be such that $\F(\mu_0) < + \infty$ and $H(\mu_0)< + \infty$. Let $n\geq0$ and $\tau<\frac{1}{\lambda_-}$. Then for all $k \in \N$, the iterates $J_{k,\tau}^0(\mu_0)$ and $J_{k,\tau}^\alpha(\mu_0)$ are well defined and satisfy:
\begin{itemize}
    \item if $\lambda=0$,
\begin{equation}\label{eq:precise bound lambda=0}
W_2\big(J_{n,\tau}^0(\mu_0),\, J_{n,\tau}^\alpha(\mu_0)\big)
\leq \sqrt{2\tau\big(\mathcal{F}(\mu_0) - \mathcal{F}(J_{n,\tau}^0(\mu_0))\big)}
   + \sum_{k=0}^{n-1} \sqrt{\tau}\,\sqrt{\mathcal{R}_k(\alpha,\tau)},
\end{equation}
\item if $\lambda\neq0$ and $\tau \leq\frac{1}{2\lambda_-} $, 
\begin{multline}\label{eq:precise bound lambda neq 0}
  W_2(J_{n,\tau}^0(\mu_0), J_{n,\tau}^\alpha(\mu_0)) \leq \sqrt2(1+4\lambda_-\tau)^{\frac32} (1-\lambda_-\tau)^{-n}\sqrt{\tau}\sqrt{\F(\mu_0)-\F(J_{n,\tau}^0(\mu_0))} \\ + \sum_{k=0}^{n-1}\left(1+\lambda\tau\right)^{k-n} \sqrt{\tau}\sqrt{\frac{\mathcal{R}_k(\alpha,\tau)}{1+\lambda\tau}},
\end{multline}
\end{itemize}
where in both cases,
\begin{multline}
\label{eq:def_Rk}
\mathcal{R}_k(\alpha,\tau)=2\F(J_\tau^0(J_{k,\tau}^\alpha(\mu_0))*\sigma_{\alpha\tau})-2\F(J_\tau^0(J_{k,\tau}^\alpha(\mu_0))) \\+\alpha \big[ H(J_{k,\tau}^\alpha(\mu_0))-H(J_{k+1,\tau}^\alpha(\mu_0))\big]-\frac{\alpha^2}{4}\int_0^\tau\int|\nabla\ln(\bar \rho^\alpha_{\tau}(t))|^2\diff\bar\rho^\alpha_{\tau}(t)\mathrm{d}t,
\end{multline}
and $\bar \rho^\alpha_{\tau}$ is the curve whose position at time $k\tau$ is $J_{k,\tau}^\alpha(\mu_0)$ for all $k \in \N$, and interpolating between these timesteps along the solutions of the dynamic Schr\"odinger problem, defined in Definition \ref{Sch Benamou Brenier}.
\end{theorem}

The main result of this subsection asserts that inequalities~\eqref{eq:precise bound lambda=0} and~\eqref{eq:precise bound lambda neq 0} are optimal in the following sense: in our toy model where $\mathcal F$ is given by~\eqref{eq:linear_F}, inequalities~\eqref{eq:precise bound lambda=0} and~\eqref{eq:precise bound lambda neq 0} are equalities up to a term converging to $0$ as $\tau$ to $0$. 

From now on, we assume that $\mathcal F$ is given by~\eqref{eq:linear_F}, and we fix a value of $\alpha >0$. We recall that for this $\mathcal F$, as soon as $\tau < 1/\lambda_-$, both schemes are well defined, and $\mathcal F$ satisfies the second and third point of Hypothesis~\ref{hypothesis}. In particular, this model falls in the framework of Theorem~\ref{thm:precise bound}, with the same value of $\lambda$ and $K = \lambda d$. The case where $\lambda = 0$ is trivial since the JKO scheme is stationary and the entropic JKO scheme follows the heat flow, so we focus on the case $\lambda \neq 0$. Our main result is:

\begin{theorem}\label{optimality discrete cororlary}
    Let $\mu_0=\mathcal N(a)$. For all $n \in \N$ and $t \geq 0$, we have
    \[
    W_2(J_{n,t/n}^0(\mu_0),J_{n,t/n}^\alpha(\mu_0))= \sum_{k=0}^{n-1}\left(1+\lambda\frac tn\right)^{k-n} \sqrt{\frac tn}\sqrt{\frac{\mathcal{R}_k(\alpha,\frac tn)}{1+\lambda\frac tn}} + \underset{n\to +\infty}{o}(1).
    \]   
\end{theorem}
\begin{proof}
    Our proof relies on the previous section. First, comparing Propositions~\ref{explicit computation continuous} and~\ref{explicit computation discrete}, we have for all $t \geq 0$ (fixed for the whole proof)
    \begin{equation*}
       \lim_{n \to + \infty} W_2(J_{n,t/n}^0(\mu_0),J_{n,t/n}^\alpha(\mu_0)) = W_2(\rho^0(t), \rho^\alpha(t)).
    \end{equation*}
    Therefore, in view of Proposition~\ref{optimality continuous level}, we just need to prove that
    \begin{equation}
    \label{eq:cv_Fisher}
        \lim_{n \to + \infty}\sum_{k=0}^{n-1}\left(1+\lambda\frac tn\right)^{k-n} \sqrt{\frac tn}\sqrt{\frac{\mathcal{R}_k(\alpha,\frac tn)}{1+\lambda\frac tn}} = \frac{\alpha}{2}\int_0^t e^{\lambda (s-t)} \sqrt{\int |\nabla\ln(\rho^\alpha_s)|^2\mathrm{d}\rho^\alpha_s}\mathrm{d}s.
    \end{equation}

    Let us artificially write the left-hand side as an integral. For all $n\in \N$ (sufficiently large to ensure that the schemes are well defined),
    \begin{equation*}
    \sum_{k=0}^{n-1}\left(1+\lambda\frac tn\right)^{k-n} \sqrt{\frac tn}\sqrt{\frac{\mathcal{R}_k(\alpha,\frac tn)}{1+\lambda\frac tn}} = \frac{1}{\sqrt{1 + \lambda\frac{t}{n}}}\int_0^t \left( 1 + \lambda \frac{t}{n} \right)^{ \lfloor\frac{sn}{t} \rfloor- n} \sqrt{\frac{\mathcal{R}_{\lfloor\frac{sn}{t} \rfloor}(\alpha,\frac tn)}{t/n}}\diff s.
    \end{equation*}
    As $n \to + \infty$, $\sqrt{1 + \lambda t/n} \to 1$, and for all $s \in [0,t]$ the quantity
    \begin{equation*}
        \left( 1 + \lambda \frac{t}{n} \right)^{ \lfloor\frac{sn}{t} \rfloor- n} 
    \end{equation*}
    is bounded uniformly in $s \in [0,t]$ and $n$ sufficiently large, and converges pointwise towards $e^{\lambda(s-t)}$. Hence, to prove~\eqref{eq:cv_Fisher}, by the dominated convergence theorem, we just need to prove that for all $s \in [0,t]$,
    \begin{equation}
    \label{eq:cv_to_prove_opt}
       \lim_{n \to + \infty}  \sqrt{\frac{\mathcal{R}_{\lfloor\frac{sn}{t} \rfloor}(\alpha,\frac tn)}{t/n}} = \frac{\alpha}{2} \sqrt{\int |\nabla\ln(\rho^\alpha_s)|^2\mathrm{d}\rho^\alpha_s},
    \end{equation}
and that the quantity in the left-hand side is bounded uniformly in $s \in [0,t]$ and $n$ sufficiently large. Otherwise stated, we just need to prove that, 
\begin{equation*}
    \lim_{\tau \to 0} \frac{\mathcal{R}_{\lfloor\frac{s}{\tau} \rfloor}(\alpha,\tau)}{\tau} = \frac{\alpha^2}{4} \int |\nabla\ln(\rho^\alpha_s)|^2\mathrm{d}\rho^\alpha_s,
\end{equation*}
and that the quantity in the left-hand side is bounded uniformly in $s \in [0, t]$ and $\tau$ sufficiently small.

    To that aim, we first prove the following identity, holding for all $\tau < 1/\lambda_-$ and $k \in \N$, providing a link between the quantities $\mathcal R_k$ defined in~\eqref{eq:def_Rk} and integrals in time of the Fisher information:
    \begin{equation}
    \label{eq:identity_Rk}
\mathcal{R}_k(\alpha, \tau ) = \frac{\alpha^2}{4}\int_{k\tau}^{(k+1)\tau}\int|\nabla\ln(\bar \rho^\alpha_{\tau}(\theta))|^2\diff\bar\rho^\alpha_{\tau}(\theta)\mathrm{d}\theta + d \alpha \lambda \tau - d\alpha \ln \left(1 + \lambda \tau\right),
    \end{equation}
    where $\bar \rho^\alpha_{\tau}$ is the curve whose value at time $k\tau$ is $J^\alpha_{k,\tau}(\mu_0)$, and interpolating between these timepoints along the solution of the dynamic formulation of the Schr\"odinger problem, as in Proposition~\ref{Sch Benamou Brenier}.
    
     Let $k \in \N$ and $\tau < 1/\lambda_-$. First, we have
    \begin{equation}
    \label{eq:diff_F}
        2\F(J_\tau^0(J_{k,\tau}^\alpha(\mu_0))*\sigma_{\alpha\tau})-2\F(J_\tau^0(J_{k,\tau}^\alpha(\mu_0)))=\lambda d\alpha \tau.
    \end{equation}
    Indeed, in view of Propositions~\ref{explicit computation continuous} and~\ref{explicit computation discrete}, $J_\tau^0(J_{k,\tau}^\alpha(\mu_0))$ is a centered Gaussian. Call $\tilde a$ its variance. Then $J_{k,\tau}^\alpha(\mu_0))*\sigma_{\alpha\tau}$ is the centered Gaussian of variance $\bar a + \alpha \tau$. Our claim therefore follows from the definition~\eqref{eq:linear_F} of $\mathcal F$. 
Second, it is well known that between the times $k \tau$ and $(k+1)\tau$, $\bar \rho_{\tau}^\alpha$ solves
\begin{equation*}
   \left\{
   \begin{gathered}
   \partial_\theta \bar \rho^\alpha_\tau + \mathrm{div}(\bar \rho^\alpha_\tau \nabla \varphi^\alpha_\tau) = \frac{\alpha}{2} \Delta \bar \rho^\alpha_\tau ,\\
   \bar \rho^\alpha_\tau|_{\theta= k\tau} = J^\alpha_{k,\tau}(\mu_0),
   \end{gathered}
   \right.
\quad \mbox{where } \varphi^\alpha_\tau \mbox{ solves} \quad
   \left\{
   \begin{gathered}
   \partial_\theta  \varphi^\alpha_\tau + \frac{1}{2}|\nabla  \varphi^\alpha_\tau|^2 + \frac{\alpha}{2} \Delta \varphi^\alpha_\tau = 0 ,\\
    \varphi^\alpha_\tau|_{\theta= (k+1)\tau} = - \frac{\lambda}{2}|x|^2.
   \end{gathered}
   \right.
\end{equation*}
(The PDEs are the optimality conditions of the dynamic Schr\"odinger problem, and the terminal condition on $\bar \varphi^\alpha_\tau$ is the optimality condition of the entropic JKO scheme, see~\cite{baradat2025usingsinkhornjkoscheme}.) These equations are solved if and only if for all $\theta \in [k\tau, (k+1)\tau]$,
\begin{equation}
\label{eq:explicit_rho_phi}
\bar \rho^\alpha_\tau(\theta) = \mathcal N(\bar a_\theta) \qquad \mbox{and} \qquad \varphi(\theta,x) = d_\theta - \frac{\lambda_\theta}{2}|x|^2, \quad x \in \R^d,
\end{equation}
for some parameters $(\bar a_\theta)$, $(\lambda_\theta)$ and $(d_\theta)$ solving the following ODEs:
\begin{equation}
    \label{eq:ODEs}
    \left\{ 
    \begin{gathered}
\dot {\bar a}_\theta = \alpha - 2 \lambda_\theta \bar a_\theta, \\
\bar a_{k\tau} = a^\tau_k ,
\end{gathered}
    \right.
    \qquad \left\{
\begin{gathered}
    \dot \lambda_\theta = \lambda_\theta^2\\
    \lambda_{(k+1)\tau} = \lambda,
\end{gathered}
    \right.\qquad 
    \left\{
    \begin{gathered}
    \dot d_\theta = \frac{\alpha d}{2} \lambda_\theta,\\
    d_\tau = 0,
    \end{gathered}
    \right.
\end{equation}
and where in view of Proposition~\ref{explicit computation discrete},
\begin{equation*}
    a^\tau_k=\frac{a}{(1+\lambda\tau)^{2k}}+\left(1-\frac{1}{(1+\lambda\tau)^{2k}}\right)\left(\frac{\alpha}{\lambda}\frac{1+\lambda\tau}{2+\lambda\tau}\right).
\end{equation*}
With these equations and notations at hand, quick computations ensure
\begin{align}
   \notag \alpha \big[ H(J_{k,\tau}^\alpha(\mu_0))-H(J_{k+1,\tau}^\alpha(\mu_0))\big] &= - \alpha \int_{k\tau}^{(k+1) \tau} \frac{\diff }{\diff \theta} H(\bar \rho^\alpha_\tau(\theta)) \diff \theta\\
    \notag &= \alpha \int_{k\tau}^{(k+1) \tau} \frac{\diff }{\diff \theta} \Big( \frac{d}{2}(1 + \ln (2 \pi \bar a_\theta)) \Big) \diff \theta\\
    \notag &= \frac{\alpha d}{2} \int_{k\tau}^{(k+1)\tau} \frac{\dot \bar{a}_\theta}{\bar a_\theta} \diff \theta\\
    \notag &=  \frac{\alpha^2}{2} \int_{k\tau}^{(k+1)\tau} \frac{d}{\bar a_\theta} \diff \theta - \alpha d \int_{k\tau}^{(k+1)\tau} \lambda_\theta \diff \theta\\
    \label{eq:diff_H}&= \frac{\alpha^2}{2} \int_{k\tau}^{(k+1)\tau} \int |\nabla \ln \bar \rho^\alpha_\tau(\theta)|^2\diff \bar \rho^\alpha_\tau(\theta)\diff \theta - \alpha d \ln (1+\lambda \tau),
\end{align}
where we used formulas~\eqref{eq:computaion entropy gaussian}, \eqref{eq:expression fisher gaussian}, \eqref{eq:explicit_rho_phi} and~\eqref{eq:ODEs}. In particular, we used the following consequence on the ODE solved by $(\lambda_\theta)$ in~\eqref{eq:ODEs}:
\begin{equation}
\label{eq:integral_lambda}
\int_{k\tau}^{(k+1)\tau} \lambda_\theta \diff \theta = \ln(1+\lambda \tau).
\end{equation}
Formula~\eqref{eq:identity_Rk} follows from plugging~\eqref{eq:diff_H} and~\eqref{eq:diff_F} in the definition~\eqref{eq:def_Rk} of $\mathcal R_k$.

    Comparing~\eqref{eq:cv_to_prove_opt} and~\eqref{eq:identity_Rk}, we see that it remains to prove that for all $s \in [0,t]$,
    \begin{equation}
    \label{eq:cv_fisher_time_s}
        \lim_{\tau \to 0} \frac{1}{\tau} \int_{\lfloor\frac{s}{\tau} \rfloor\tau}^{(\lfloor\frac{s}{\tau} \rfloor+1)\tau}\int\left|\nabla\ln(\bar \rho^\alpha_{\tau}(\theta))\right|^2\diff\bar\rho^\alpha_{\tau}(\theta) \diff \theta = \int |\nabla\ln(\rho^\alpha_s)|^2\mathrm{d}\rho^\alpha_s,
    \end{equation}
    and that the quantity in the left-hand side is bounded uniformly in $s \in [0,t]$ and $\tau$ sufficiently small. On the right-hand side of~\eqref{eq:cv_fisher_time_s}, Propositions~\ref{explicit computation continuous} and~\ref{prop:computation fisher entropy wassertein gaussian} imply for all $s \in [0,t]$:
    \begin{equation*}
        \int |\nabla\ln(\rho^\alpha_s)|^2\mathrm{d}\rho^\alpha_s = \frac{d}{a_s}, \qquad \mbox{where} \qquad a_s := ae^{-2\lambda s} + \alpha \frac{1-e^{-2\lambda s}}{2\lambda}.
    \end{equation*}
    Concerning the left-hand side of~\eqref{eq:cv_fisher_time_s}, for all $k \in \N$, using the notations of~\eqref{eq:explicit_rho_phi} and~\eqref{eq:ODEs},
    \begin{equation*}
        \frac{1}{\tau} \int_{k\tau}^{(k+1)\tau}\int\left|\nabla\ln(\bar \rho^\alpha_{\tau}(\theta))\right|^2\diff\bar\rho^\alpha_{\tau}(\theta) \diff \theta = \frac{d}{\tau} \int_{k\tau}^{(k+1)\tau} \frac{\diff \theta}{\bar a_\theta}.
    \end{equation*}
    But for all $ \theta \geq 0$, in view of~\eqref{eq:ODEs},
    \begin{equation*}
        \frac{1}{\bar a_\theta} = \frac{1}{\alpha} \frac{\diff }{\diff \theta} \ln \bar a_\theta + \frac{2}{\alpha} \lambda_\theta.
    \end{equation*}
    Therefore, using~\eqref{eq:integral_lambda}, we find
    \begin{align*}
         \frac{1}{\tau} \int_{k\tau}^{(k+1)\tau}\int\left|\nabla\ln(\bar \rho^\alpha_{\tau}(\theta))\right|^2\diff\bar\rho^\alpha_{\tau}(\theta) \diff \theta &= \frac{d}{\tau \alpha} \left( \ln \frac{a^\tau_{k+1}}{a^\tau_k} + 2 \ln (1 + \lambda \tau)\right)\\
         &= \frac{d}{\tau \alpha} \ln \left( \frac{(1 + \lambda \tau)^2a_{k+1}^\tau}{a_k^\tau}\right)\\
         &= \frac{d}{\tau \alpha} \ln \left( 1 + \frac{\alpha \tau (1 + \lambda \tau)}{a_k^\tau} \right),
    \end{align*}
    where in the last line, we used the following induction relation stated in Proposition~\ref{explicit computation discrete}:
    \begin{equation*}
        a_{k+1}^\tau = \frac{a_k^\tau}{(1 + \lambda\tau)^2} + \frac{\alpha \tau}{1 + \lambda \tau}.
    \end{equation*}
    Therefore, expanding the logarithm, the conclusion follows from the easy fact that for all $s \in [0,t]$,
    \begin{equation*}
        \lim_{\tau \to 0}a_{\lfloor \frac{s}{\tau} \rfloor}^\tau = a_s,
    \end{equation*}
    and that the left-hand side is bounded from below uniformly in $s \in [0,t]$ and $\tau$ sufficiently small. 
\end{proof}

\subsection{A formal remark on the optimality in $\alpha$}\label{remark formal computation optimality in alpha}
As we saw the previous section, our proof of optimality in our toy model relies on the fact that for this model, the quantity:
        \begin{equation}
        \label{eq:dif_F_H}
        \alpha \big[ H(J_{k,\tau}^\alpha(\mu_0))-H(J_{k+1,\tau}^\alpha(\mu_0))\big]+2\F(J_\tau^0(J_{k,\tau}^\alpha(\mu_0))*\sigma_{\alpha\tau})-2\F(J_\tau^0(J_{k,\tau}^\alpha(\mu_0))),
        \end{equation}
which appears in the definition~\eqref{eq:def_Rk} of $\mathcal R_k$ in Theorem~\ref{thm:precise bound}, is equal to
        \begin{equation}
        \label{eq:Fischer+O}
        \frac{\alpha^2}{2}\int_{k\tau}^{(k+1)\tau}\int|\nabla\ln(\bar \rho^\alpha_{\tau}(s))|^2\diff\bar\rho^\alpha_{\tau}(s)\mathrm{d}s+O_{\tau \to 0}(\tau^2).\end{equation}
Indeed, this fact implied
\begin{equation*}
\mathcal R_k(\alpha, \tau) = \frac{\alpha^2}{4}\int_{k\tau}^{(k+1)\tau}\int|\nabla\ln(\bar \rho^\alpha_{\tau}(s))|^2\diff\bar\rho^\alpha_{\tau}(s)\mathrm{d}s+O_{\tau \to 0}(\tau^2),
\end{equation*}
so that the right hand side in~\eqref{eq:precise bound lambda=0} and~\eqref{eq:precise bound lambda neq 0} are reminiscent of the first inequalities of~\eqref{eq:bound continuous} and~\eqref{eq:bound continuous lambda}.
        
        The fact that the quantity in \eqref{eq:dif_F_H} is~\eqref{eq:Fischer+O} can be obtained formally.
        Indeed, using Proposition~\ref{Sch Benamou Brenier} there exists a pair $(\rho_t,\varphi_t)$ such that $\rho_0=J_{k,\tau}^\alpha(\mu_0)$, $\rho_\tau=J_{k+1,\tau}^\alpha(\mu_0)$ and $\partial_t\rho+\operatorname{div}(\rho\nabla\varphi)=\frac{\alpha}{2}\Delta\rho$. Then, computing the derivative of the entropy along this interpolation, we obtain that:
        \[
        \frac{\diff}{\diff s}H(\rho_s)=\int \ln(\rho_s)\partial_s\rho_s=-\int \ln(\rho_s)\operatorname{div}(\rho_s\nabla\varphi_s)+\frac{\alpha}{2}\int \ln(\rho_s)\Delta\rho_s,
        \]
        and then, integrating by parts,
        \[
        \frac{\diff}{\diff s}H(\rho_s)=\int \nabla\varphi_s\cdot\nabla\ln(\rho_s)\rho_s-\frac{\alpha}{2}\int \nabla\ln(\rho_s)\cdot\nabla\rho_s=\int \nabla\varphi_s\cdot\nabla\rho_s-\frac{\alpha}{2}\int |\nabla\ln(\rho_s)|^2\diff\rho_s.
        \]
        Therefore, integrating between times $0$ and $\tau$, we find
        \begin{equation*}
            \alpha \big[ H(J_{k,\tau}^\alpha(\mu_0))-H(J_{k+1,\tau}^\alpha(\mu_0))\big] = \frac{\alpha^2}{2} \int_0^\tau \int  |\nabla\ln(\rho_s)|^2\diff\rho_s \diff s - \alpha \int_0^\tau \int \nabla\rho_s\cdot\nabla\varphi_s \diff s. 
        \end{equation*}
        On the other hand, concerning the second term, deriving along the heat flow, we find:
        \[        2\F(J_\tau^0(J_{k,\tau}^\alpha(\mu_0))*\sigma_{\alpha\tau})-2\F(J_\tau^0(J_{k,\tau}^\alpha(\mu_0)))=\alpha\int_0^\tau \Delta \frac{\delta \F}{\delta \rho}(J_\tau^0(J_{k,\tau}^\alpha(\mu_0))*\sigma_{\alpha s})\diff (J_\tau^0(J_{k,\tau}^\alpha(\mu_0))*\sigma_{\alpha s}) \diff s,
        \]
        so that integrating by parts,
        \[        2\F(J_\tau^0(J_{k,\tau}^\alpha(\mu_0))*\sigma_{\alpha\tau})-2\F(J_\tau^0(J_{k,\tau}^\alpha(\mu_0)))=-\alpha\int_0^\tau \nabla \frac{\delta \F}{\delta \rho}(J_\tau^0(J_{k,\tau}^\alpha(\mu_0))*\sigma_{\alpha s})\cdot\nabla (J_\tau^0(J_{k,\tau}^\alpha(\mu_0))*\sigma_{\alpha s})\diff s.
        \]
        Combining both identities, we obtain that the quantity in~\eqref{eq:dif_F_H} equals
        \begin{equation*}
        \frac{\alpha^2}{2}\int_0^\tau\int |\nabla\ln(\rho_s)|^2\diff\rho_s\diff s -\alpha\int_0^\tau \hspace{-5pt}\int \left[\nabla\varphi_s \cdot\nabla\rho_s+ \nabla \frac{\delta \F}{\delta \rho}(J_\tau^0(J_{k,\tau}^\alpha(\mu_0))*\sigma_{\alpha s})\cdot\nabla (J_\tau^0(J_{k,\tau}^\alpha(\mu_0))*\sigma_{\alpha s})\right] \diff s.
        \end{equation*}
        Moreover, the Euler Lagrange equation of \eqref{Ent JKO} is, see \cite{baradat2025usingsinkhornjkoscheme}, 
        \[
        \varphi_\tau=-\frac{\delta\F}{\delta\rho}(J_{k+1,\tau}^\alpha(\mu_0)).
        \]
        Hence, since $J_\tau^0(J_{k,\tau}^\alpha(\mu_0))*\sigma_{\alpha\tau}$ should be close to $J_{k+1,\tau}^\alpha(\mu_0)$, we expect to have for all $s \in [0,\tau]$
        \begin{equation*}
        \nabla \varphi_s \approx - \nabla \frac{\delta \F}{\delta \rho}(J_\tau^0(J_{k,\tau}^\alpha(\mu_0))*\sigma_{\alpha s}).
        \end{equation*}
        Lastly, all the densities are close to each other, as they are all close to $J^\alpha_{k,\tau}(\mu_0)$. 
        
        All in all, we expect the last integral above to be at least $o(\tau)$ as $\tau \to 0$, and maybe even a $O(\tau^2)$, as announced. The crucial point to prove this asymptotics rigorously is to establish the convergence of the integral in time of the Fischer information of the different curves involved at $\tau >0$ towards the one of the limiting curve. This convergence is necessary to compare the right hand side in~\eqref{eq:precise bound lambda=0} and~\eqref{eq:precise bound lambda neq 0} with~\eqref{eq:bound continuous} and~\eqref{eq:bound continuous lambda}. We have mathematical reasons to believe that it would be sufficient as well.
    \appendix
    \section*{Appendix}
    The purpose of this Appendix is to first prove Theorem~\ref{bound continuous case}, and then Proposition~\ref{prop:lower bound along JKO implies wlog F lower bounded} that we used in Section~\ref{sec:proof_cor}, see Remark~\ref{rem:below_bound_F}. During the proof of Theorem~1.4 we also establish that under Hypothesis~\ref{hypothesis}, the entropy is increasing at most linearly along the JKO scheme. We used this fact in Subsection~\ref{general cnx geo explain} to ensure that the densities are always absolutely continuous with respect to the Lebesgue measure.

    \section{Proof of Theorem~\ref{bound continuous case}}
    
\textbf{Sketch of proof:} To prove Theorem~\ref{bound continuous case}, we must establish two inequalities. First, we need to show a bound on the Fisher information, and second, we need to show that the Wasserstein distance between our gradient flows is smaller than the Fisher information.
\begin{itemize}
    \item For the first one, we will first establish the inequality at the JKO level, and then take the limit using the lower semicontinuity of the entropy and the Fisher information. This inequality will also ensure that the Fisher information is finite.
    \item For the second one, we will compute the derivative of the square of the Wasserstein distance between our gradient flows by differentiating the dual formulation of the Wasserstein distance (see Subsection~\ref{subsec:wasserstein dual formulation}) thanks to the envelope theorem and the PDEs that our densities solve. The $\lambda$-convexity of $\F$  implies that some terms should cancel out, leaving only those that can be bounded by the Fisher information.
\end{itemize}
We start by proving the estimate for Fisher information by showing an analogous inequality for one step of the JKO scheme for the functional $\F+\frac{\alpha}{2}H$. We will only use the convexity on $\F$ to ensure the existence of minimizers at each stage of the JKO scheme, and to guarantee that the scheme converges towards the gradient flow.
\begin{proposition}\label{prop:linear increasing of entropy along JKO alpha neq0}
    If $\F$ satisfies Hypothesis~\ref{hypothesis}, then for every $\tau<\frac{1}{\lambda_-}$, for every $\alpha\geq 0$, for every $\mu$ with $H(\mu)< + \infty$ and $\F(\mu)<+\infty$, then
    \begin{equation} \tag{JKO+H}\label{eq:JKO+H}\mathcal{J}_\tau^\alpha(\mu):=\argmin_{\rho\in\mathcal{P}_2(\R^d)}\left\{\frac{W_2^2(\mu,\rho)}{2\tau}+\F(\rho)+\frac{\alpha}{2}H(\rho)\right\}
    \end{equation}
    is well defined. Moreover,
    \begin{equation*}
        H(\mathcal{J}_\tau^\alpha(\mu))-H(\mu)+\frac{\alpha\tau}{2}\int|\nabla\ln(\mathcal{J}_\tau^\alpha(\mu))|^2\diff\mathcal{J}_\tau^\alpha(\mu)\leq K\tau.
    \end{equation*}
    In particular $H(\mathcal J^\alpha_\tau(\mu))<+\infty$.     
\end{proposition}
\begin{proof}
    Since $\F+\frac{\alpha}{2}H$ is $\lambda$-convex along generalized geodesics $\mathcal{J}_\tau^\alpha(\mu)$ is well defined as a consequence of Theorem~\ref{welldefined JKO}. By optimality of $\mathcal{J}_\tau^\alpha(\mu)$ in equation~\eqref{eq:JKO+H}, for all $s>0$,
    \[
    \frac{W_2^2(\mathcal{J}_\tau^\alpha(\mu),\mu)}{2\tau}+\F(\mathcal{J}_\tau^\alpha(\mu))+\frac{\alpha}{2}H(\mathcal{J}_\tau^\alpha(\mu))\leq \frac{W_2^2(\mathcal{J}_\tau^\alpha(\mu)*\sigma_s,\mu)}{2\tau}+\F(\mathcal{J}_\tau^\alpha(\mu)*\sigma_s)+\frac{\alpha}{2}H(\mathcal{J}_\tau^\alpha(\mu)*\sigma_s),
    \]
    so that 
    \begin{equation}\label{eq:optimality for entropy upper boundalpha neq0}
    \frac{W_2^2(\mathcal{J}_\tau^\alpha(\mu),\mu)}{2s}-\frac{W_2^2(\mathcal{J}_\tau^\alpha(\mu)*\sigma_s,\mu)}{2s}\leq\frac{\tau}{s}(\F(\mathcal{J}_\tau^\alpha(\mu)*\sigma_s)-\F(\mathcal{J}_\tau^\alpha(\mu))+\frac{\alpha\tau}{2s}(H(\mathcal{J}_\tau^\alpha(\mu)*\sigma_s)-H(\mathcal{J}_\tau^\alpha(\mu)).
    \end{equation}
    On the one hand, the third point of Hypothesis~\ref{hypothesis} implies:
    \begin{equation}\label{eq:hypothesis 3 for upper bound entropyalpha neq0}
    \frac{\tau}{s}(\F(\mathcal J_\tau^0(\mu)*\sigma_s)-\F(\mathcal J_\tau^0(\mu))\leq K\frac{\tau}{2}.
    \end{equation}
    On the other hand, as $u \mapsto \rho*\sigma_u$ is the gradient flow of $\frac{1}{2}H$, which is convex along generalized geodesics, the E.V.I.\ of Theorem \ref{E.V.I} provides for all $u\geq 0$:
    \[
    \frac12\frac{\diff}{\diff u}W_2^2(\mathcal J_\tau^0(\mu)*\sigma_u,\mu)\leq \frac12H(\mu)-\frac12H(\mathcal J_\tau^0(\mu)*\sigma_u).
    \]
    Therefore,
    \begin{equation}
        \label{eq:e.v.i for upper bound entropyalpha neq0}
    \begin{aligned}
    \frac{W_2^2(\mathcal{J}_\tau^\alpha(\mu),\mu)}{2s}-\frac{W_2^2(\mathcal{J}_\tau^\alpha(\mu)*\sigma_s,\mu)}{2s}&=-\int_0^s\frac1{2s}\frac{\diff}{\diff u}W_2^2(\mathcal{J}_\tau^\alpha(\mu)*\sigma_u,\mu)\\
    &\geq \frac1{2s}\int_0^sH(\mathcal{J}_\tau^\alpha(\mu)*\sigma_u)\diff u-\frac12H(\mu).
    \end{aligned}    
    \end{equation}
    Combining the equations \eqref{eq:optimality for entropy upper boundalpha neq0}, \eqref{eq:hypothesis 3 for upper bound entropyalpha neq0} and \eqref{eq:e.v.i for upper bound entropyalpha neq0} we obtain for all $s\in\R_+$:
    \begin{equation}\label{eq:combined eq for upper bound entropyalpha neq0}    \frac1{2s}\int_0^sH(\mathcal{J}_\tau^\alpha(\mu)*\sigma_u)\diff u-\frac12H(\mu)+\frac{\alpha\tau}{2s}(H(\mathcal{J}_\tau^\alpha(\mu))-H(\mathcal{J}_\tau^\alpha(\mu)*\sigma_s)\leq K\frac\tau2.
    \end{equation}
    Moreover, the function $s\mapsto H(\mathcal{J}_\tau^\alpha(\mu)*\sigma_s)$ is nonincreasing and lower semicontinuous, hence right-continuous. Therefore, 
    $$
\lim_{s\to0}\frac1{s}\int_0^sH(\mathcal{J}_\tau^\alpha(\mu)*\sigma_u)\diff u=H(\mathcal{J}_\tau^\alpha(\mu)).
    $$
        Finally, for every $\nu\in\mathcal{P}_2(\R^d)$ such that $H(\nu)<+\infty$, we have the following equality in $\R_+\cup\{+\infty\}$:
    \begin{equation*}
    \lim_{s\to 0}\frac{H(\nu)-H(\nu*\sigma_s)}{s}=\frac12\int|\nabla\ln(\nu)|^2\diff\nu,
    \end{equation*}
Consequently, we conclude the proof by sending $s$ to $0$ in equation \eqref{eq:combined eq for upper bound entropyalpha neq0}.
\end{proof}
Sending $\tau$ to $0$ lets us deduce a similar inequality at the continuous level.
   \begin{proposition}\label{bound on Fisher}
Let $\F$ satisfy Hypothesis~\ref{hypothesis}, and let $(\rho_t^\alpha)$ be the solution to the regularized gradient flow~\eqref{eq:gradient_flow_reg} associated with a parameter $\alpha>0$, starting from $\mu_0 \in \mathcal P_2(\R^d)$, a measure such that both $H(\mu_0)< + \infty$ and $\mathcal F(\mu_0) < + \infty$. For all $t \geq 0$, if $\lambda = 0$, then
    \[
    \frac{\alpha}{2}\left(\int_0^t \sqrt{\int |\nabla\ln(\rho^\alpha)|^2\mathrm{d}\rho^\alpha}\mathrm{d}s\right)^2\leq t (H(\mu_0)-H(\rho^\alpha_t)+K t),
    \]
    while if $\lambda\neq0$,
    \[
    \frac{\alpha}{2}\left(\int_0^t e^{\lambda s} \sqrt{\int |\nabla\ln(\rho^\alpha_s)|^2\mathrm{d}\rho^\alpha_s}\mathrm{d}s\right)^2\leq \frac{e^{2\lambda t}-1}{2\lambda} (H(\mu_0)-H(\rho^\alpha_t)+Kt).
    \]    
\end{proposition}
\begin{proof}
    Let $t \geq 0$. By the Cauchy-Schwartz inequality we have:
    \[
    \left(\int_0^t e^{\lambda s} \sqrt{\int |\nabla\ln(\rho^\alpha_s)|^2\mathrm{d}\rho^\alpha_s}\mathrm{d}s\right)^2\leq \int_0^t e^{2\lambda s}\mathrm{d}s\int_0^t \int |\nabla\ln(\rho^\alpha_s)|^2\mathrm{d}\rho^\alpha_s\mathrm{d}s,
    \]
    where
    \[
    \int_0^t e^{2\lambda s}\mathrm{d}s=\left\{\begin{aligned}
        &t, && \text{if $\lambda=0$},\\
        &\frac{e^{2\lambda t}-1}{2\lambda}, &&\text{if $\lambda\neq0$.}
    \end{aligned}\right.
    \]
    Now, using iteratively Proposition~\ref{prop:linear increasing of entropy along JKO alpha neq0} we find:
    \[
    H(\mathcal{J}_{n,t/n}^\alpha(\mu_0))-H(\mu_0)+\frac{\alpha\tau}{2}\sum_{k=1}^n\int|\nabla\ln(\mathcal{J}_{k,t/n}^\alpha(\mu_0))|^2\diff \mathcal{J}_{k,t/n}^\alpha(\mu_0)\leq Kt,
    \]
    or otherwise stated
    \[
    H(\mathcal{J}_{n,t/n}^\alpha(\mu_0))-H(\mu_0)+\frac{\alpha}{2}\int_0^t\int|\nabla\ln(\mathcal{J}_{\lceil\frac{ns}{t}\rceil,t/n}^\alpha(\mu_0))|^2\diff \mathcal{J}_{\lceil\frac{ns}{t}\rceil,t/n}^\alpha(\mu_0)\diff s\leq Kt.
    \]
    But for all $s \geq 0$, $\mathcal{J}_{\lceil\frac{ns}{t}\rceil,t/n}^\alpha(\mu_0)\xrightarrow[n\to+\infty]{W_2}\rho^\alpha(s)$, so the result is a consequence of the lower semicontinuity of the entropy and of the Fisher information.
    \end{proof}
    We will now establish the main inequality of this section by estimating the Wasserstein distance along a geodesics between $\rho^0(t)$ and $\rho^\alpha(t)$ for each $t\geq 0$. We restrict ourserleves to the case where $\mathcal F$ is of the form~\eqref{eq:explicit_F}.
\begin{proposition}\label{app:fanch}
     Let $\mathcal F$ be of the form~\eqref{eq:explicit_F}. We assume that:
    \begin{itemize}
        \item The PDEs~\eqref{eq:gradient_flow} and~\eqref{eq:gradient_flow_reg} admit regular solutions $\rho^0 $ and $\rho^\alpha$ for all times $t\geq 0$,
        \item $\mathcal F$ is $\lambda$-geodesically convex for some $\lambda \in \R$.
    \end{itemize}
    Then, for all $t>0$, if $\lambda=0$, then
        \begin{equation*}           W_2(\rho^0(t),\rho^\alpha(t))\leq\frac{\alpha}{2}\int_0^t \sqrt{\int |\nabla\ln(\rho^\alpha)|^2\mathrm{d}\rho^\alpha}\mathrm{d}s,
        \end{equation*}
    while if $\lambda\neq0$, then
    \begin{equation*}
        W_2(\rho^0(t),\rho^\alpha(t))\leq\frac{\alpha}{2}\int_0^t e^{\lambda (s-t)} \sqrt{\int |\nabla\ln(\rho^\alpha)|^2\mathrm{d}\rho^\alpha}\mathrm{d}s.
    \end{equation*} 
\end{proposition}

 As in the introduction, in the case where $\F$ is of the form~\eqref{eq:explicit_F}, and in a coherent manner with respect to the theory of Wasserstein gradient flows~\cite{santambrogio2015optimal}, we define for all $\rho \in \mathcal P_2(\R^d)$ absolutely continuous with respect to the Lebesgue measure:
    \begin{equation*}
    \frac{\partial \mathcal F}{\partial \rho} (\rho) := V + W \ast \rho + f'(\rho).
\end{equation*}
In that way, as soon as the curve $t \mapsto \rho_t$ is sufficiently regular, we have
 \begin{equation}
 \label{eq:diff_Frho}
    \frac{\diff }{\diff t} \mathcal F(\rho_t) =\int \frac{\delta\F}{\delta\rho}(\rho_t)\partial_t\rho_t.
   \end{equation}
   In fact, the computations in the proof of Proposition~\ref{app:fanch} are justified thanks to the following remark.
\begin{remark}
    Since $\rho^0$ is the flow of $\F$ and $\rho^\alpha$ is the flow of $\F+\frac{\alpha}{2}H$, for all $t \geq 0$, the following integrals are finite (see \cite[Chapter 10]{ambrosio2005gradient}):
    \[
    \int_0^t\int\left|\nabla\frac{\delta\F}{\delta\rho}(\rho^0_s)\right|^2\diff\rho^0_s\diff s<+\infty,\qquad\int_0^t\int\left|\nabla\frac{\delta\F}{\delta\rho}(\rho^\alpha_s)+\frac\alpha2\nabla\ln(\rho^\alpha_s)\right|^2\diff\rho^\alpha_s\diff s<+\infty.
    \]
    Moreover, we have shown in Proposition~\ref{bound on Fisher} that
    $$    \int_0^t\int\left|\nabla\ln(\rho^\alpha_s)\right|^2\diff\rho^\alpha_s\diff s<+\infty.
    $$
    Therefore, by the triangle inequality in $L^2(\diff s \otimes \rho^\alpha_s)$, we also have
    \[
    \int_0^t\int\left|\nabla\frac{\delta\F}{\delta\rho}(\rho^\alpha_s)\right|^2\diff\rho^\alpha_s\diff s<+\infty.
    \]
    Thus, for almost every $s \geq 0$,
    \[
    \int\left|\nabla\frac{\delta\F}{\delta\rho}(\rho^0_s)\right|^2\diff\rho^0_s<+\infty,\qquad\int\left|\nabla\frac{\delta\F}{\delta\rho}(\rho^\alpha_s)\right|^2\diff\rho^\alpha_s<+\infty \quad\text{and}\quad\int\left|\nabla\ln(\rho^\alpha_s)\right|^2\diff\rho^\alpha_s<+\infty.
    \]
\end{remark}
In the proof of Proposition~\ref{app:fanch}, we will need the following Lemma.
\begin{lemma}\label{F cnx ine along geo}
    Under the assumption of Proposition~\ref{app:fanch}. For all $\mu,\nu\in\mathcal{P}_2(\R^d)$ such that $\F(\mu)<+\infty$ and $\F(\nu)<+\infty$, denoting by $\varphi$ the Kantorovich potential from $\mu$ to $\nu$ and $\psi$ the Kantorovich potential from $\nu$ to $\mu$ then
    \[
    \int \nabla\frac{\delta \F}{\delta\rho}(\mu)\nabla\varphi\mu+\int \nabla\frac{\delta \F}{\delta\rho}(\nu)\nabla\psi\nu \geq \lambda W^2_2(\mu,\nu).
    \]
    \end{lemma}
    \begin{proof}
        Let us consider $(\rho_s,v_s)$ the Wasserstein geodesics from $\mu$ to $\nu$. Then the map
        \(        f:[0,1]\in s\rightarrow\F(\rho_s)
        \) is $\lambda W_2^2(\mu,\nu)$ convex, in particular 
        $$
        f'(0)\leq f'(1)-\lambda W_2^2(\mu,\nu).
        $$
        Bur for all $s\in[0,1]$, in view of~\eqref{eq:diff_Frho},
        \[
        f'(s)=\frac{\diff}{\diff s}\F(\rho_s)=\int\frac{\delta\F}{\delta\rho}(\rho_s)\partial_s\rho_s.
        \]
        As a consequence of Remark~\ref{rem:derivative at time 0 and 1 along geodesics}, we find
        \[
    f'(0) = -\int \nabla\frac{\delta \F}{\delta\rho}(\mu)\nabla\varphi\mu \qquad \mbox{and}\qquad f'(1) = \int \nabla\frac{\delta \F}{\delta\rho}(\nu)\nabla\psi\nu,
        \]
        and hence
        \[
        -\int \nabla\frac{\delta \F}{\delta\rho}(\mu)\nabla\varphi\mu\leq \int \nabla\frac{\delta \F}{\delta\rho}(\nu)\nabla\psi\nu - \lambda W^2_2(\mu,\nu),
        \]
        as announced. 
    \end{proof}
    We are now ready to prove Proposition~\ref{app:fanch}.
\begin{proof}[Proof of Proposition~\ref{app:fanch}]

To prove Proposition~\ref{app:fanch}, the idea is to differentiate the dual formulation of the squared Wasserstein distance between our solutions. By Subsection~\ref{subsec:wasserstein dual formulation}, we have for all $t \geq 0$
\[
\frac{W_2^2(\rho^0_t,\rho^\alpha_t)}{2} = \max_{\substack{\varphi, \psi \\ \varphi \oplus \psi \leq \frac{|x - y|^2}{2}}} \int \varphi\, \mathrm{d}\rho^0_t + \int \psi\, \mathrm{d}\rho^\alpha_t.
\]
Let us call $(\varphi_t,\psi_t)$ a pair of maximizers. Applying the envelope theorem, we have for all $t \geq 0$
\[
\frac{\diff }{\diff t}\left(\frac{W_2^2(\rho^0_t,\rho^\alpha_t)}{2}\right)
= \int \varphi_t\, \partial_t\rho^0_t + \int \psi_t\, \partial_t\rho^\alpha_t.
\]
Using the PDEs satisfied by $\rho^0_t$ and $\rho^\alpha_t$, we obtain:
\begin{align*}
\frac{\diff }{\diff t}\left(\frac{W_2^2(\rho^0_t,\rho^\alpha_t)}{2}\right)
&= \int \varphi_t\, \mathrm{div}\left(\rho^0_t \nabla \frac{\delta \mathcal{F}}{\delta \rho}(\rho^0_t)\right) 
+ \int \psi_t\, \mathrm{div}\left(\rho^\alpha_t \nabla \frac{\delta \mathcal{F}}{\delta \rho}(\rho^\alpha_t)\right) 
+ \frac{\alpha}{2} \int \psi_t\, \Delta \rho^\alpha_t \\
&= - \int \nabla \varphi_t \cdot \nabla \frac{\delta \mathcal{F}}{\delta \rho}(\rho^0_t)\, \rho^0_t 
- \int \nabla \psi_t \cdot \nabla \frac{\delta \mathcal{F}}{\delta \rho}(\rho^\alpha_t)\, \rho^\alpha_t 
- \frac{\alpha}{2} \int \nabla \psi_t \cdot \nabla \rho^\alpha_t.
\end{align*}
Then, by applying Lemma~\ref{F cnx ine along geo}, we get the bound:
\begin{equation}\label{eq:derivative wasserstein leq fisher*wasserstein}
\frac{\diff }{\diff t}\left(\frac{W_2^2(\rho^0_t,\rho^\alpha_t)}{2}\right) \leq -\lambda W_2^2(\rho^0_t,\rho^\alpha_t) - \frac{\alpha}{2} \int \nabla \psi_t \cdot \nabla \rho^\alpha_t.
\end{equation}
Using the identity $\nabla \rho^\alpha_t = \rho^\alpha_t \nabla \ln \rho^\alpha_t$ and the Cauchy--Schwarz inequality in $L^2(\rho^\alpha_t)$, we get:
\[
- \frac{\alpha}{2} \int \nabla \psi_t \cdot \nabla \rho^\alpha_t 
= - \frac{\alpha}{2} \int \nabla \psi_t \cdot \nabla \ln(\rho^\alpha_t)\, \mathrm{d}\rho^\alpha_t 
\leq \frac{\alpha}{2} W_2(\rho^0_t,\rho^\alpha_t) \sqrt{\int |\nabla \ln(\rho^\alpha_t)|^2\, \mathrm{d}\rho^\alpha_t}.
\]
After dividing both sides by $W_2(\rho^0_t,\rho^\alpha_t)$ in~\eqref{eq:derivative wasserstein leq fisher*wasserstein}, we obtain:
\[
\frac{\diff }{\diff t} W_2(\rho^0_t,\rho^\alpha_t) \leq -\lambda W_2(\rho^0_t,\rho^\alpha_t) 
+ \frac{\alpha}{{2}} \sqrt{\int |\nabla \ln(\rho^\alpha_t)|^2\, \mathrm{d}\rho^\alpha_t}.
\]
Therefore, the results follows from applying Grönwall's lemma.
\end{proof}

\section{A truncation argument for the small values of $\F$}
In Section~\ref{sec:proof_cor}, we used the fact that replacing $\F$ by $\max\{\F,M\}$ for a sufficiently small value of $M$ does not affect the first iterates of the JKO scheme, see Remark~\ref{rem:below_bound_F}. This is the content of the following proposition.

\begin{proposition}\label{prop:lower bound along JKO implies wlog F lower bounded}
    Let $\F:\mathcal P_2(\R^d)\to\R$, $\tau>0$ and $\mu\in\mathcal P_2(\R^d)$. Let us assume that the functional
    \begin{equation*}
        \frac{W_2^2(\mu,\rho)}{2\tau}+\F(\rho)
    \end{equation*}
    admits at least one minimizer, and that there exists $M \in \R$ such that each of these minimizers $\nu$ satisfy $\F(\nu)\geq M$. Then, calling $\F^M: \rho \mapsto \max\{M,\F(\rho)\}$, we have    
    $$
    \argmin_{\rho\in\mathcal P_2(\R^d)}\left\{\frac{W_2^2(\mu,\rho)}{2\tau}+\F(\rho)\right\}=\argmin_{\rho\in\mathcal P_2(\R^d)}\left\{\frac{W_2^2(\mu,\rho)}{2\tau}+\F^M(\rho)\right\}.
    $$
\end{proposition}
\begin{proof}
First, as $\F \leq \F^M$, we have
\begin{equation}
\label{eq:ineg_min}
    \min_{\rho\in\mathcal P_2(\R^d)}\left\{\frac{W_2^2(\mu,\rho)}{2\tau}+\F(\rho)\right\} \leq \inf_{\rho\in\mathcal P_2(\R^d)}\left\{\frac{W_2^2(\mu,\rho)}{2\tau}+\F^M(\rho)\right\}.
\end{equation}
But by assumption, if $\nu$ is any minimizer in the left-hand side, $\F^M(\nu) = \F(\nu)$, so that
    \[
    \frac{W_2^2(\mu,\nu)}{2\tau}+\F(\nu)=\frac{W_2^2(\mu,\nu)}{2\tau}+\F^M(\nu)\geq \inf_{\rho\in\mathcal P_2(\R^d)}\left\{\frac{W_2^2(\mu,\rho)}{2\tau}+\F^M(\rho)\right\}.
    \]
So the inequality in~\eqref{eq:ineg_min} is in fact an equality, and the minimizers in the left-hand side are also minimizers in the right-hand side. As a consequence, if $\nu$ is any minimizer in the right-hand side of~\eqref{eq:ineg_min}, we have
 \[
    \frac{W_2^2(\mu,\nu)}{2\tau}+\F^M(\nu)=\min_{\rho \in\mathcal P_2(\R^d)}\left\{\frac{W_2^2(\mu,\rho)}{2\tau}+\F(\rho)\right\}\leq \frac{W_2^2(\mu,\nu)}{2\tau}+\F(\nu)\leq \frac{W_2^2(\mu,\nu)}{2\tau}+\F^M(\nu).
    \]
    Therefore, all the inequalities are in fact equalities, and $\nu$ is also a minimizer in the left-hand side of~\eqref{eq:ineg_min}, which concludes the proof.
\end{proof}
 \bibliographystyle{plain}
\bibliography{refs}
\end{document}